\documentclass[a4paper,reqno]{amsart}
\usepackage[a4paper,margin=1in]{geometry}
\usepackage[T1]{fontenc}

\usepackage{microtype}
\usepackage{mathtools}
\usepackage{amsmath}
\usepackage{amssymb}
\usepackage{amsthm}
\usepackage[dvipsnames]{xcolor}
\usepackage{hyperref}
\hypersetup{
    colorlinks,
    linkcolor={MidnightBlue},
    citecolor={Green},
    urlcolor={Maroon}
}
\usepackage{tikz-cd}
\usetikzlibrary{arrows}
\usepackage{ytableau}
\usepackage{relsize} 
\usepackage{savesym}  
\usepackage{enumitem}
\usepackage{comment}

\usepackage{aliascnt}
\usepackage[nameinlink,capitalise,noabbrev]{cleveref}

\newtheoremstyle{plainnopunct}
  {}{}{\itshape}{}{\bfseries}{}{.5em}{}

\theoremstyle{plainnopunct}

\newtheorem{theorem}{Theorem}[section]
\newaliascnt{proposition}{theorem}
\newtheorem{proposition}[proposition]{Proposition}
\aliascntresetthe{proposition}
\newaliascnt{lemma}{theorem}
\newtheorem{lemma}[lemma]{Lemma}
\aliascntresetthe{lemma}
\newaliascnt{corollary}{theorem}
\newtheorem{corollary}[corollary]{Corollary}
\aliascntresetthe{corollary}

\newtheoremstyle{definitionnopunct}
  {}{}{\normalfont}{}{\bfseries}{}{.5em}{}

\theoremstyle{definitionnopunct}

\newaliascnt{definition}{theorem}
\newtheorem{definition}[definition]{Definition}
\aliascntresetthe{definition}
\newaliascnt{example}{theorem}
\newtheorem{example}[example]{Example}
\aliascntresetthe{example}
\newaliascnt{remark}{theorem}
\newtheorem{remark}[remark]{Remark}
\aliascntresetthe{remark}

\savesymbol{wedge}
\restoresymbol{AMS}{wedge}
\newcommand{\superwedge}{\mathlarger{\mathlarger{\wedge}}}

\newcommand{\C}{\mathbb{C}}
\renewcommand{\P}{\mathbb{P}}

\newcommand{\N}{\mathbb{N}}
\renewcommand{\S}{\mathbb{S}}
\newcommand{\GL}{\mathrm{GL}}
\newcommand{\SL}{\mathrm{SL}}
\newcommand{\Lie}{\mathfrak}
\newcommand{\cO}{\mathcal{O}}
\newcommand{\cM}{\mathcal{M}}
\newcommand{\cR}{\mathcal{R}}
\newcommand{\cI}{\mathcal{I}}

\newcommand{\rank}{\operatorname{rank}}
\newcommand{\codim}{\operatorname{codim}}

\newcommand{\Sym}{\operatorname{Sym}}
\newcommand{\ind}{\operatorname{Ind}}
\newcommand{\dist}{\operatorname{dist}}

\newcommand{\be}{\mathbf{e}}
\newcommand{\bx}{\mathbf{x}}
\newcommand{\bv}{\mathbf{v}}
\newcommand{\bd}{\mathbf{d}}

\newcommand{\SchurDouble}[1]{I_{\mathrm S}(#1)^{\langle 2\rangle}}

\newcommand{\bepsilon}{\boldsymbol{\varepsilon}}

\DeclareMathOperator{\expdim}{expdim}
\DeclareMathOperator{\cont}{cont}
\DeclareMathOperator{\Gr}{Gr}
\DeclareMathOperator{\Fl}{Fl}

\usepackage[textwidth=0.8in]{todonotes}
\title{Secant varieties of flag varieties via Schur apolarity}

\author{Alessandra Bernardi}
\address[Alessandra Bernardi]{Università di Trento\\
Via Sommarive, 14\\38123 Povo (Trento), 
Italy}
\email{alessandra.bernardi@unitn.it}

\author{Stefano Canino}
\address[Stefano Canino]{Università di Trento\\
Via Sommarive, 14\\38123 Povo (Trento), 
Italy}
\email{stefano.canino@unitn.it}

\author{Vincenzo Antonio Isoldi}
\address[Vincenzo Antonio Isoldi]{Max Planck Institute for Mathematics in the Sciences\\
Inselstr. 22\\04103 Leipzig, 
Germany }
\email{isoldi@mis.mpg.de}

\makeatletter
\let\original@setaddresses\@setaddresses
\renewcommand{\@setaddresses}{%
  \begingroup
  \setlength{\parindent}{0pt}%
  \original@setaddresses
  \endgroup
}
\makeatother

\begin{document}

\begin{abstract}
We develop a general first-order theory of Schur apolarity for the
study of secant varieties of flag varieties in arbitrary homogeneous
embeddings. Extending the classical apolarity--fat-point
correspondence for Veronese varieties, we show that in the Schur
setting the algebraic square of the apolar ideal need not coincide
with the geometric double-point conditions. We introduce a geometric
Schur square whose
relevant component is the conormal space, yielding a Schur Dual
Terracini Lemma. Our construction recovers classical apolarity in the
symmetric case. A slot-by-slot Consistency Theorem realizes these
intrinsic conditions as multigraded double points. As an application, we determine the dimensions of all secant varieties of $\Fl(1,2;V_n)$ embedded by
$\cO(1,1)$: the only defective cases are
$\sigma_2(\Fl(1,2;V_3))$ and $\sigma_3(\Fl(1,2;V_4))$, both of defect
one.
\end{abstract}

\maketitle

\section{Introduction}

\subsection{Motivation and guiding questions}

Secant varieties lie at the intersection of projective geometry,
representation theory, and tensor decomposition. Let
$X\subseteq\P(W)$
be an irreducible nondegenerate projective variety. Its $s$-th secant
variety $\sigma_s(X)$ is the Zariski closure of the union of the
projective linear spaces spanned by $s$ points of $X$. When $X$
parametrizes tensors of a prescribed decomposable type, membership in
$\sigma_s(X)$ expresses border rank at most $s$ with respect to the
corresponding notion of decomposability. Determining the dimensions
of secant varieties, and in particular identifying the defective
ones, is therefore a central problem in the geometry of tensors; see
\cite{bernardi_hitchhiker_2018}.

Apolarity has played a fundamental role in this problem. In the
symmetric setting, classical apolarity translates questions about
Veronese varieties and Waring decompositions into questions about
ideals, catalecticant maps, and zero-dimensional schemes; see
\cite{IarrobinoKanev1999,bernardi_hitchhiker_2018}. Its interaction
with Terracini's Lemma turns the computation of secant dimensions into
an interpolation problem for fat points. This point of view underlies
the classification of the defective secant varieties of Veronese
varieties provided by the Alexander--Hirschowitz theorem
\cite{AlexanderHirschowitz1995,BrambillaOttaviani2008}. The
skew-symmetric apolarity, introduced in
\cite{ArrondoBernardiMarquesMourrain2021} gives an analogous
apolarity language for skew-symmetric tensors and Grassmannians.

Schur apolarity, introduced by Staffolani in
\cite{StaffolaniThesis,Staffolani2023}, places the symmetric and
skew-symmetric constructions in a common representation-theoretic
framework. Indeed, the Schur apolarity action specializes to
classical symmetric apolarity for one-row partitions and to
skew-symmetric apolarity for one-column partitions. More generally,
it is defined on arbitrary Schur modules, whose decomposable tensors
are parametrized by closed highest-weight orbits, and hence by partial
flag varieties in homogeneous embeddings.

This makes Schur apolarity a natural candidate for extending the
apolarity-based study of secant varieties from Veronese varieties and
Grassmannians to arbitrary flag varieties. The Schur Apolarity Lemma
of \cite{Staffolani2023} already gives a uniform criterion for
membership in the span of points of a highest-weight orbit. It is,
however, a zero-order statement: it detects linear spans of points,
but does not by itself describe tangent spaces or the first-order
conditions required by Terracini's Lemma.

The purpose of this paper is to develop precisely this missing
first-order theory. Our guiding principle is that Schur apolarity
should provide not only an apolarity lemma for decompositions, but
also a geometric and computational language for the dimensions of
secant varieties of flag varieties.

The classical model is the Veronese variety. Let
\[
p=[v]\in\P(V),
\qquad
P=\nu_d(p)=[v^d]\in\nu_d(\P(V)).
\]
The classical statement usually attributed to Lasker
\cite{Lasker1904} gives
\[
\left(
\widehat T_P\nu_d(\P(V))
\right)^\perp
=
\left(
I_{p,\P(V)}^2
\right)_d;
\]
see also
\cite[Proposition~2.1]{BrambillaOttaviani2008} and
\cite{Ottaviani2013}. Thus, the conormal space is the degree-$d$
component of the ideal of the geometric double point supported at
$p$. Combined with the dual form of Terracini's Lemma, this
identification transforms the computation of secant dimensions into
the postulation of unions of double points.

For a point
\[
P\in F_{\lambda,n}
\subseteq
\P(\S_\lambda V),
\]
where $F_{\lambda,n}$ denotes the closed highest-weight orbit
associated with $\lambda$, the most direct Schur-theoretic analogue would be to take the $\lambda$-component of the algebraic square of the Schur apolar ideal.
A first observation of this paper is that this construction fails in
general. Multiplication in the dual Schur algebra is defined
representation-theoretically and does not necessarily encode
first-order vanishing along the corresponding highest-weight orbit.
Consequently, a genuine extension of the classical double-point
method requires a new first-order construction.

This raises two related questions. What is the correct Schur-apolar
analogue of the ideal of a geometric double point? Once this object
has been identified intrinsically, can it be realized in a form
suitable for interpolation, specialization, and Horace-type
arguments?

Our first main contribution is a first-order extension of Schur
apolarity. We show that the direct analogue of Lasker's Lemma obtained
by taking the algebraic square of the Schur apolar ideal fails in
general; see \Cref{prop: algebraic Schur square fails}. We then
introduce the \emph{geometric Schur square} $\SchurDouble{P}$ and
prove the Tangential Schur Apolarity Theorem
\[
\left(
\SchurDouble{P}
\right)_\lambda
=
\left(
\widehat T_PF_{\lambda,n}
\right)^\perp;
\]
see \Cref{thm: tangential Schur apolarity}. Together with
Terracini's Lemma, this gives the Schur Dual Terracini Lemma and
provides an intrinsic Schur-apolar description of the conormal spaces
governing the dimensions of secant varieties of flag varieties.

Our second main contribution is a computational realization of these
intrinsic conditions. We work in the multigraded slot-by-slot algebra
associated with the Pl\"ucker factors of the flag variety. The Consistency Theorem, \Cref{thm: consistency}, gives a necessary
and sufficient combinatorial condition on the Young diagram of
$\lambda$ for ordinary multihomogeneous double points in this algebra
to recover the relevant $\lambda$-component of the geometric Schur
square or, equivalently, the conormal space. Thus, for every consistent
embedding, the secant-dimension problem can be translated into a
multigraded interpolation problem without losing its Schur-apolar
meaning.

The principal application of this framework is the complete
classification of the secant varieties of
\[
F_{(2,1),n}
\simeq
\Fl(1,2;V_n)
\subseteq
\P(\S_{(2,1)}V_n)
\]
in its $\cO(1,1)$-embedding. We prove that, for every $n\geq3$ and
every $s\geq1$,
\[
\sigma_s(F_{(2,1),n})
\text{ is defective}
\quad\Longleftrightarrow\quad
(n,s)\in\{(3,2),(4,3)\},
\]
and both exceptional secant varieties have defect one; see
\Cref{thm: classification two step flags}.

The proof is based on a Horace-type induction carried out inside the
slot-by-slot realization of the Schur-apolar conormal conditions.
Thus, the trace and residual constructions are not applied after
discarding the Schur structure: the relevant linear systems arise
directly from the Schur Dual Terracini Lemma and the Consistency Theorem. The
classification above should therefore be regarded as the first
full-scale application of the framework developed here, rather than
as its endpoint.

\subsection{Previous work}

The secant geometry of homogeneous varieties has been studied from
several perspectives. Baur and Draisma investigated higher secant
varieties of the minimal adjoint orbit and secant dimensions of
low-dimensional homogeneous varieties, obtaining classifications and
explicit defective examples in small dimension
\cite{Baur2004,BD10}.

Barbosa Freire, Casarotti, and Massarenti established
nondefectivity and identifiability results for partial flag varieties
by degenerating tangent spaces to higher osculating spaces
\cite{BFCM22}. Blomenhofer and Casarotti subsequently obtained
general nondefectivity bounds for invariant varieties contained in
projectivizations of irreducible representations; their results
apply, in particular, to closed highest-weight orbits
\cite{BC23}. These works provide low-dimensional classifications and broad ranges
of nondefectivity. They do not, however, settle the critical secant
orders near the generic rank that are needed for a complete
classification of
\[
F_{(2,1),n}
\simeq
\Fl(1,2;V_n)
\subseteq
\P(\S_{(2,1)}V_n)
\]
in its $\cO(1,1)$-embedding.

The present paper approaches the problem from a different direction.
Rather than beginning with a particular family of flag varieties, we
develop the Schur-apolar machinery needed to study secant dimensions
uniformly. To the best of our knowledge, Schur apolarity had not
previously been developed into a general first-order method for the
secant geometry of flag varieties.

Our contribution complements the existing nondefectivity results in
two ways. First, it extends Schur apolarity from the study of linear
spans of decomposable Schur tensors to tangent and conormal spaces.
Second, it connects this intrinsic construction with ordinary
multigraded double points, making interpolation and specialization
methods available in the Schur setting. The complete classification
of $F_{(2,1),n}$ addresses, in particular, the critical secant orders
near the generic rank that are not controlled by the available
general bounds.

\subsection{Organization of the paper}

In \Cref{sect: preliminary}, we recall the required background on
partitions, tableaux, Schur modules, homogeneous embeddings of partial
flag varieties, secant varieties, and Terracini's Lemma. We also fix
the Cartan realization and the tableau bases used throughout the
paper.

In \Cref{sect: tangential Schur apolarity}, we recall the dual Schur
algebra and the Schur apolarity action, emphasizing their
specializations to symmetric and skew-symmetric apolarity. We then
describe tangent and conormal spaces, prove the failure of the
algebraic Schur square, introduce the geometric Schur square, and
establish Tangential Schur Apolarity and the Schur Dual Terracini
Lemma.

In \Cref{sect: slotwise double points}, we introduce the multigraded
slot-by-slot algebra and compare ordinary geometric double points in
the product of Pl\"ucker spaces with the intrinsic geometric Schur
square. This leads to the Consistency Theorem.

In \Cref{sect: defectiveness}, we use the slot-by-slot realization to
construct and study defective secant varieties of low-dimensional
flag varieties.

Finally, \Cref{sec: secants two step flags} is devoted to $F_{(2,1),n}
\simeq
\Fl(1,2;V_n)$.
We translate the secant-dimension problem into a multigraded
double-point interpolation problem, develop the inductive
specialization and restriction arguments, and complete the
classification of all secant varieties of this family.

\section*{Acknowledgements}
\noindent We thank J. Buczyński and J. Jelisiejew for fruitful discussions during the 2025 AIVAR workshop (IMPAN in Warsaw, Poland). 

Part of this work was developed in VAI's Master's thesis \cite{IsoldiThesis} and while AB was visiting the Simons Institute for the Theory of Computing, during the program on Complexity and Linear Algebra, in Fall 2025.

AB and SC are members of GNSAGA (INdAM). 

AB has been partially funded by the European Union under NextGeneration EU. PRIN 2022 Prot. n. 2022ZRRL4C 004. Views and opinions expressed are however those of the authors only and do not necessarily reflect those of the European Union or European Commission. Neither the European Union nor the granting authority can be held responsible for them.

SC has been funded by the Italian Ministry of University and Research in the framework of the Call for Proposals for scrolling of final rankings of the PRIN 2022 call - Protocol no. 2022NBN7TL and by the project \textit{Thematic Research Programmes}, Action I.1.5 of the program \textit{Excellence Initiative -- Research University} (IDUB) of the Polish Ministry of Science and Higher Education.

VAI gratefully acknowledges the Max Planck Society, in particular MPI-CBG and MPI MiS, for their hospitality and support during his stay from January to June 2026, when this work was completed.

The authors acknowledge the TensorDec Laboratory of the Department of Mathematics of the University of Trento, of which AB and SC are currently members and VAI is a former member, for helpful discussions.

\section{Schur modules, flag varieties, and secant varieties}
\label{sect: preliminary}

\noindent Throughout, $V$ is a complex vector space of dimension $n$, with fixed basis $\{e_1,\dots,e_n\}$ and dual basis $\{x_1,\dots,x_n\}$. We write $G\coloneqq\GL(V)$. 

\subsection{Young tableaux, Schur modules, and flag varieties}
\label{subsect: tableaux}

We use the standard conventions for partitions, Young diagrams, tableaux, and Schur modules. We recall only the notation needed in the sequel and refer to~\cite{Fulton1997,FultonHarris1991} for the classical background.

A partition of a positive integer $d$ is a weakly decreasing finite sequence $\lambda=(\lambda_1,\ldots,\lambda_{\ell(\lambda)})$
of positive integers such that $|\lambda|
\coloneq
\sum_{i=1}^{\ell(\lambda)}\lambda_i
=
d$.
We write $\lambda\vdash d$. Its conjugate partition is denoted by $\lambda'$. For $m\geq1$, we set
\[
[m]\coloneq\{1,\ldots,m\},
\qquad
[0]\coloneq\emptyset.
\]
For two partitions $\mu$ and $\lambda$, we write $\mu\subseteq\lambda$ if $\mu_i\leq\lambda_i$ for every $i$, and we denote by $\lambda/\mu$ the corresponding skew diagram. A tableau of shape $\lambda$ is called \emph{semistandard} if its entries are strictly increasing down each column and weakly increasing along each row. Unless otherwise stated, all tableaux considered below have entries in $[n]$.

\begin{remark}\label{rem: determinant twist}
Suppose that $\lambda'$ contains $q$ columns of height $n$, and let $\overline{\lambda}$ be the partition obtained by deleting these columns. Then
\[
\S_\lambda V
\simeq
(\det V)^{\otimes q}\otimes\S_{\overline{\lambda}}V.
\]
Since the determinant factor is one-dimensional, it does not change the associated closed orbit in projective space. Therefore, when studying the projective geometry of highest-weight orbits, we may and will assume that $\ell(\lambda)<n$; see \cite{FultonHarris1991}.
\end{remark}

Accordingly, throughout the paper we write the conjugate partition in block form as
\[\lambda'
=
(n_k^{d_k},\ldots,n_1^{d_1}),
\qquad
0<n_1<\cdots<n_k<n,
\qquad
d_i\geq1,
\]meaning that $\lambda'$ has exactly $d_i$ columns of height $n_i$. We also set
\[
\bd
\coloneq
(d_1,\ldots,d_k),
\qquad
n_0\coloneq0,
\qquad
n_{k+1}\coloneq n.
\]
For an increasing set 
$J=\{j_1<\cdots<j_r\}\subseteq[n]$,
we use the notation
\[
e_J
\coloneq
e_{j_1}\wedge\cdots\wedge e_{j_r},
\qquad
x_J
\coloneq
x_{j_1}\wedge\cdots\wedge x_{j_r}.
\]

\begin{definition}
\label{def: tableau data}
Let $R$ be a tableau of shape $\lambda$ with entries in $[n]$ and strictly increasing columns, with no row condition unless explicitly stated. If the entries in its $c$-th column form the increasing set $J_R^{(c)}\subseteq[n]$, we define its \emph{column index sequence} by
\[
\ind(R)
\coloneq
\bigl(
J_R^{(1)},\ldots,J_R^{(\lambda_1)}
\bigr).
\]
The \emph{content} of $R$ is the vector
\[
\cont(R)
\coloneq
(m_1,\ldots,m_n)\in\N^n,
\qquad
m_a
\coloneq
\#\bigl\{
c\mid a\in J_R^{(c)}
\bigr\}.
\]
We denote by $S_0$ the \emph{highest-weight tableau} of shape $\lambda$, characterized by
\[
J_{S_0}^{(c)}
=
[\lambda'_c]
\qquad
\text{for every }
c=1,\ldots,\lambda_1.
\]
Accordingly, we call $[\lambda'_c]$ the highest-weight column of height $\lambda'_c$.
For two tableaux $R_1$ and $R_2$ of shape $\lambda$ with strictly increasing columns, their \emph{column distance} is
\[
\dist(R_1,R_2)
\coloneq
\sum_{c=1}^{\lambda_1}
\left|
J_{R_1}^{(c)}
\setminus
J_{R_2}^{(c)}
\right|.
\]
Since corresponding columns have the same cardinality, one has
\[
\left|
J_{R_1}^{(c)}
\setminus
J_{R_2}^{(c)}
\right|
=
\left|
J_{R_2}^{(c)}
\setminus
J_{R_1}^{(c)}
\right|,
\]
and hence $\dist(R_1,R_2)=\dist(R_2,R_1)$.
An entry in the $c$-th column of $R$ is called \emph{non-highest-weight} if it belongs to $J_R^{(c)}\setminus[\lambda'_c]$.
The $c$-th column is called non-highest-weight if $J_R^{(c)}
\neq
[\lambda'_c]$.
\end{definition}

For a partition $\lambda$ with $\ell(\lambda)\leq n$, we denote by $\S_\lambda V$ the irreducible polynomial $G$-module of highest weight $\lambda$, and by $\S_\lambda V^*$ its dual. We use the usual column realization of Schur modules; see \cite{Fulton1997,FultonHarris1991}.

\begin{definition}[Embedded flag variety]
\label{def: embedded flag variety}
Let $\bv_{\lambda,n}\in\S_\lambda V$ be a nonzero highest-weight vector and set $P_{\lambda,n}
\coloneq
[\bv_{\lambda,n}]
\in
\P(\S_\lambda V)$.
We denote by $F_{\lambda,n}
\coloneq
G\cdot P_{\lambda,n}
\subseteq
\P(\S_\lambda V)$
the embedded flag variety associated with $\lambda$ and $n$. We also set $N_{\lambda,n}
\coloneq
\dim\P(\S_\lambda V)
=
\dim\S_\lambda V-1$.
\end{definition}

\begin{remark}
\label{rem: geometric description embedded flag}
By the highest-weight orbit construction, $F_{\lambda,n}$ is the unique closed $G$-orbit in $\P(\S_\lambda V)$ and is naturally isomorphic to the partial flag variety
\[
\Fl(n_1,\ldots,n_k;V)
\coloneq
\left\{
W_1\subsetneq\cdots\subsetneq W_k\subsetneq V
\ \middle|\
\dim W_i=n_i
\right\}.
\]
Under the identification above, the embedding $F_{\lambda,n}
\subseteq
\P(\S_\lambda V)$ is induced by the very ample line bundle $\cO_{\Fl(n_1,\ldots,n_k;V)}
(d_1,\ldots,d_k)$. These are standard consequences of highest-weight theory and of the Pl\"ucker description of partial flag varieties; see \cite{Fulton1997,FultonHarris1991}. Let $W_i^0
\coloneq
\langle e_1,\ldots,e_{n_i}\rangle$.
The standard flag
$W_1^0
\subsetneq
\cdots
\subsetneq
W_k^0$
corresponds to the point $P_{\lambda,n}$. In the usual column realization, one has, up to a nonzero scalar,
\[
\bv_{\lambda,n}
\sim
e_{[n_k]}^{\otimes d_k}
\otimes\cdots\otimes
e_{[n_1]}^{\otimes d_1}.
\]
In particular, $\bv_{\lambda,n}$ corresponds to the highest-weight tableau $S_0$ of \Cref{def: tableau data}. Finally,
\[\dim F_{\lambda,n}
=
\sum_{i=1}^k
(n_i-n_{i-1})(n-n_i).
\]\end{remark}

\subsection{Cartan realization and tableau bases}
\label{subsect: cartan realization}
The following classical construction realizes $\S_\lambda V$ as the Cartan component of a tensor product of symmetric powers of exterior powers. We refer to \cite{Fulton1997,FultonHarris1991} for the representation-theoretic background.

\begin{proposition}
\label{prop: cartan realization}
The irreducible $G$-module $\S_\lambda V$ occurs with multiplicity $1$ as the Cartan component of $\bigotimes_{i=1}^k \Sym^{d_i}\!\left(\superwedge^{n_i}V\right)$. We fix a nonzero $G$-equivariant inclusion and its dual surjection
\[\iota_\lambda:
\S_\lambda V
\lhook\joinrel\longrightarrow
\bigotimes_{i=1}^k
\Sym^{d_i}\!\left(\superwedge^{n_i}V\right),
\qquad
\pi_\lambda:
\bigotimes_{i=1}^k
\Sym^{d_i}\!\left(\superwedge^{n_i}V^*\right)
\longrightarrow
\S_\lambda V^*.
\]
Under this realization, the embedding $F_{\lambda,n} \subseteq \P(\S_\lambda V)$ is obtained by restricting the product of the Pl\"ucker embeddings, followed by the Segre--Veronese embedding of multidegree $\bd$, to the incidence locus defining the partial flag variety. Its linear span is the projectivization of the Cartan component $\iota_\lambda(\S_\lambda V)$.
\end{proposition}

\begin{proof}
The highest weight of $\bigotimes_{i=1}^k \Sym^{d_i}\!\left(\superwedge^{n_i}V\right)$ is the sum of the weights of its highest-weight factors, namely the weight corresponding to the partition $\lambda$. The associated irreducible summand occurs with multiplicity one and is the Cartan component. The geometric statement is the standard Pl\"ucker--Segre--Veronese realization of the homogeneous embedding defined by $\cO(d_1,\ldots,d_k)$; see \cite{Fulton1997,FultonHarris1991}.
\end{proof}

\begin{definition}
\label{def: tableau monomials}
Let $S$ be a tableau of shape $\lambda$ with strictly increasing columns, and write $\ind(S)
=
\bigl(
J_S^{(1)},\ldots,J_S^{(\lambda_1)}
\bigr)$ as in \Cref{def: tableau data}. We associate with $S$ the element
\[\widetilde{x}_S
\coloneq
\bigotimes_{i=1}^k
\left(
\prod_{\substack{1\leq c\leq\lambda_1\\
                  \lambda'_c=n_i}}
x_{J_S^{(c)}}
\right)
\in
\bigotimes_{i=1}^k
\Sym^{d_i}\!\left(\superwedge^{n_i}V^*\right).\]
The product in each tensor factor is well defined because the columns of a fixed height are grouped inside the corresponding symmetric power.
\end{definition}

For a semistandard tableau $S$ of shape $\lambda$ with entries in $[n]$, we set
\[\bx_S
\coloneq
\pi_\lambda\!\left(\widetilde{x}_S\right)
\in
\S_\lambda V^*.\]

\begin{proposition}
\label{prop: semistandard tableau basis}
The set $\left\{ \bx_S \ \middle|\ S \text{ is a semistandard tableau of shape $\lambda$ with entries in $[n]$} \right\}$ is a basis of $\S_\lambda V^*$.
\end{proposition}

\begin{proof}
This is the classical semistandard tableau basis of the Schur module, expressed through the Cartan projection fixed in \Cref{prop: cartan realization}.
\end{proof}

We denote by $\left\{ \be_S \ \middle|\ S \text{ is a semistandard tableau of shape $\lambda$ with entries in $[n]$} \right\}$ the corresponding dual basis of $\S_\lambda V$. Thus, with respect to the natural evaluation pairing
\[
\langle\ ,\ \rangle_\lambda:
\S_\lambda V^*\otimes\S_\lambda V
\longrightarrow
\C,
\]
one has
\begin{equation}
\label{eq: dual tableau bases}
\langle\bx_S,\be_T\rangle_\lambda
=
\delta_{S,T}.
\end{equation}

\begin{remark}
\label{rem: highest weight basis vector}
We normalize the Cartan maps and the highest-weight vector so that $\bv_{\lambda,n}
=
\be_{S_0}$. 
Under the inclusion $\iota_\lambda$, this vector is represented, up to the chosen normalization, as described in \Cref{rem: geometric description embedded flag}, by 
$e_{[n_k]}^{\otimes d_k} \otimes\cdots\otimes e_{[n_1]}^{\otimes d_1}$.
For a general semistandard tableau $S$, however, $\be_S$ denotes the element of the dual tableau basis and should not be identified with the simple tensor obtained by tensoring the exterior products associated with the columns of $S$. Its expression in the Cartan component may be a linear combination of such simple tensors.
\end{remark}

\subsection{Secant varieties and Terracini's Lemma}
\label{subsect: secant and terracini}

We recall the standard notions concerning secant varieties and
Terracini's Lemma, referring to \cite{bernardi_hitchhiker_2018} for further details. 

Let $W$ be a finite-dimensional complex vector space and let $X\subseteq\P(W)$ be an irreducible nondegenerate projective variety.

\begin{definition}
\label{def: secant variety}
For an integer $s\geq1$, the \emph{$s$-th secant variety} of $X$ is
\[
\sigma_s(X)
\coloneq
\overline{
\bigcup_{P_1,\ldots,P_s\in X}
\langle P_1,\ldots,P_s\rangle
}
\subseteq
\P(W),
\]
where the  closure is taken under the Zariski topology and $\langle P_1,\ldots,P_s\rangle$ is the projective linear span of the points $P_1,\ldots,P_s$. 
\end{definition}

\begin{definition}
\label{def: secant defectivity}
Set $m\coloneq\dim X$, and $N\coloneq\dim\P(W)$. The \emph{expected dimension} of $\sigma_s(X)$ is
\[\expdim\sigma_s(X)
\coloneq
\min\bigl\{
s(m+1)-1,\,
N
\bigr\}.
\]The variety $X$ is called \emph{$s$-defective} if
\[
\dim\sigma_s(X)
<
\expdim\sigma_s(X).
\]
In this case, the integer
\[\delta_s(X)
\coloneq
\expdim\sigma_s(X)-\dim\sigma_s(X)
\]is called the \emph{$s$-th defect} of $X$.
\end{definition}

\begin{definition}
\label{def: affine tangent space}
Let $\widehat X\subseteq W$ be the affine cone over $X$. For a smooth
point $P=[p]\in X$, with $p\in\widehat X\setminus\{0\}$,
let $T_p\widehat X$ denote the Zariski tangent space to
$\widehat X$ at $p$, and set $\widehat T_PX
\coloneq
T_p\widehat X
\subseteq
W$.
This linear subspace is independent of the chosen nonzero affine
representative of $P$.
\end{definition}

Throughout the paper, all tangent spaces appearing in formulas are affine tangent spaces unless explicitly stated otherwise. Thus, when no confusion can arise, we shall simply refer to $\widehat T_PX$ as the tangent space to $X$ at $P$. The following classical result will be used repeatedly; we refer to~\cite{bernardi_hitchhiker_2018} for further details.

\begin{theorem}[Terracini's Lemma]
\label{thm: terracini}
Let $P_1,\ldots,P_s\in X$ be general points and let $Q\in
\langle P_1,\ldots,P_s\rangle$ be a general point. Then
\[\widehat T_Q\sigma_s(X)
=
\widehat T_{P_1}X
+\cdots+
\widehat T_{P_s}X
\subseteq W.
\]\end{theorem}

Taking annihilators gives the form of Terracini's Lemma that will be used throughout the paper.

\begin{corollary}[Dual form of Terracini's Lemma]
\label{cor: dual terracini}
With the assumptions of \Cref{thm: terracini}, one has
\[\bigl(
\widehat T_Q\sigma_s(X)
\bigr)^\perp
=
\bigcap_{i=1}^s
\bigl(
\widehat T_{P_i}X
\bigr)^\perp
\subseteq
W^*.
\]Consequently, if $N=\dim\P(W)$, then
\[\dim\sigma_s(X)
=
N-
\dim
\left(
\bigcap_{i=1}^s
\bigl(
\widehat T_{P_i}X
\bigr)^\perp
\right).
\]\end{corollary}

For the embedded flag variety $F_{\lambda,n}\subseteq\P(\S_\lambda V)$, \Cref{cor: dual terracini} reduces the computation of the dimensions of its secant varieties to the study of the annihilators $\bigl(\widehat T_PF_{\lambda,n}\bigr)^\perp \subseteq \S_\lambda V^*$ and of their intersections at general points. Our purpose is to describe these spaces intrinsically in the framework of Schur apolarity. As we shall see, the algebraic square of the Schur apolar ideal does not provide the correct answer in general.

\section{Tangential Schur Apolarity}
\label{sect: tangential Schur apolarity}

\subsection{The Schur algebra and Schur apolarity}
\label{subsect: Schur algebra and apolarity}

We recall the constructions from Schur apolarity, we refer to~\cite{StaffolaniThesis, Staffolani2023} for the complete theory and to~\cite{Fulton1997} for the classical Littlewood--Richardson rule. 

Set
\[
\S^\bullet V
\coloneq
\bigoplus_{\alpha}\S_\alpha V,
\qquad
\S^\bullet V^*
\coloneq
\bigoplus_{\alpha}\S_\alpha V^*,
\]
where the direct sums run over the partitions $\alpha$ with $\ell(\alpha)\leq n$. By the Littlewood--Richardson rule, for any two partitions $\alpha$ and $\beta$ one has
\[
\S_\alpha V\otimes\S_\beta V
\simeq
\bigoplus_{\gamma}
\left(\S_\gamma V\right)^{\oplus c_{\alpha,\beta}^{\gamma}},
\]
where $c_{\alpha,\beta}^{\gamma}$ is the corresponding Littlewood--Richardson coefficient. Following~\cite{Staffolani2023}, we fix coherent families of $G$-equivariant multiplication maps
\[
\mathcal N_{\gamma}^{\alpha,\beta}:
\S_\alpha V\otimes\S_\beta V
\longrightarrow
\S_\gamma V \qquad \text{and} \qquad \mathcal M_{\gamma}^{\alpha,\beta}:
\S_\alpha V^*\otimes\S_\beta V^*
\longrightarrow
\S_\gamma V^*.
\]
Their direct sums define products, both denoted by $\diamond$, on $\S^\bullet V$ and $\S^\bullet V^*$.

\begin{definition}
\label{def: Schur algebras}
The algebras $(\S^\bullet V,\diamond)$ and $(\S^\bullet V^*,\diamond)$ are called the \emph{Schur algebra} and the \emph{dual Schur algebra}, respectively. For a vector subspace $A\subseteq\S^\bullet V$, we set
\[
(A)_\alpha
\coloneq
A\cap\S_\alpha V
\]
and call it the \emph{$\alpha$-component} of $A$. We use the same notation for subspaces of $\S^\bullet V^*$.
\end{definition}

\begin{remark}
\label{rem: choice Schur multiplication}
If $c_{\alpha,\beta}^{\gamma}>1$, an equivariant map from $\S_\alpha V\otimes\S_\beta V$ to $\S_\gamma V$ is not unique up to scalar. All products appearing in this paper refer to the coherent choice fixed in~\cite{Staffolani2023}.
\end{remark}

We next recall the Schur apolarity action defined in \cite{Staffolani2023}. For a partition $\alpha$, set
\[
\superwedge^{\alpha'}V
\coloneq
\bigotimes_{c=1}^{\alpha_1}
\superwedge^{\alpha'_c}V,
\qquad
\superwedge^{\alpha'}V^*
\coloneq
\bigotimes_{c=1}^{\alpha_1}
\superwedge^{\alpha'_c}V^*.
\]
Let $\mu\subseteq\lambda$, and set $\mu'_c=0$ for $c>\mu_1$.
Applying the usual skew-symmetric contraction column by column gives a map
\[\widetilde{*}_{\mu,\lambda}:
\superwedge^{\mu'}V^*
\otimes
\superwedge^{\lambda'}V
\longrightarrow
\bigotimes_{c=1}^{\lambda_1}
\superwedge^{\lambda'_c-\mu'_c}V.
\]

\begin{definition}[Schur apolarity action, {\cite[Definition~3.3]{Staffolani2023}}]
\label{def: Schur apolarity}
Let $\lambda$ and $\mu$ be partitions with $\ell(\lambda),\ell(\mu)<n$. If $\mu\subseteq\lambda$, the \emph{Schur apolarity action}
\[
*:
\S_\mu V^*\otimes\S_\lambda V
\longrightarrow
\S_{\lambda/\mu}V
\]
is obtained by restricting the columnwise contraction displayed above to the standard column realizations of $\S_\mu V^*$ and $\S_\lambda V$. If $\mu\not\subseteq\lambda$, we set
\[
\varphi*p=0
\qquad
\text{for every }
\varphi\in\S_\mu V^*
\text{ and }
p\in\S_\lambda V.
\]
Here $\S_{\lambda/\mu}V$ denotes the skew Schur module associated with the skew diagram $\lambda/\mu$.
\end{definition}

\begin{remark}
\label{rem: classical specializations Schur apolarity}
The Schur apolarity action simultaneously extends the two classical constructions. If $\lambda=(d)$ and $\mu=(e)$, then $\S_\lambda V=\Sym^dV$, and $\S_\mu V^*=\Sym^eV^*$, i.e. \Cref{def: Schur apolarity} recovers the usual symmetric apolarity action (\cite{bernardi_hitchhiker_2018}). If $\lambda=(1^d)$, $\mu=(1^e)$, then it recovers the skew-symmetric apolarity action (\cite{ArrondoBernardiMarquesMourrain2021})
$\superwedge^eV^*
\otimes
\superwedge^dV
\longrightarrow
\superwedge^{d-e}V$.
\end{remark}

When $\mu=\lambda$, the skew diagram $\lambda/\mu$ is empty and $\S_{\lambda/\lambda}V\simeq\C$. Thus, Schur apolarity induces a pairing\[\S_\lambda V^*
\otimes
\S_\lambda V
\longrightarrow
\C,
\qquad
(\varphi,p)
\longmapsto
\varphi*p.\]
This pairing is nonzero and $G$-equivariant. By Schur's Lemma, it is perfect and is a nonzero scalar multiple of the evaluation pairing introduced in \eqref{eq: dual tableau bases}. In particular, the two pairings define the same annihilators.

\begin{example}[A columnwise Schur contraction]
\label{ex: columnwise Schur contraction}
Let $\lambda=(2,2,1)$ and $\mu=(2)$, so that $\lambda'=(3,2)$ and  $\mu'=(1,1)$.  In the usual column realizations, consider
\[
p
=
(e_1\wedge e_2\wedge e_3)
\otimes
(e_1\wedge e_2)
\in
\S_\lambda V,
\]
\[
\varphi
=
\alpha\otimes\alpha
\in
\S_\mu V^*
=
\Sym^2V^*,
\quad
\alpha\in V^*.
\]
Since $\mu\subseteq\lambda$, the Schur apolarity action $\varphi*p$ is obtained by applying the contractions
\[
V^*\otimes\superwedge^3V \longrightarrow \superwedge^2V \qquad \text{and} \qquad V^*\otimes\superwedge^2V \longrightarrow V
\]
to the two corresponding columns and then tensoring the resulting elements. More precisely,
\[
\alpha\lrcorner
(e_1\wedge e_2\wedge e_3)
=
\alpha(e_1)e_2\wedge e_3
-
\alpha(e_2)e_1\wedge e_3
+
\alpha(e_3)e_1\wedge e_2,
\]
\[
\alpha\lrcorner
(e_1\wedge e_2)
=
\alpha(e_1)e_2
-
\alpha(e_2)e_1.
\]
Therefore,
\[
\begin{aligned}
\varphi*p
={}&
\alpha(e_1)^2
e_2\wedge e_3\otimes e_2
-
\alpha(e_1)\alpha(e_2)
e_2\wedge e_3\otimes e_1
\\
&-
\alpha(e_1)\alpha(e_2)
e_1\wedge e_3\otimes e_2
+
\alpha(e_2)^2
e_1\wedge e_3\otimes e_1
\\
&+
\alpha(e_1)\alpha(e_3)
e_1\wedge e_2\otimes e_2
-
\alpha(e_2)\alpha(e_3)
e_1\wedge e_2\otimes e_1.
\end{aligned}
\]
This element belongs to $\S_{\lambda/\mu}V
\subseteq
\superwedge^2V\otimes V$.
Thus, in this example, the Schur apolarity action is obtained explicitly by performing one skew-symmetric contraction in each column of the chosen realizations.
\end{example}

For a vector subspace $U\subseteq\S_\lambda V$, we therefore consider the $*$-orthogonal space
\[
U^\perp
\coloneq
\left\{
\varphi\in\S_\lambda V^*
\ \middle|\
\varphi*u=0
\text{ for every }u\in U
\right\}.
\]

\begin{definition}[Schur apolar ideal, {\cite[Definition~3.9]{Staffolani2023}}]
\label{def: Schur apolar ideal}
Let $p\in\S_\lambda V$ be nonzero and let $P=[p]\in\P(\S_\lambda V)$. The \emph{Schur apolar ideal} of $P$ is
\[I_{\mathrm S}(P)
\coloneq
\left\{
\varphi\in\S^\bullet V^*
\ \middle|\
\varphi*p=0
\right\}.\]
This definition does not depend on the choice of the nonzero affine representative $p$ of $P$.
\end{definition}

The following fundamental property was proved by Staffolani.

\begin{proposition}[Ideal property, {\cite[Proposition~4.6 and
Remark~4.7]{Staffolani2023}}]
\label{prop: ideal property}
For every $P\in\P(\S_\lambda V)$, the Schur apolar ideal $I_{\mathrm S}(P) \subseteq \S^\bullet V^*$ is an ideal of the dual Schur algebra $(\S^\bullet V^*,\diamond)$.
\end{proposition}

\begin{remark}
\label{rem: Staffolani decomposition side}
Staffolani also associates with a point $P\in F_{\lambda,n}$ an auxiliary ideal $I_{\mathrm R}(P)
\subseteq
I_{\mathrm S}(P)$ and proves a Schur Apolarity Lemma describing linear decompositions of elements of $\S_\lambda V$ in terms of inclusions of such ideals;
see~\cite[Section~4]{Staffolani2023}. Moreover, $\bigl(I_{\mathrm R}(P)\bigr)_\lambda
=
\bigl(I_{\mathrm S}(P)\bigr)_\lambda$. This develops the decomposition-theoretic side of Schur apolarity.
Our purpose is different: we seek the corresponding tangential construction, capable of describing the annihilator of the tangent space to $F_{\lambda,n}$ and hence, through Terracini's Lemma, the dimensions of its secant varieties.
\end{remark}

\subsection{Tangent and conormal spaces}
\label{subsect: tangent and conormal spaces}

Let $\rho_\lambda: G \rightarrow \GL(\S_\lambda V)$ be the representation associated with the Schur module, and let $d\rho_\lambda: \Lie{gl}(V) \rightarrow \operatorname{End}(\S_\lambda V)$ be its differential.
Since the center of $G$ acts on $\S_\lambda V$ by scalar multiplication, the punctured affine cone over $F_{\lambda,n}$ is the $G$-orbit of $\bv_{\lambda,n}$. Therefore,
\begin{equation}
\label{eq: tangent via action}
\widehat{T}_{P_{\lambda,n}}F_{\lambda,n}
=
d\rho_\lambda\bigl(\Lie{gl}(V)\bigr)
\cdot
\bv_{\lambda,n}
\subseteq
\S_\lambda V.
\end{equation}
We refer to the \emph{affine conormal space} of $F_{\lambda,n}$ at $P_{\lambda,n}$ as
\[
\left(
\widehat{T}_{P_{\lambda,n}}F_{\lambda,n}
\right)^\perp
\subseteq
\S_\lambda V^*.
\]
We compute these two spaces using the standard opposite big cell of the flag variety. We regard $\lambda$ as an $n$-tuple by setting $\lambda_h=0$ for $\ell(\lambda)<h\leq n$, and, for $i<j$, we denote by $E_{ji}\in\operatorname{End}(V)$ the linear map defined by $E_{ji}(e_i)=e_j$ and $E_{ji}(e_h)=0$ for $h\neq i$.
Set
\[\mathfrak u^-_\lambda
\coloneq
\left\{
A=
\sum_{\substack{i<j\\ \lambda_i>\lambda_j}}
a_{ji}E_{ji}
\right\}.
\]The map
\[
A
\longmapsto
\left[
\bigl(\operatorname{Id}_V+A\bigr)\cdot\bv_{\lambda,n}
\right]
\]
identifies $\mathfrak u^-_\lambda$ with the standard affine open neighborhood of $P_{\lambda,n}$ in $F_{\lambda,n}$. We use $A
\mapsto
\bigl(\operatorname{Id}_V+A\bigr)\cdot\bv_{\lambda,n}$ as its affine version. This is the usual opposite-big-cell parametrization of the highest-weight orbit; see, for instance, \cite{Fulton1997,FultonHarris1991}.

\begin{lemma}
\label{lem: minor vanishing order}
Let $r\in\{n_1,\ldots,n_k\}$ be a column height of $\lambda$ and let $J\subseteq[n]$ with $|J|=r$. Denote by $\Delta_J\bigl(\operatorname{Id}_V+A\bigr)$ the minor of $\operatorname{Id}_V+A$ with row set $J$ and column set $[r]$. Let $\mathfrak m = (a_{ji}) \subseteq \C[\mathfrak u^-_\lambda]$ be the maximal ideal of the origin. Then
\[\operatorname{ord}_0
\Delta_J\bigl(\operatorname{Id}_V+A\bigr)
=
|J\setminus[r]|.
\]Here $\operatorname{ord}_0$ denotes the $\mathfrak m$-adic order at the origin of $\mathfrak u^-_\lambda$.
\end{lemma}

\begin{proof}
If $J=[r]$, then the corresponding principal submatrix of $\operatorname{Id}_V+A$ is lower unitriangular, and hence $\Delta_{[r]}
\bigl(\operatorname{Id}_V+A\bigr) = 1$.

Assume now that $J\neq[r]$ and set $q \coloneq |J\setminus[r]| = |[r]\setminus J|$.
Every term in the determinant contains an entry of $A$ for each row belonging to $J\setminus[r]$. Its degree is therefore at least $q$.  The homogeneous component of degree $q$ is, up to sign,
$\det
\left(
a_{ph}
\right)_{
p\in J\setminus[r],\,
h\in[r]\setminus J
}$. This is a nonzero polynomial in $\C[\mathfrak u^-_\lambda]$. Indeed, if $h\in[r]\setminus J$, $p\in J\setminus[r]$, then $h\leq r<p$. Since $r$ is a column height of $\lambda$, one has $\lambda_r>\lambda_{r+1}$, and hence $\lambda_h\geq\lambda_r >\lambda_{r+1} \geq \lambda_p$.
Thus, every $a_{ph}$ occurring in the determinant is one of the coordinates of $\mathfrak u^-_\lambda$. The determinant is therefore a nonzero polynomial of degree $q$, which proves the lemma.\end{proof}

The preceding calculation gives a geometric interpretation of the column distance introduced in \Cref{def: tableau data}.

\begin{proposition}[Column distance and vanishing order]
\label{prop: column distance vanishing order}
Let $S$ be a semistandard tableau of shape $\lambda$ with entries in $[n]$. Then the order of vanishing at $P_{\lambda,n}$ of the section $\bx_S$, restricted to $F_{\lambda,n}$, is
$\operatorname{ord}_{P_{\lambda,n}}(\bx_S)
=
\dist(S,S_0)$.
\end{proposition}

\begin{proof}
Write
$\ind(S)
=
\bigl(
J_S^{(1)},\ldots,J_S^{(\lambda_1)}
\bigr)$.
For $A\in\mathfrak u^-_\lambda$, set
\[f_S(A)
\coloneq
\left\langle
\bx_S,\,
\bigl(\operatorname{Id}_V+A\bigr)\cdot\bv_{\lambda,n}
\right\rangle_\lambda.
\]
Since $\iota_\lambda$ is $G$-equivariant and $\pi_\lambda$ is its dual, the function $f_S$ is, up to a nonzero scalar,
\[
\prod_{c=1}^{\lambda_1}
\Delta_{J_S^{(c)}}
\bigl(\operatorname{Id}_V+A\bigr).
\]
The scalar does not affect its order of vanishing. By \Cref{lem: minor vanishing order},
\[
\begin{aligned}
\operatorname{ord}_{P_{\lambda,n}}(\bx_S)
&=
\operatorname{ord}_0(f_S)
\\
&=
\sum_{c=1}^{\lambda_1}
\left|
J_S^{(c)}
\setminus
[\lambda'_c]
\right|
\\
&=
\dist(S,S_0),
\end{aligned}
\]
where the last equality follows from \Cref{def: tableau data}.
\end{proof}

We can now identify the tangent and conormal spaces in the tableau
bases introduced in \Cref{subsect: cartan realization}.

\begin{proposition}\label{prop: tangent and conormal bases}
With the notation above, one has
\begin{align}
\widehat{T}_{P_{\lambda,n}}F_{\lambda,n}
&=
\left\langle
\be_S
\ \middle|\
\dist(S,S_0)\leq1
\right\rangle,
\label{eq: tangent basis distance}
\\
\left(
\widehat{T}_{P_{\lambda,n}}F_{\lambda,n}
\right)^\perp
&=
\left\langle
\bx_S
\ \middle|\
\dist(S,S_0)\geq2
\right\rangle.
\label{eq: tangent apolar distance}
\end{align}
In particular, the sets $\left\{
\be_S
\ \middle|\
\dist(S,S_0)\leq1
\right\}$ and $\left\{
\bx_S
\ \middle|\
\dist(S,S_0)\geq2
\right\}$ are bases of the affine tangent space and of its conormal space, respectively.
\end{proposition}

\begin{proof}
For $\varphi\in\S_\lambda V^*$, define
$f_\varphi(A)
\coloneq
\left\langle
\varphi,\,
\bigl(\operatorname{Id}_V+A\bigr)\cdot\bv_{\lambda,n}
\right\rangle_\lambda$.
Its constant term is $f_\varphi(0)
=
\langle\varphi,\bv_{\lambda,n}\rangle_\lambda$, whereas its derivative in the direction $B\in\mathfrak u^-_\lambda$ is $d(f_\varphi)_0(B) = \left\langle
\varphi,\, B\cdot\bv_{\lambda,n} \right\rangle_\lambda$.
It follows from \eqref{eq: tangent via action} and the preceding big-cell parametrization that
\[
\varphi\in
\left(
\widehat T_{P_{\lambda,n}}F_{\lambda,n}
\right)^\perp
\quad\Longleftrightarrow\quad
f_\varphi\in\mathfrak m^2.
\]Equivalently, the conormal space is the kernel of the first-jet map
\[j^1_{P_{\lambda,n}}:
\S_\lambda V^*
\longrightarrow
\C[\mathfrak u^-_\lambda]/\mathfrak m^2,
\qquad
\varphi
\longmapsto
f_\varphi
\pmod{\mathfrak m^2}.
\]
We now identify the constant and linear terms of the tableau sections. There is exactly one tableau at distance zero from $S_0$, namely $S_0$ itself, and $f_{S_0}(A)=1$.

Suppose that $\dist(S,S_0)=1$. There is a unique column, say the $c$-th one, which is not a highest-weight column. If its height is $r$, then its content has the form $[r]\setminus\{i\}\cup\{j\}$ for $i\leq r<j$.
The row condition forces $c=\lambda_i$.
Indeed, the entry in row $i$ of the modified column is strictly larger than $i$. Hence no column lying to its right can reach row $i$, and the modified column must be the last column of row $i$. Moreover, since the modified column has height $r<j$, one has $\lambda_i>\lambda_j$.

Conversely, for every pair $i<j$, $\lambda_i>\lambda_j$, setting $c=\lambda_i$, $r=\lambda'_c$,  and replacing $i$ by $j$ in the $c$-th highest-weight column produces a unique semistandard tableau $S_{ji}$ satisfying $\dist(S_{ji},S_0)=1$. By \Cref{lem: minor vanishing order}, its local function has the form $
f_{S_{ji}}(A)
=
c_{ji}a_{ji}
\pmod{\mathfrak m^2}
$
for some $c_{ji}\in\C^*$. Therefore, the first jets of the sections $\bx_S$ with $\dist(S,S_0)\leq1$ are the class of $1$ together with nonzero multiples of the distinct coordinate classes $a_{ji}$. They form a basis of $\C[\mathfrak u^-_\lambda]/\mathfrak m^2$.

On the other hand, by \Cref{prop: column distance vanishing order}, every $\bx_S$ with $\dist(S,S_0)\geq2$ has zero first jet. Since the elements $\bx_S$ form a basis of $\S_\lambda V^*$, the kernel of this first-jet map is therefore
$
\left\langle
\bx_S
\ \middle|\
\dist(S,S_0)\geq2
\right\rangle$. Together with the preceding characterization of the conormal space, this proves \eqref{eq: tangent apolar distance}.

Finally, by \eqref{eq: dual tableau bases}, the annihilator of $\left\langle
\be_S
\ \middle|\
\dist(S,S_0)\leq1
\right\rangle$  is exactly $\left\langle \bx_S
\ \middle|\ \dist(S,S_0)\geq2 \right\rangle$.
By \eqref{eq: tangent apolar distance}, this is also the annihilator of $\widehat{T}_{P_{\lambda,n}}F_{\lambda,n}$.
Since the evaluation pairing is perfect, the two subspaces with the same annihilator coincide. This proves \eqref{eq: tangent basis distance}.
\end{proof}

\begin{remark}[Homogeneity]
\label{rem: tangent conormal homogeneity}
Let $P=g\cdot P_{\lambda,n}
\in F_{\lambda,n}$, $g\in G$.
Then
$
\widehat{T}_PF_{\lambda,n}
=
g\cdot
\widehat{T}_{P_{\lambda,n}}F_{\lambda,n}$,
and, with respect to the dual action of $G$ on $\S_\lambda V^*$,
$
\left(
\widehat{T}_PF_{\lambda,n}
\right)^\perp
=
g\cdot
\left(
\widehat{T}_{P_{\lambda,n}}F_{\lambda,n}
\right)^\perp$.
Thus, all intrinsic statements concerning tangent and conormal spaces may be verified at the highest-weight point.
\end{remark}

\begin{example}
\label{ex: tangent conormal two step three dimensional}
Let $V\simeq\C^3$ and $\lambda=(2,1)$.
Then $F_{(2,1),3}
\simeq
\Fl(1,2;V)
\subseteq
\P(\S_{(2,1)}V)$. The semistandard tableaux of shape $(2,1)$ with entries in $[3]$ are
\[
\begin{ytableau}
1 & 1\\
2
\end{ytableau},
\quad
\begin{ytableau}
1 & 1\\
3
\end{ytableau},
\quad
\begin{ytableau}
1 & 2\\
2
\end{ytableau},
\quad
\begin{ytableau}
1 & 2\\
3
\end{ytableau},
\quad
\begin{ytableau}
1 & 3\\
2
\end{ytableau},
\quad
\begin{ytableau}
1 & 3\\
3
\end{ytableau},
\quad
\begin{ytableau}
2 & 2\\
3
\end{ytableau},
\quad
\begin{ytableau}
2 & 3\\
3
\end{ytableau}.
\]
We denote them, in the displayed order, by $S_0,S_1,\ldots,S_7$. Their column distances from $S_0$ are, respectively, $0,1,1,2,1,2,2,2$. Therefore,
\[
\widehat{T}_{P_{(2,1),3}}F_{(2,1),3}
=
\left\langle
\be_{S_0},
\be_{S_1},
\be_{S_2},
\be_{S_4}
\right\rangle,
\]
\[
\left(
\widehat{T}_{P_{(2,1),3}}F_{(2,1),3}
\right)^\perp
=
\left\langle
\bx_{S_3},
\bx_{S_5},
\bx_{S_6},
\bx_{S_7}
\right\rangle.
\]
Thus,  $\left\{ \be_{S_0}, \be_{S_1},
\be_{S_2}, \be_{S_4} \right\}
$ is a basis of the affine tangent space, and $\left\{ \bx_{S_3}, \bx_{S_5}, \bx_{S_6},
\bx_{S_7} \right\} $ is a basis of its conormal space. In particular, $\dim \widehat{T}_{P_{(2,1),3}}F_{(2,1),3} = 4$, as expected from $\dim F_{(2,1),3}=3$.
\end{example}

\subsection{The Veronese model}
\label{subsect: Veronese model}

The classical symmetric case provides the model for the tangential construction developed in the following subsections. Let $d\geq1$. In the notation of \Cref{def: embedded flag variety}, the partition $\lambda=(d)$ gives $F_{(d),n} = \nu_d\bigl(\P(V)\bigr)
\subseteq \P(\Sym^dV)$, where $\nu_d$ is the $d$-th Veronese embedding of $\P(V)$.
Let
\[
p=[v]\in\P(V),
\qquad
P=\nu_d(p)=[v^d]\in F_{(d),n}.
\]
We denote by
\[
I_{p,\P(V)}
\subseteq
\Sym^\bullet V^*
\]
the homogeneous ideal of $p$.

\begin{proposition}[Lasker's Lemma]
\label{prop: Lasker lemma}
With the notation above, one has
\[\left(
\widehat T_PF_{(d),n}
\right)^\perp
=
\left(
I_{p,\P(V)}^2
\right)_d
\subseteq
\Sym^dV^*,\]
where the annihilator is taken with respect to the classical symmetric apolarity pairing.
\end{proposition}

This classical result is usually attributed to Lasker; see, for instance, \cite{Ottaviani2013}.

\begin{remark}
\label{rem: Veronese algebraic geometric squares}
The ideal $I_{p,\P(V)}$ is generated by linear forms. Consequently,  its algebraic square is the homogeneous ideal of the geometric double point supported at $p$:\[I_{p,\P(V)}^2
=
I_{p,\P(V)}^{\langle2\rangle}.
\]
Hence \Cref{prop: Lasker lemma} can equivalently be written as
\[\left(
\widehat T_PF_{(d),n}
\right)^\perp
=
\left(
I_{p,\P(V)}^{\langle2\rangle}
\right)_d.
\]
Thus, in the Veronese case, the same vector space admits three compatible interpretations:

\begin{itemize}
    \item it is the apolar annihilator of the affine tangent space;
    \item it is the degree-$d$ component of the algebraic square of the point ideal;
    \item it is the space of degree-$d$ hypersurfaces vanishing at  $p$ together with their first derivatives.
\end{itemize}
\end{remark}

The preceding description is also compatible with Schur apolarity.
Indeed, for every $0\leq e\leq d$, the Schur apolarity action on
\[
\S_{(e)}V^*\otimes\S_{(d)}V
=
\Sym^eV^*\otimes\Sym^dV
\]
is the classical symmetric apolarity action. Therefore,
\[\left(
I_{\mathrm S}(P)
\right)_{(e)}
=
\left(
I_{p,\P(V)}
\right)_e
\qquad
\text{for every }0\leq e\leq d.
\]Moreover, only one-row partitions can contribute to the $(d)$-component of a product whose target is $\S_{(d)}V^*$. It follows that
\[\left(
I_{\mathrm S}(P)^2
\right)_{(d)}
=
\left(
I_{p,\P(V)}^2
\right)_d.
\]Combining \Cref{prop: Lasker lemma} with the preceding equality, we obtain
\begin{equation}
\label{eq: Veronese triple identification}
\left(
\widehat T_PF_{(d),n}
\right)^\perp
=
\left(
I_{\mathrm S}(P)^2
\right)_{(d)}
=
\left(
I_{p,\P(V)}^{\langle2\rangle}
\right)_d.
\end{equation}

\begin{corollary}[Dual Terracini for Veronese varieties]
\label{cor: dual Terracini Veronese}
Let $P_i=\nu_d(p_i)\in F_{(d),n}$, $i=1,\ldots,s$, be general points, and let $Q\in
\langle P_1,\ldots,P_s\rangle$ be general. Then
\begin{align*}
\left(
\widehat T_Q\sigma_s(F_{(d),n})
\right)^\perp
&=
\bigcap_{i=1}^s
\left(
I_{p_i,\P(V)}^{\langle2\rangle}
\right)_d\\
&=
\bigcap_{i=1}^s
\left(
I_{\mathrm S}(P_i)^2
\right)_{(d)}.
\end{align*}
\end{corollary}

\begin{proof}
The statement follows from
\Cref{cor: dual terracini} and \Cref{prop: Lasker lemma} and from
\eqref{eq: Veronese triple identification}.
\end{proof}

As a consequence, the computation of the dimensions of secant varieties of Veronese varieties becomes a polynomial interpolation problem for hypersurfaces singular at general points. From the Schur-apolar point of view, the crucial feature is the equality
\[
\left(
I_{\mathrm S}(P)^2
\right)_{(d)}
=
\left(
\widehat T_PF_{(d),n}
\right)^\perp.
\]
This naturally raises the question whether $\left( \widehat T_PF_{\lambda,n} \right)^\perp \stackrel{?}{=} \left( I_{\mathrm S}(P)^2 \right)_\lambda$ for an arbitrary partition $\lambda$. The next subsection shows that, unlike in the Veronese case, this identity fails for general Schur modules.

\subsection{Failure of the algebraic Schur square}
\label{subsect: failure algebraic Schur square}

The Veronese model suggests that, for a point $P\in F_{\lambda,n}$, the conormal space to the embedded flag variety might be recovered as the $\lambda$-component of the algebraic square of its Schur apolar ideal:
\begin{equation}
\label{eq: naive algebraic Schur square}
\left(
\widehat T_PF_{\lambda,n}
\right)^\perp
\stackrel{?}{=}
\left(
I_{\mathrm S}(P)^2
\right)_\lambda.
\end{equation}
Here the square is taken with respect to the multiplication $\diamond$ in the dual Schur algebra. Unlike in the Veronese case, the equality in \eqref{eq: naive algebraic Schur square} does not hold for arbitrary Schur modules.

\begin{proposition}[Failure of the algebraic Schur square]
\label{prop: algebraic Schur square fails}
Let $V\simeq\C^3$ and $\lambda=(4,2)$. Then $F_{(4,2),3}
\simeq
\Fl(1,2;V)
\subseteq
\P(\S_{(4,2)}V)$ is the embedding induced by $\cO(2,2)$. Let
\[
S_0
=
\begin{ytableau}
1 & 1 & 1 & 1\\
2 & 2
\end{ytableau}
\qquad\text{and}\qquad
S
=
\begin{ytableau}
1 & 1 & 2 & 2\\
2 & 2
\end{ytableau}.
\]
Equivalently, $ \ind(S) = \bigl( \{1,2\}, \{1,2\}, \{2\}, \{2\} \bigr)$. Then
\[\bx_S
\in
\left(
\widehat T_{P_{(4,2),3}}F_{(4,2),3}
\right)^\perp
\setminus
\left(
I_{\mathrm S}(P_{(4,2),3})^2
\right)_{(4,2)}.
\]
Therefore,
\[
\left(
\widehat T_{P_{(4,2),3}}F_{(4,2),3}
\right)^\perp
\neq
\left(
I_{\mathrm S}(P_{(4,2),3})^2
\right)_{(4,2)}.
\]
\end{proposition}

\begin{proof}
By construction, $\dist(S,S_0)=2$.
Therefore,
\Cref{prop: tangent and conormal bases} gives $\bx_S
\in
\left(
\widehat T_{P_{(4,2),3}}F_{(4,2),3}
\right)^\perp$.
The non-inclusion $\bx_S \notin \left( I_{\mathrm S}(P_{(4,2),3})^2 \right)_{(4,2)}$ is proved in \Cref{lem: schur square failure 42}.
\end{proof}

For the proof of the non-inclusion, we introduce the following standard weight-space notation, which will also be used in the slot-by-slot construction in the sequel.

\begin{definition}[Content spaces]
\label{def: content spaces}
Let $\mathbb T\subseteq G$ be the diagonal torus and let $M$ be a finite-dimensional $\mathbb T$-stable subspace of a dual Schur module or of one of the dual tensor spaces considered below. For $\nu=(\nu_1,\ldots,\nu_n)\in\N^n$, we define the \emph{content-$\nu$ subspace} of $M$ by
\[
M[\nu]
\coloneq
\left\{
m\in M
\ \middle|\
t\cdot m
=
t_1^{-\nu_1}\cdots t_n^{-\nu_n}m
\text{ for every }
t=\operatorname{diag}(t_1,\ldots,t_n)\in\mathbb T
\right\}.
\]
\end{definition}

Since $\mathbb T$ is diagonalizable, every finite-dimensional $\mathbb T$-stable space decomposes as
\[
M
=
\bigoplus_{\nu\,:\,M[\nu]\neq0}
M[\nu].
\]
Moreover, every $\mathbb T$-equivariant linear map preserves the content spaces. In particular, all $G$-equivariant Schur multiplication maps preserve content.

\begin{lemma}[The factorization obstruction in type $(4,2)$]
\label{lem: schur square failure 42}
Let $V\simeq\C^3$ and  $\lambda=(4,2)$, and let $S$ be the semistandard tableau with column index sequence $\ind(S)
=
\bigl(
\{1,2\},
\{1,2\},
\{2\},
\{2\}
\bigr)$. Then
\[
\bx_S
\notin
\left(
I_{\mathrm S}(P_{(4,2),3})^2
\right)_{(4,2)}.
\]
\end{lemma}

\begin{proof}
Set $\gamma
\coloneq
\cont(S)
=
(2,4,0)$.

\paragraph{\textit{Reduction to a content-homogeneous factorization.}}

We first show that, if
\[
\bx_S
\in
\left(
I_{\mathrm S}(P_{(4,2),3})^2
\right)_{(4,2)},
\]
then there exist partitions $\alpha,\beta$, contents $\eta,\theta$ and nonzero elements
\[
u
\in
\bigl(I_{\mathrm S}(P_{(4,2),3})\bigr)_\alpha[\eta],
\qquad
v
\in
\bigl(I_{\mathrm S}(P_{(4,2),3})\bigr)_\beta[\theta]
\]
such that $\eta+\theta=\gamma$ and $\bx_S
=
\cM^{\alpha,\beta}_{(4,2)}
(u\otimes v)$.

Let $\varphi
\in
\bigl(I_{\mathrm S}(P_{(4,2),3})\bigr)_\mu$.
Since $
\left\langle\bv_{(4,2),3}\right\rangle$ is $\mathbb T$-stable, for every $t\in\mathbb T$ there exists a nonzero scalar $c_t$ such that
\[
t\cdot\bv_{(4,2),3}
=
c_t\bv_{(4,2),3}.
\]
By the equivariance and bilinearity of the Schur apolarity action,
\[
\begin{aligned}
c_t
\bigl(
(t\cdot\varphi)*\bv_{(4,2),3}
\bigr)
&=
(t\cdot\varphi)*
(t\cdot\bv_{(4,2),3})
\\
&=
t\cdot
\bigl(
\varphi*\bv_{(4,2),3}
\bigr)
=
0.
\end{aligned}
\]
Since $c_t\neq0$, it follows that $(t\cdot\varphi)*\bv_{(4,2),3}=0$, and hence
\[
t\cdot\varphi
\in
I_{\mathrm S}(P_{(4,2),3}).
\]
Moreover, $t\cdot\varphi\in\S_\mu V^*$, because $\S_\mu V^*$ is $\mathbb T$-stable. Therefore every component $\bigl(I_{\mathrm S}(P_{(4,2),3})\bigr)_\mu$ is $\mathbb T$-stable. Since the multiplication maps $\cM^{\alpha,\beta}_{(4,2)}$ are $G$-equivariant, they preserve the content spaces introduced in \Cref{def: content spaces}. The Schur apolar ideal decomposes as
\[
I_{\mathrm S}(P_{(4,2),3})
=
\bigoplus_\alpha
\bigl(I_{\mathrm S}(P_{(4,2),3})\bigr)_\alpha.
\]
Since each component on the right-hand side is $\mathbb T$-stable, it further decomposes into content spaces:
\[
I_{\mathrm S}(P_{(4,2),3})
=
\bigoplus_{\alpha,\eta}
\bigl(I_{\mathrm S}(P_{(4,2),3})\bigr)_\alpha[\eta].
\]
It follows that
\begin{equation}
\label{eq: content decomposition schur square}
\begin{aligned}
&
\left(
I_{\mathrm S}(P_{(4,2),3})^2
\right)_{(4,2)}[\gamma]
\\
&\qquad =
\sum_{\alpha,\beta,\eta,\theta}
\cM^{\alpha,\beta}_{(4,2)}
\left(
\bigl(I_{\mathrm S}(P_{(4,2),3})\bigr)_\alpha[\eta]
\otimes
\bigl(I_{\mathrm S}(P_{(4,2),3})\bigr)_\beta[\theta]
\right),
\end{aligned}
\end{equation}
where the sum runs over the partitions and contents satisfying $c_{\alpha,\beta}^{(4,2)}>0$,  $\eta+\theta=\gamma$. The tableau $S$ is the unique semistandard tableau of shape $(4,2)$ and content $\gamma$. Therefore, $\S_{(4,2)}V^*[\gamma] = \C\bx_S$.

Suppose, by contradiction, that $\bx_S
\in
\left(
I_{\mathrm S}(P_{(4,2),3})^2
\right)_{(4,2)}$.
By \eqref{eq: content decomposition schur square}, the vector $\bx_S$ is a sum of elements of the form $\cM^{\alpha,\beta}_{(4,2)}
(u'\otimes v')$, where
$u'
\in
\bigl(I_{\mathrm S}(P_{(4,2),3})\bigr)_\alpha[\eta]$, $v'
\in
\bigl(I_{\mathrm S}(P_{(4,2),3})\bigr)_\beta[\theta]$, $\eta+\theta=\gamma$.
Every such summand belongs to $\S_{(4,2)}V^*[\gamma] = \C\bx_S$.
Since their sum is the nonzero vector $\bx_S$, at least one summand equals $c\bx_S$ for some $c\in\C^*$. Rescaling one of its factors, we obtain partitions $\alpha,\beta$, contents $\eta,\theta$, and nonzero elements
\[u
\in
\bigl(I_{\mathrm S}(P_{(4,2),3})\bigr)_\alpha[\eta], \qquad v
\in
\bigl(I_{\mathrm S}(P_{(4,2),3})\bigr)_\beta[\theta]
\]
such that
\[\eta+\theta=\gamma,
\qquad
\bx_S
=
\cM^{\alpha,\beta}_{(4,2)}
(u\otimes v).
\]

\paragraph{\textit{Possible Littlewood--Richardson factorizations.}}

The ideal $I_{\mathrm S}(P_{(4,2),3})$ contains no nonzero scalars, since
\[
c*\bv_{(4,2),3}
=
c\bv_{(4,2),3}
\qquad
\text{for every }c\in\C.
\]
Thus, since $u$ and $v$ are nonzero, the partitions $\alpha$ and $\beta$ are both nonempty. Moreover, the dual Schur algebra is graded by the size of the partitions, so $|\alpha|+|\beta| = |(4,2)| = 6$.
The Littlewood--Richardson rule gives the following possible unordered pairs:
\[
\begin{gathered}
((1),(4,1)),
\quad
((1),(3,2)),
\quad
((2),(4)),
\quad
((2),(3,1)),
\quad
((2),(2,2)),
\\
((1,1),(3,1)),
\quad
((2,1),(3)),
\quad
((2,1),(2,1)),
\quad
((3),(3)).
\end{gathered}
\]
In each case, we examine the factor having the indicated shape, whether it is $u$ or $v$, and relabel the two factors accordingly.
Since $\eta,\theta\in\N^3$,  $\eta+\theta=(2,4,0)$, we have $\eta_3=\theta_3=0$.
Hence every semistandard tableau indexing a basis element of $\S_\alpha V^*[\eta]$ or  $\S_\beta V^*[\theta]$ has entries only in $\{1,2\}$. All the partitions in the preceding list have at most two rows. For each such shape, there is at most one semistandard tableau with a prescribed content and entries in $\{1,2\}$. Therefore the spaces $\S_\alpha V^*[\eta]$ and $\S_\beta V^*[\theta]$ have dimension at most one. Since $u$ and $v$ are nonzero, they are scalar multiples of the corresponding tableau basis elements. Up to a nonzero scalar, one has
\begin{equation}
\label{eq: highest weight vector 42}
\bv_{(4,2),3}
\sim
e_{[2]}
\otimes
e_{[2]}
\otimes
e_1
\otimes
e_1.
\end{equation}

\paragraph{\textit{Exclusion of the factors.}}
\begin{itemize}
    \item 
Assume first that $\alpha=(1)$.
This occurs for the pairs $((1),(4,1))$ and $((1),(3,2))$. Since $u\neq0$, its content $\eta$ is either $(1,0,0)$ or $(0,1,0)$, and $u$ is a scalar multiple of $x_1$ or $x_2$, respectively. Since
\[
x_1\mathbin{\lrcorner}e_{[2]}
=
e_2,
\qquad
x_2\mathbin{\lrcorner}e_{[2]}
=
-e_1,
\]
neither $x_1$ nor $x_2$ annihilates $\bv_{(4,2),3}$ via Schur apolarity. Thus,
\[
\bigl(I_{\mathrm S}(P_{(4,2),3})\bigr)_{(1)}[(1,0,0)]
=
\bigl(I_{\mathrm S}(P_{(4,2),3})\bigr)_{(1)}[(0,1,0)]
=
0.
\]
This contradicts $u
\in
\bigl(I_{\mathrm S}(P_{(4,2),3})\bigr)_{(1)}[\eta]$ and excludes both pairs.
\item 
Assume now that $\alpha=(2)$.
This occurs for the pairs $((2),(4))$, $((2),(3,1))$ and $((2),(2,2))$. 
The content $\eta$ is one of $(2,0,0)$, $(1,1,0)$, $(0,2,0)$, so $u$ is a scalar multiple of $x_1^2$, $x_1x_2$ and $x_2^2$ respectively. A direct computation shows that none of these elements annihilates $\bv_{(4,2),3}$ via Schur apolarity. Therefore,
\[
\bigl(I_{\mathrm S}(P_{(4,2),3})\bigr)_{(2)}[(2,0,0)]
 =
\bigl(I_{\mathrm S}(P_{(4,2),3})\bigr)_{(2)}[(1,1,0)] =
\bigl(I_{\mathrm S}(P_{(4,2),3})\bigr)_{(2)}[(0,2,0)]
=0.
\]
This contradicts $u
\in
\bigl(I_{\mathrm S}(P_{(4,2),3})\bigr)_{(2)}[\eta]$
and excludes all three pairs.
\item  Suppose that $\alpha=(1,1)$.
In this case, $(\alpha,\beta) = ((1,1),(3,1))$. Since $u\neq0$, the column condition forces $\eta=(1,1,0)$, and $u$ is a scalar multiple of $x_{[2]} = x_1\wedge x_2$. Since $x_{[2]}\mathbin{\lrcorner}e_{[2]} = 1$, this element does not annihilate $\bv_{(4,2),3}$. Hence
\[
\bigl(I_{\mathrm S}(P_{(4,2),3})\bigr)_{(1,1)}[(1,1,0)]
=
0,
\]
contradicting $u \in \bigl(I_{\mathrm S}(P_{(4,2),3})\bigr)_{(1,1)}[\eta]$.

\item Suppose that $\alpha=(2,1)$.
This occurs for the pairs $((2,1),(3))$, and $((2,1),(2,1))$. Let $T_1$ and $T_2$ be the semistandard tableaux of shape $(2,1)$ with column index sequences
\[
\ind(T_1)
=
\bigl(
\{1,2\},
\{1\}
\bigr),
\qquad
\ind(T_2)
=
\bigl(
\{1,2\},
\{2\}
\bigr).
\]
Their contents are
\[
\cont(T_1)=(2,1,0),
\qquad
\cont(T_2)=(1,2,0),
\]
and they are the only semistandard tableaux of shape $(2,1)$ with entries in $\{1,2\}$. Since $u\neq0$, it is a scalar multiple of either $\bx_{T_1}$ or $\bx_{T_2}$. We find that $\bx_{T_1}*\bv_{(4,2),3}$ and $\bx_{T_2}*\bv_{(4,2),3}$ are nonzero scalar multiples of $e_2\otimes e_1\otimes e_1$ and $e_1\otimes e_1\otimes e_1$, respectively. Thus,
\[
\bigl(I_{\mathrm S}(P_{(4,2),3})\bigr)_{(2,1)}[(2,1,0)]
=\bigl(I_{\mathrm S}(P_{(4,2),3})\bigr)_{(2,1)}[(1,2,0)]
=0.
\]
This contradicts $u
\in
\bigl(I_{\mathrm S}(P_{(4,2),3})\bigr)_{(2,1)}[\eta]$ and excludes both pairs.

\item The only remaining possibility is $\alpha=\beta=(3)$.
The possible contents of $u$ are $(3,0,0)$, $(2,1,0)$, $(1,2,0)$ and $(0,3,0)$ and the corresponding content spaces are generated by $x_1^3$, $x_1^2x_2$, $x_1x_2^2$ and $x_2^3$, respectively. A direct computation shows that the first three elements act nontrivially on $\bv_{(4,2),3}$ via Schur apolarity. Hence,
\[
\bigl(I_{\mathrm S}(P_{(4,2),3})\bigr)_{(3)}[(3,0,0)]
 = 
\bigl(I_{\mathrm S}(P_{(4,2),3})\bigr)_{(3)}[(2,1,0)]
=
\bigl(I_{\mathrm S}(P_{(4,2),3})\bigr)_{(3)}[(1,2,0)]
=0.
\]
On the other hand, $x_2^3*\bv_{(4,2),3}=0$, since one of the factors $x_2$ is contracted with one of the factors $e_1$ of \eqref{eq: highest weight vector 42}, and $x_2(e_1)=0$.
Therefore,
\[
\bigl(I_{\mathrm S}(P_{(4,2),3})\bigr)_{(3)}[(0,3,0)]
=
\C x_2^3.
\]
Since $u
\in
\bigl(I_{\mathrm S}(P_{(4,2),3})\bigr)_{(3)}[\eta]$, and  $v \in \bigl(I_{\mathrm S}(P_{(4,2),3})\bigr)_{(3)}[\theta]$ are nonzero, both $u$ and $v$ are scalar multiples of $x_2^3$.
Thus, $\eta=\theta=(0,3,0)$, which contradicts $\eta+\theta
=
\gamma
=
(2,4,0)$.
\end{itemize}\end{proof}

\begin{remark}[The embedding in the counterexample]
\label{rem: embedding of square counterexample}
The counterexample concerns the $\cO(2,2)$-embedding of $\Fl(1,2;V)$ corresponding to the partition $(4,2)$. This is different from the $\cO(1,1)$-embedding associated with the partition $(2,1)$, which will be the main object of the applications in the second part of the paper.
\end{remark}

The same counterexample also rules out the algebraic square of Staffolani's auxiliary ideal. Indeed, $I_{\mathrm R}(P_{(4,2),3})
\subseteq
I_{\mathrm S}(P_{(4,2),3})$, and therefore $I_{\mathrm R}(P_{(4,2),3})^2
\subseteq
I_{\mathrm S}(P_{(4,2),3})^2$.
It follows from \Cref{prop: algebraic Schur square fails} that also $\bx_S
\notin
\left(
I_{\mathrm R}(P_{(4,2),3})^2
\right)_{(4,2)}$. The obstruction is therefore not caused by the choice between $I_{\mathrm R}(P)$ and $I_{\mathrm S}(P)$. Rather, it shows that multiplication in the dual Schur algebra does not, in general, encode first-order vanishing along the highest-weight orbit. To recover the tangential geometry while remaining inside the dual Schur algebra, we must introduce a different square, defined through Schur apolarity and the infinitesimal $G$-action.

\subsection{The geometric Schur square and the Schur Dual Terracini Lemma}
\label{subsect: geometric Schur square}

The counterexample in \Cref{prop: algebraic Schur square fails} shows that multiplication in the dual Schur algebra does not, in general, encode first-order vanishing along the closed orbit. We now introduce the correct replacement for the algebraic square.

The natural $G$-actions on the Schur modules induce a componentwise action of $\Lie{gl}(V)$ on $\S^\bullet V^*$. Since the multiplication maps defining $\diamond$ are $G$-equivariant, this infinitesimal action is by derivations: for every $\xi\in\Lie{gl}(V)$, $\varphi,\psi\in\S^\bullet V^*$:
\[\xi\cdot(\varphi\diamond\psi)
=
(\xi\cdot\varphi)\diamond\psi
+
\varphi\diamond(\xi\cdot\psi).
\]
We shall also use the infinitesimal form of the equivariance of Schur apolarity.

\begin{lemma}
\label{lem: infinitesimal equivariance Schur apolarity}
Let $\varphi\in\S_\mu V^*$, $p\in\S_\lambda V$, $\xi\in\Lie{gl}(V)$. Then
\[(\xi\cdot\varphi)*p
+
\varphi*(\xi\cdot p)
=
\xi\cdot(\varphi*p).
\]In particular, if $\varphi*p=0$, then
\[
(\xi\cdot\varphi)*p
=
-
\varphi*(\xi\cdot p).
\]\end{lemma}

\begin{proof}
The Schur apolarity action is $G$-equivariant, so $(g\cdot\varphi)*(g\cdot p)
=
g\cdot(\varphi*p)$ for every $g\in G$. Differentiating this identity at the identity of $G$ in the direction $\xi$ gives the asserted identity.
\end{proof}

\begin{definition}[Geometric Schur square]
\label{def: geometric Schur square}
Let $P=[p]\in\P(\S_\lambda V)$. The \emph{geometric Schur square} of the Schur apolar ideal of $P$ is
\[I_{\mathrm S}(P)^{\langle2\rangle}
\coloneq
\left\{
\varphi\in\S^\bullet V^*
\ \middle|\
\begin{array}{l}
\varphi*p=0,\\[2pt]
(\xi\cdot\varphi)*p=0
\text{ for every }\xi\in\Lie{gl}(V)
\end{array}
\right\}.
\]Equivalently,
\[I_{\mathrm S}(P)^{\langle2\rangle}
=
\left\{
\varphi\in I_{\mathrm S}(P)
\ \middle|\
\Lie{gl}(V)\cdot\varphi
\subseteq
I_{\mathrm S}(P)
\right\}.
\]
The definition is independent of the choice of the nonzero affine representative $p$ of $P$.
\end{definition}

\begin{remark}
\label{rem: gl versus sl geometric Schur square}
The definition is unchanged if $\Lie{gl}(V)$ is replaced by $\Lie{sl}(V)$. Indeed, $\Lie{gl}(V) = \Lie{sl}(V)\oplus\C\operatorname{Id}_V$, and the identity endomorphism acts on each component $\S_\mu V^*$ by scalar multiplication. Therefore, the condition corresponding to $\operatorname{Id}_V$ is already implied by $\varphi*p=0$. We use $\Lie{gl}(V)$ throughout in accordance with the convention $G=\GL(V)$.
\end{remark}

\begin{proposition}
\label{prop: geometric Schur square ideal}
Let $P\in\P(\S_\lambda V)$. Then $I_{\mathrm S}(P)^{\langle2\rangle}
\subseteq
\S^\bullet V^*$ is an ideal of the dual Schur algebra. Moreover,
\[I_{\mathrm S}(P)^2
\subseteq
I_{\mathrm S}(P)^{\langle2\rangle}.
\]\end{proposition}

\begin{proof}
Let $\varphi\in I_{\mathrm S}(P)^{\langle2\rangle}$, $\psi\in\S^\bullet V^*$. Since $\varphi\in I_{\mathrm S}(P)$ and $I_{\mathrm S}(P)$ is an ideal by \Cref{prop: ideal property}, one has $(\psi\diamond\varphi)*p=0$.

Now let $\xi\in\Lie{gl}(V)$. By the derivation rule displayed above,  $\xi\cdot(\psi\diamond\varphi)
=
(\xi\cdot\psi)\diamond\varphi
+
\psi\diamond(\xi\cdot\varphi)$.
The first summand belongs to $I_{\mathrm S}(P)$ because $\varphi\in I_{\mathrm S}(P)$, whereas the second belongs to $I_{\mathrm S}(P)$ because $\xi\cdot\varphi\in I_{\mathrm S}(P)$ by the definition of the geometric Schur square. Hence $\bigl( \xi\cdot(\psi\diamond\varphi) \bigr)*p = 0$. Therefore, $\psi\diamond\varphi
\in
I_{\mathrm S}(P)^{\langle2\rangle}$, which proves that $I_{\mathrm S}(P)^{\langle2\rangle}$ is an ideal.

To prove the asserted inclusion, let $
\varphi,\psi\in I_{\mathrm S}(P)$. Since $I_{\mathrm S}(P)$ is an ideal, $(\varphi\diamond\psi)*p=0$. Moreover, $\xi\cdot(\varphi\diamond\psi) = (\xi\cdot\varphi)\diamond\psi + \varphi\diamond(\xi\cdot\psi)$.
Both summands belong to $I_{\mathrm S}(P)$: the first because $\psi\in I_{\mathrm S}(P)$ and the second because $\varphi\in I_{\mathrm S}(P)$. Thus, $\bigl( \xi\cdot(\varphi\diamond\psi) \bigr)*p = 0$ for every $\xi\in\Lie{gl}(V)$, and hence $\varphi\diamond\psi \in I_{\mathrm S}(P)^{\langle2\rangle}$. Taking linear combinations proves the desired inclusion.
\end{proof}

We can now identify the component of the geometric Schur square that corresponds to the given embedding.

\begin{theorem}[Tangential Schur Apolarity]
\label{thm: tangential Schur apolarity}
Let $P\in F_{\lambda,n}$. Then
\[\left(
I_{\mathrm S}(P)^{\langle2\rangle}
\right)_\lambda
=
\left(
\widehat T_PF_{\lambda,n}
\right)^\perp
\subseteq
\S_\lambda V^*.
\]
\end{theorem}

\begin{proof}
Choose a nonzero affine representative $p\in\S_\lambda V$ with  $P=[p]$.
Since $F_{\lambda,n}$ is a $G$-orbit, its affine tangent space is
\[\widehat T_PF_{\lambda,n}
=
\langle p\rangle
+
d\rho_\lambda\bigl(\Lie{gl}(V)\bigr)\cdot p.
\]
Let first $\varphi
\in
\left(
I_{\mathrm S}(P)^{\langle2\rangle}
\right)_\lambda$. By definition, $\varphi*p=0$ and $(\xi\cdot\varphi)*p=0$ for every $\xi\in\Lie{gl}(V)$. Since $\varphi*p=0$,
\Cref{lem: infinitesimal equivariance Schur apolarity} gives
\[
\varphi*(\xi\cdot p)
=
-
(\xi\cdot\varphi)*p
=
0.
\]
Thus, $\varphi$ annihilates both $p$ and every vector of the form $\xi\cdot p$. By the preceding description of the tangent space, it follows that
$\varphi
\in
\left(
\widehat T_PF_{\lambda,n}
\right)^\perp$.

Conversely, let
$\varphi
\in
\left(
\widehat T_PF_{\lambda,n}
\right)^\perp$.
Since $p\in\widehat T_PF_{\lambda,n}$, one has $\varphi*p=0$. Furthermore, $\xi\cdot p
\in
\widehat T_PF_{\lambda,n}$ for every $\xi\in\Lie{gl}(V)$, and therefore $\varphi*(\xi\cdot p)=0$.
Applying
\Cref{lem: infinitesimal equivariance Schur apolarity}, we obtain
\[
(\xi\cdot\varphi)*p=0
\qquad
\text{for every }\xi\in\Lie{gl}(V).
\]
Hence $\varphi
\in
\left(
I_{\mathrm S}(P)^{\langle2\rangle}
\right)_\lambda$.
\end{proof}

Combining \Cref{cor: dual terracini} and \Cref{thm: tangential Schur apolarity} gives the form of Terracini's Lemma that will be used in the rest of the paper.

\begin{corollary}[Schur Dual Terracini Lemma]
\label{cor: schur dual terracini}
Let $P_1,\ldots,P_s
\in
F_{\lambda,n}$ be general points, and let $Q
\in
\langle P_1,\ldots,P_s\rangle$ be general. Then
\[\left(
\widehat T_Q\sigma_s(F_{\lambda,n})
\right)^\perp
=
\bigcap_{i=1}^s
\left(
I_{\mathrm S}(P_i)^{\langle2\rangle}
\right)_\lambda.
\]Consequently,
\[\dim\sigma_s(F_{\lambda,n})
=
N_{\lambda,n}
-
\dim
\left(
\bigcap_{i=1}^s
\left(
I_{\mathrm S}(P_i)^{\langle2\rangle}
\right)_\lambda
\right).\]\end{corollary}

\begin{remark}
\label{rem: geometric meaning Schur square}
Under the natural identification
$\S_\lambda V^*
\simeq
H^0
\left(
F_{\lambda,n},
\cO_{F_{\lambda,n}}(\bd)
\right)$,
let
$\cI_{P,F_{\lambda,n}}
\subseteq
\cO_{F_{\lambda,n}}$
be the ideal sheaf of a point
$P\in F_{\lambda,n}$.
The double point supported at $P$, namely the first infinitesimal
neighbourhood of $P$ in $F_{\lambda,n}$, is the closed subscheme
defined by $\cI_{P,F_{\lambda,n}}^2$. Then
\[
\left(
I_{\mathrm S}(P)^{\langle2\rangle}
\right)_\lambda
=
H^0
\left(
F_{\lambda,n},
\cO_{F_{\lambda,n}}(\bd)
\otimes
\cI_{P,F_{\lambda,n}}^2
\right).
\]
More generally, let
$P_1,\ldots,P_s\in F_{\lambda,n}$
be general points, and let
$Z\subseteq F_{\lambda,n}$
be the scheme-theoretic union of the double points supported at them.
Equivalently,
\[
\cI_{Z,F_{\lambda,n}}
=
\bigcap_{i=1}^s
\cI_{P_i,F_{\lambda,n}}^2.
\]
Then
\[
\bigcap_{i=1}^s
\left(
I_{\mathrm S}(P_i)^{\langle2\rangle}
\right)_\lambda
=
H^0
\left(
F_{\lambda,n},
\cO_{F_{\lambda,n}}(\bd)
\otimes
\cI_{Z,F_{\lambda,n}}
\right).
\]
Therefore,
\[
\dim\sigma_s(F_{\lambda,n})
=
N_{\lambda,n}
-
h^0
\left(
F_{\lambda,n},
\cO_{F_{\lambda,n}}(\bd)
\otimes
\cI_{Z,F_{\lambda,n}}
\right).
\]
Thus, the $\lambda$-component of
$I_{\mathrm S}(P)^{\langle2\rangle}$
is precisely the space of sections of
$\cO_{F_{\lambda,n}}(\bd)$
vanishing along the double point supported at $P$.
The full geometric Schur square
$I_{\mathrm S}(P)^{\langle2\rangle}$,
however, is an ideal of the dual Schur algebra
$(\S^\bullet V^*,\diamond)$,
defined intrinsically through Schur apolarity and the infinitesimal
$G$-action. In particular, it may have nonzero components in several
Schur modules occurring in $\S^\bullet V^*$.
\end{remark}

\begin{remark}[Algebraic versus geometric Schur squares]
\label{rem: algebraic versus geometric Schur squares}
By
\Cref{prop: geometric Schur square ideal} and \Cref{thm: tangential Schur apolarity},
one always has
\[\left(
I_{\mathrm S}(P)^2
\right)_\lambda
\subseteq
\left(
I_{\mathrm S}(P)^{\langle2\rangle}
\right)_\lambda
=
\left(
\widehat T_PF_{\lambda,n}
\right)^\perp.
\]
For Veronese varieties the inclusion is an equality, as shown in
\eqref{eq: Veronese triple identification}. In general it may be
strict: the example of
\Cref{prop: algebraic Schur square fails}
gives
\[
\left(
I_{\mathrm S}(P_{(4,2),3})^2
\right)_{(4,2)}
\subsetneq
\left(
I_{\mathrm S}(P_{(4,2),3})^{\langle2\rangle}
\right)_{(4,2)}.
\]
\end{remark}

The geometric Schur square gives the correct intrinsic description
of first-order vanishing, but it is not yet adapted to explicit
computations. In the next section, we introduce the slot-by-slot
algebra and study when its multigraded geometric double points recover
the component $\left(
I_{\mathrm S}(P)^{\langle2\rangle}
\right)_\lambda$.

\section{Slot-by-slot double points and the Consistency Theorem}
\label{sect: slotwise double points}

\noindent The geometric Schur square introduced in
\Cref{def: geometric Schur square} gives an intrinsic description of
first-order vanishing in the dual Schur algebra. Its definition,
however, is not immediately suited to explicit computations.

To obtain a more concrete model, we lift the problem to the product of
the Pl\"ucker spaces associated with the different column heights of
$\lambda$. Its multigraded section ring is a tensor product of
symmetric algebras, so geometric double points can be described by
ordinary multihomogeneous ideals. The Cartan projection then brings
these conditions back to the Schur component.

Throughout the remainder of the paper, if $Y$ is a projective variety
embedded in a multiprojective space and $Z\subseteq Y$ is a closed
subscheme, we denote by $I_{Z,Y}$ the multihomogeneous ideal of $Z$ in the multigraded homogeneous
coordinate ring of $Y$ induced by the chosen embedding. More precisely, for a point $P\in F_{\lambda,n}$, we shall compare
\[
\pi_\lambda
\left(
\left(
I_{P,X_\lambda}^{\langle2\rangle}
\right)_{\bd}
\right)
\subseteq
\S_\lambda V^*
\qquad
\text{with}
\qquad
\left(
I_{\mathrm S}(P)^{\langle2\rangle}
\right)_\lambda
=
\left(
\widehat T_PF_{\lambda,n}
\right)^\perp.
\]
The Consistency Theorem will characterize exactly when these two
spaces coincide.

\subsection{The multigraded slot-by-slot algebra}
\label{subsect: slot theory}

We retain the notation introduced in
\Cref{subsect: tableaux}. Thus,
$\lambda'
=
(n_k^{d_k},\ldots,n_1^{d_1})$,
$0<n_1<\cdots<n_k<n$,
$\bd=(d_1,\ldots,d_k)$.

\begin{definition}[The slot-by-slot algebra]
\label{def: slot-by-slot algebra}
We set
\[
X_\lambda
\coloneq
\prod_{i=1}^k
\P\left(
\superwedge^{n_i}V
\right)
\]
and define the \emph{slot-by-slot algebra associated with $\lambda$}
to be
\[
\cR_\lambda
\coloneq
\bigotimes_{i=1}^k
\Sym^\bullet
\left(
\superwedge^{n_i}V^*
\right).
\]
It is endowed with the product $\star$ induced by the ordinary
symmetric product in each tensor factor. The algebra $\cR_\lambda$ is naturally $\N^k$-graded. For $\mathbf m=(m_1,\ldots,m_k)\in\N^k$,
its multidegree-$\mathbf m$ component is $(\cR_\lambda)_{\mathbf m}
=
\bigotimes_{i=1}^k
\Sym^{m_i}
\left(
\superwedge^{n_i}V^*
\right)$.
Equivalently, $\cR_\lambda$ is the multigraded section ring of $X_\lambda$:
\[
\cR_\lambda
\simeq
\bigoplus_{\mathbf m\in\N^k}
H^0
\left(
X_\lambda,
\cO_{X_\lambda}(\mathbf m)
\right).
\]
\end{definition}

Via the product of the Pl\"ucker embeddings, we regard the underlying
partial flag variety of $F_{\lambda,n}$ as the incidence subvariety
of $X_\lambda$. We use the same letter
$P\in F_{\lambda,n}\subseteq X_\lambda$ for a flag and for its image in the product of the Pl\"ucker spaces. The component associated with the embedding of
$F_{\lambda,n}$ is
\[
(\cR_\lambda)_{\bd}
=
\bigotimes_{i=1}^k
\Sym^{d_i}
\left(
\superwedge^{n_i}V^*
\right).
\]
Its dual is $(\cR_\lambda^*)_{\bd}
=
\bigotimes_{i=1}^k
\Sym^{d_i}
\left(
\superwedge^{n_i}V
\right)$. With these identifications, the Cartan maps fixed in
\Cref{prop: cartan realization} take the form
\[
\iota_\lambda:
\S_\lambda V
\lhook\joinrel\longrightarrow
(\cR_\lambda^*)_{\bd},
\qquad
\pi_\lambda:
(\cR_\lambda)_{\bd}
\longrightarrow
\S_\lambda V^*.
\]

\begin{remark}[Multiplication before projection]
\label{rem: multiplication before projection}
The Cartan projection $\pi_\lambda$ is linear and
$G$-equivariant, but the Cartan projections do not, in general,
assemble into an algebra homomorphism from the slot-by-slot algebra to
the dual Schur algebra. Consequently, the two operations considered
in this paper are genuinely different:

\begin{itemize}
    \item the algebraic Schur square is formed using the product
    $\diamond$ in $\S^\bullet V^*$;
    \item the slot-by-slot square is formed using the product
    $\star$ in $\cR_\lambda$ and is projected only afterwards by
    $\pi_\lambda$.
\end{itemize}
\end{remark}

For $i\in[k]$, let
$\bepsilon_i
=
(0,\ldots,0,1,0,\ldots,0)
\in\N^k$
be the $i$-th standard multidegree. We call
\[
\Sym^\bullet
\left(
\superwedge^{n_i}V^*
\right)
\]
the \emph{$i$-th slot} of $\cR_\lambda$ and say that this slot has
\emph{length} $n_i$. If
$u\in\superwedge^{n_i}V^*,$
we write
\[
(u)_i
\coloneq
1\otimes\cdots\otimes1
\otimes u
\otimes1\otimes\cdots\otimes1
\in
(\cR_\lambda)_{\bepsilon_i},
\]
where $u$ occurs in the $i$-th tensor factor. We say that $(u)_i$ is
\emph{supported in the $i$-th slot}. In particular, for an increasing
set
$J\subseteq[n]$,
$|J|=n_i$,
we use the notation $(x_J)_i$.

\begin{definition}
\label{def: slot tableaux}
Fix a multidegree
$\mathbf m=(m_1,\ldots,m_k)\in\N^k$.
For each $i\in[k]$, an \emph{$i$-th slot tableau of degree $m_i$}
is a rectangular tableau consisting of $m_i$ columns of height $n_i$,
whose entries belong to $[n]$ and are strictly increasing down each
column. No row condition is imposed. If the column index sets of such a tableau $T_i$ are
$J_{i,1},\ldots,J_{i,m_i}$,
we associate with it the element
\[
x_{T_i}
\coloneq
x_{J_{i,1}}\cdots x_{J_{i,m_i}}
\in
\Sym^{m_i}
\left(
\superwedge^{n_i}V^*
\right).
\]
Moreover, if $T_i=\emptyset$ is the empty tableau, we set $x_\emptyset \coloneqq 1$. Since the $i$-th slot is a symmetric algebra, the order of these
columns does not affect $x_{T_i}$. To obtain a fixed representative,
we arrange columns of equal height in decreasing lexicographic order. Collecting the tableaux $T_1,\ldots,T_k$
and arranging their columns by decreasing height gives a tableau $T$
whose conjugate shape is
$(n_k^{m_k},\ldots,n_1^{m_1})$.
We call $T$ a \emph{slot tableau of multidegree $\mathbf m$}. The
associated \emph{slot monomial} is
\[
\widetilde{x}_T
\coloneq
x_{T_1}\otimes\cdots\otimes x_{T_k}
\in
(\cR_\lambda)_{\mathbf m}.
\]
\end{definition}

The monomial bases of the symmetric powers in the different tensor
factors show that
\[
\left\{
\widetilde{x}_T
\ \middle|\
T
\text{ is a slot tableau of multidegree }\mathbf m
\right\}
\]
is a basis of $(\cR_\lambda)_{\mathbf m}$. We call it the
\emph{slot-monomial basis}. We deliberately avoid the expression ``standard monomial basis'' for
this basis: standard monomial theory will enter later through the
straightening relations defining the incidence subvariety
$F_{\lambda,n}\subseteq X_\lambda$.

\begin{example}
\label{ex: slot tableau}
Let $V\simeq\C^5$, $\lambda=(4,3,1^2)$.
Then
$\lambda'=(4,2^2,1)$,
$n_1=1$, $n_2=2$, $n_3=4$, $\bd=(1,2,1)$.
Hence
\[
(\cR_\lambda)_{\bd}
=
V^*
\otimes
\Sym^2
\left(
\superwedge^2V^*
\right)
\otimes
\superwedge^4V^*.
\]
Consider the following slot tableaux:
\[
\begin{array}{c@{\qquad}c@{\qquad}ccc}
\text{slot }1
&
\text{slot }2
&
\text{slot }3
\\[2pt]
\begin{ytableau}
2
\end{ytableau}
&
\begin{ytableau}
1 & 1\\
4 & 3
\end{ytableau}
&
\begin{ytableau}
1\\
2\\
4\\
5
\end{ytableau}
&
\rightsquigarrow
&
T=
\begin{ytableau}
1 & 1 & 1 & 2\\
2 & 4 & 3\\
4\\
5
\end{ytableau}.
\end{array}
\]
The corresponding slot monomial is
\[
\widetilde{x}_T
=
x_2
\otimes
(x_1\wedge x_4)(x_1\wedge x_3)
\otimes
(x_1\wedge x_2\wedge x_4\wedge x_5).
\]
The tableau $T$ is not semistandard, since its second row decreases
from $4$ to $3$. This is allowed in the slot-by-slot algebra, where
only the columns are required to be strictly increasing.
\end{example}

When $\mathbf m=\bd$ and $T$ has shape $\lambda$, the notation
$\widetilde{x}_T$ agrees with
\Cref{def: tableau monomials}. In particular, if $S$ is semistandard,
then $\pi_\lambda
\left(
\widetilde{x}_S
\right)
=
\bx_S$.

\subsection{Contents and torus weights}
\label{subsect: torus weights}

The general weight-space notation needed here was introduced in
\Cref{def: content spaces}. We record its consequences for the
slot-by-slot algebra.

If $T$ is a slot tableau, its content $\cont(T)
=
(\nu_1,\ldots,\nu_n)$
is defined as in \Cref{def: tableau data}: the integer $\nu_a$ is the
total number of occurrences of $a$ in all the columns of $T$.

\begin{proposition}[Content compatibility]
\label{prop: slot content compatibility}
Let $\mathbf m\in\N^k
$ and let $T$ be a slot tableau of multidegree $\mathbf m$. Then
\[
\widetilde{x}_T
\in
(\cR_\lambda)_{\mathbf m}[\cont(T)].
\]
Moreover, the Cartan projection preserves contents:
\[
\pi_\lambda
\left(
(\cR_\lambda)_{\bd}[\nu]
\right)
\subseteq
\S_\lambda V^*[\nu]
\qquad
\text{for every }\nu\in\N^n.
\]
More generally, if $A\subseteq
(\cR_\lambda)_{\bd}$
is a $\mathbb T$-stable vector subspace, then
\[
\pi_\lambda
\left(
A[\nu]
\right)
=
\pi_\lambda(A)
\cap
\S_\lambda V^*[\nu].
\]
\end{proposition}

\begin{proof}
Let $t=\operatorname{diag}(t_1,\ldots,t_n)
\in\mathbb T$.
The dual action satisfies $t\cdot x_a=t_a^{-1}x_a$
and hence
\[
t\cdot x_J
=
\left(
\prod_{a\in J}t_a^{-1}
\right)x_J.
\]
It follows immediately that
$t\cdot\widetilde{x}_T
=
t_1^{-\nu_1}\cdots t_n^{-\nu_n}
\widetilde{x}_T$, $\nu=\cont(T)$,
which proves the first statement.

The map $\pi_\lambda$ is $G$-equivariant and therefore
$\mathbb T$-equivariant. It consequently preserves the content
spaces, proving the second statement.

Finally, let
$y
\in
\pi_\lambda(A)
\cap
\S_\lambda V^*[\nu]$.
Choose $x\in A$ such that
$\pi_\lambda(x)=y$.
Since $A$ is $\mathbb T$-stable, its content decomposition has the
form
$x=\sum_\eta x_\eta$,
$x_\eta\in A[\eta]$.
Applying $\pi_\lambda$ gives
$y
=
\sum_\eta
\pi_\lambda(x_\eta)$,
and each summand has content $\eta$. Since $y$ has content $\nu$,
uniqueness of the weight-space decomposition gives
$\pi_\lambda(x_\nu)=y$. 
Thus, $y\in\pi_\lambda(A[\nu])$,
which proves the reverse inclusion.
\end{proof}

\subsection{Geometric point ideals and straightening relations}
\label{subsect: point ideals and straightening}

Let $P\in F_{\lambda,n}\subseteq X_\lambda$ correspond to the flag
$W_1\subsetneq\cdots\subsetneq W_k\subsetneq V$, $\dim W_i=n_i$.

\begin{definition}
\label{def: point ideal in slot algebra}
For every $i\in[k]$, we identify $
\left(
\superwedge^{n_i}W_i
\right)^\perp
\subseteq
\superwedge^{n_i}V^*$
with its copy in
$(\cR_\lambda)_{\bepsilon_i}$. The multihomogeneous ideal of $P$ in $X_\lambda$ is
\[
I_{P,X_\lambda}
=
\left\langle
\left(
\superwedge^{n_i}W_i
\right)^\perp
\ \middle|\
i\in[k]
\right\rangle
\subseteq
\cR_\lambda.
\]
We call the elements of $
\left(
\superwedge^{n_i}W_i
\right)^\perp
\subseteq
(\cR_\lambda)_{\bepsilon_i}$
the \emph{degree-one generators in the $i$-th slot}. We denote by $I_{P,X_\lambda}^{\langle2\rangle}$
the multihomogeneous ideal of the geometric double point at
$P$.
\end{definition}

\begin{remark}
\label{rem: geom int of pilambda and straightening}
In the multigraded section ring of the product of projective spaces
$X_\lambda$, the ordinary square of the point ideal is already
saturated. {Indeed after, a multihomogeneous change of coordinates, we may assume
that $P$ is the coordinate point determined by the first coordinate
in every factor. Then $I_{P,X_\lambda}$ is generated by all the
remaining variables. Its square is a monomial ideal. A monomial not
belonging to $I_{P,X_\lambda}^2$ contains at most one of these
generators, and multiplying it by any power of the product of the
distinguished nonvanishing coordinates does not put it in the square.
Hence $I_{P,X_\lambda}^2$ is saturated with respect to the irrelevant
ideal.} Hence
$I_{P,X_\lambda}^{\langle2\rangle}
=
I_{P,X_\lambda}^{2}$.
We nevertheless use the notation
$I_{P,X_\lambda}^{\langle2\rangle}$ whenever we regard it as the
ideal of the geometric double point.

For the highest-weight point $P_{\lambda,n}$, corresponding to $
W_i^0
=
\langle e_1,\ldots,e_{n_i}\rangle$,
one has
\[
\left(
I_{P_{\lambda,n},X_\lambda}
\right)_{\bepsilon_i}
=
\left\langle
(x_J)_i
\ \middle|\
J\subseteq[n],\ |J|=n_i,\ J\neq[n_i]
\right\rangle.
\]
Therefore,
$\left(
I_{P_{\lambda,n},X_\lambda}^{\langle2\rangle}
\right)_{\bd}$
is spanned by the slot monomials containing at least two
non-highest-weight columns, counted with multiplicity.

Finally, the Cartan projection is the restriction map
\[
\pi_\lambda:
H^0
\left(
X_\lambda,\cO_{X_\lambda}(\bd)
\right)
\longrightarrow
H^0
\left(
F_{\lambda,n},\cO_{F_{\lambda,n}}(\bd)
\right),
\]
and therefore
\[
\ker(\pi_\lambda)
=
\left(
I_{F_{\lambda,n},X_\lambda}
\right)_{\bd}.
\]
\end{remark}

We recall the quadratic relations generating the ideal of
$F_{\lambda,n}$ in $X_\lambda$.

\begin{definition}
\label{def: straightening relations}
Let $1\leq r\leq s\leq k$,
and let $A=(a_1,\ldots,a_{n_r-1})$,
$B=(b_1,\ldots,b_{n_s+1})$
be increasing sequences in $[n]$. We set
\[
R_{A,B}^{r,s}
\coloneq
\sum_{t=1}^{n_s+1}
(-1)^{t-1}
\bigl(
x_{a_1,\ldots,a_{n_r-1},b_t}
\bigr)_r
\star
\bigl(
x_{b_1,\ldots,\widehat{b_t},\ldots,b_{n_s+1}}
\bigr)_s,
\]
where $\widehat{b_t}$ means that $b_t$ is omitted. As usual, an
exterior monomial containing a repeated index is understood to be
zero.
\end{definition}

\begin{proposition}
\label{prop: straightening Cartan kernel}
The multihomogeneous ideal
$I_{F_{\lambda,n},X_\lambda}
\subseteq
\cR_\lambda$
is generated by the relations
$R_{A,B}^{r,s}$, $1\leq r\leq s\leq k$.
Consequently,
\[
\ker(\pi_\lambda)
=
\operatorname{span}
\left\{
M\star R_{A,B}^{r,s}
\ \middle|\
M\in
(\cR_\lambda)_{\bd-\bepsilon_r-\bepsilon_s}
\right\},
\]
where all admissible choices of $r,s,A,B$ are understood, and a
multigraded component with a negative entry is zero. For $r=s$, these are the usual Pl\"ucker relations; for $r<s$, they
express the incidence relations between the corresponding
Grassmannian factors.
\end{proposition}

\begin{proof}
This is the classical straightening description of the ideal of a
partial flag variety in the product of its Pl\"ucker spaces; see
\cite[Sections~8.4 and~9.1]{Fulton1997}.
\end{proof}

\subsection{The Consistency Theorem}
\label{subsect: consistency}

We now characterize when the slot-by-slot geometric double point
recovers the geometric Schur square in the component corresponding to
the embedding.

\begin{definition}[Consistent embedding]
\label{def: consistent pair}
We say that the embedded flag variety
$F_{\lambda,n}
\subseteq
\P(\S_\lambda V)$
is \emph{consistent} if, for every $i\in[k]$ such that $d_i=1$, one
has $\min
\left\{
n_i-n_{i-1},
n_{i+1}-n_i
\right\}
\leq1$,
where
$n_0=0$,
$n_{k+1}=n$.
\end{definition}

\begin{remark}
\label{rem: combinatorial meaning consistency}
In terms of the Young diagram of $\lambda$, consistency means that no
column occurring with multiplicity one is separated by gaps of size at
least two from both the preceding and the following column heights.
Here the formal boundary heights $0$ and $n$ are also taken into
account.
\end{remark}

The following lemma provides the obstruction that appears when
consistency fails.

\begin{lemma}[The isolated-column obstruction]
\label{lem: isolated column obstruction}
Assume that $F_{\lambda,n}$ is not consistent. Then there exists an
index $i\in[k]$ such that $d_i=1$, $n_i-n_{i-1}\geq2$, $n_{i+1}-n_i\geq2$.
Set $r\coloneq n_i$.
Then $2\leq r\leq n-2$.

Let $S_{\mathrm{bad}}$ be the tableau obtained from the highest-weight
tableau $S_0$ by replacing its unique column of height $r$, whose
content is $[r]$, with the column of content $[r-2]\cup\{r+1,r+2\}$,
and leaving all other columns unchanged. Set $\mu
\coloneq
\cont(S_{\mathrm{bad}})$. Then, $S_{\mathrm{bad}}$ is semistandard and $\dist(S_{\mathrm{bad}},S_0)=2$.
Moreover, every tableau of shape $\lambda$ with strictly increasing
columns, no row condition, and content $\mu$ has the same column
contents as $S_{\mathrm{bad}}$. As a consequence, every slot monomial in
$(\cR_\lambda)_{\bd}$ with content $\mu$ has exactly one
non-highest-weight column.
\end{lemma}

\begin{proof}
The inequalities $n_i-n_{i-1}\geq2$, $n_{i+1}-n_i\geq2$ imply that $\lambda$ has no columns of heights $r-1$ and $r+1$.
Therefore, every column to the left of the unique column of height
$r$ has height at least $r+2$, whereas every column to its right has
height at most $r-2$. In rows $r-1$ and $r$, every column to the left contains the entries
$r-1$ and $r$, whereas the modified column contains $r+1$ and
$r+2$. Since the columns to the right have height at most $r-2$,
there are no further entries in these rows. Thus, the row condition
is satisfied. The modified column is strictly increasing because its
content is $[r-2]\cup\{r+1,r+2\}$.
Hence $S_{\mathrm{bad}}$ is semistandard. Its modified column differs
from $[r]$ precisely by replacing $r-1$ and $r$ with $r+1$ and
$r+2$, and therefore $\dist(S_{\mathrm{bad}},S_0)=2$.

Let $R$ be a tableau of shape $\lambda$ with strictly increasing
columns and content $\mu$. For $q\in[n]$, let $N_R(q)$ denote the
number of entries of $R$ belonging to $[q]$. A column of height $h$
contains at most $\min\{h,q\}$ entries in $[q]$. Hence
\begin{equation}
\label{eq: prefix content bound}
N_R(q)
\leq
\sum_{c=1}^{\lambda_1}
\min\{\lambda'_c,q\}.
\end{equation}
The right-hand side is $N_{S_0}(q)$. Relative to $S_0$, the content $\mu$ has one fewer occurrence of
$r-1$ and $r$, and one more occurrence of $r+1$ and $r+2$.
Hence, since $\cont(R)=\mu$, equality holds in
\eqref{eq: prefix content bound} for every $q\leq r-2$ and $q\geq r+2$. The inequality in \eqref{eq: prefix content bound} is obtained by
adding the corresponding inequalities for the individual columns.
Therefore, whenever equality holds, each column of height $h$ contains
exactly $\min\{h,q\}$ entries in $[q]$.

If $h\leq r-2$, take $q=h$. A column of height $h$ then contains
exactly $h$ entries in $[h]$, and hence its content is $[h]$.
Similarly, if $h\geq r+2$, taking $q=h$ shows that every column of
height $h$ has content $[h]$. The unique column of height $r$ contains $[r-2]$. For $r>2$, this
follows by taking $q=r-2$; for $r=2$, it follows from the convention
$[0]=\emptyset$. Since there are no columns of heights $r-1$ and
$r+1$, all other columns have now been forced to be highest-weight
columns. The total content $\mu$ then forces the remaining two
entries of the column of height $r$ to be $r+1$ and $r+2$. Thus, $R$ has the same column contents as $S_{\mathrm{bad}}$. 

Since slot monomials in $(\cR_\lambda)_{\bd}$ are indexed by tableaux of
shape $\lambda$ with strictly increasing columns and no row condition,
the final assertion follows.
\end{proof}

\begin{theorem}[Consistency Theorem]
\label{thm: consistency}
Let $P\in F_{\lambda,n}\subseteq X_\lambda$.
Then
\[
\pi_\lambda
\left(
\left(
I_{P,X_\lambda}^{\langle2\rangle}
\right)_{\bd}
\right)
\subseteq
\left(
I_{\mathrm S}(P)^{\langle2\rangle}
\right)_\lambda
=
\left(
\widehat T_PF_{\lambda,n}
\right)^\perp.
\]
Moreover, equality holds if and only if the embedded flag variety
$F_{\lambda,n}$ is consistent.
\end{theorem}

\begin{proof}
By homogeneity, it is enough to prove the statement at the
highest-weight point $P_{\lambda,n}
=
[\bv_{\lambda,n}]$.
Write $W_i^0
=
\langle e_1,\ldots,e_{n_i}\rangle$,
so that $P_{\lambda,n}$ corresponds to the standard flag $W_1^0
\subsetneq
\cdots
\subsetneq
W_k^0$.

We first prove the inclusion
\[
\pi_\lambda
\left(
\left(
I_{P_{\lambda,n},X_\lambda}^{\langle2\rangle}
\right)_{\bd}
\right)
\subseteq
\left(
\widehat T_{P_{\lambda,n}}F_{\lambda,n}
\right)^\perp.
\]
The ideal $I_{P_{\lambda,n},X_\lambda}^{\langle2\rangle}$
is the ideal of the geometric double point supported at
$P_{\lambda,n}$ in $X_\lambda$. Hence its elements vanish at
$P_{\lambda,n}$ together with their first derivatives. Their
restrictions to $F_{\lambda,n}$ therefore vanish to first order at
$P_{\lambda,n}$. Since $\pi_\lambda$ is the restriction map in
multidegree $\bd$, the asserted inclusion follows from
\Cref{thm: tangential Schur apolarity} and \Cref{rem: geometric meaning Schur square}.
Assume now that $F_{\lambda,n}$ is consistent. By
\Cref{prop: tangent and conormal bases}, it is enough to consider a
semistandard tableau $S$ of shape $\lambda$ satisfying $\dist(S,S_0)\geq2$
and prove that
\[
\bx_S
\in
\pi_\lambda
\left(
\left(
I_{P_{\lambda,n},X_\lambda}^{\langle2\rangle}
\right)_{\bd}
\right).
\]
Suppose first that the non-highest-weight entries of $S$ occur in at
least two distinct columns. Let these columns have contents $J_1$ and
$J_2$, and heights $n_i$ and $n_j$, respectively. We allow $i=j$,
provided that the two columns are distinct.

Since both columns are non-highest-weight,
$(x_{J_1})_i
\in
I_{P_{\lambda,n},X_\lambda}$,
$(x_{J_2})_j
\in
I_{P_{\lambda,n},X_\lambda}$.
Let
$M
\in
(\cR_\lambda)_{\bd-\bepsilon_i-\bepsilon_j}
$
be the slot monomial determined by all columns of $S$ except the two
selected ones, and set
\[
\beta
\coloneq
M\star(x_{J_1})_i\star(x_{J_2})_j.
\]
Then
$\beta
\in
\left(
I_{P_{\lambda,n},X_\lambda}^{\langle2\rangle}
\right)_{\bd}$.
Moreover, $\beta=\widetilde{x}_S$, and therefore $\pi_\lambda(\beta)=\bx_S$.

It remains to consider the case in which all non-highest-weight
entries of $S$ occur in a single column of height $n_i$. Let $J$ be
the content of this column. Since $\dist(S,S_0)\geq2$,
the set $J\setminus[n_i]$
contains at least two elements.

Suppose first that $d_i\geq2$. Then $S$ contains another column of
height $n_i$. Since all non-highest-weight entries occur in the column
with content $J$, this second column has content $[n_i]$. Choose $a\in J\setminus[n_i]$. For every $b\in[n_i]\setminus J$, set $J_b
\coloneq
(J\setminus\{a\})\cup\{b\}$, $K_b
\coloneq
([n_i]\setminus\{b\})\cup\{a\}$,
where the corresponding index sets are written in increasing order. Since $J\setminus[n_i]$ contains at least two elements, one has $J_b\neq[n_i]$. Moreover, $K_b\neq[n_i]$
because $a>n_i$. Hence
$(x_{J_b})_i,
(x_{K_b})_i
\in
I_{P_{\lambda,n},X_\lambda}$. Set $A
\coloneq
J\setminus\{a\}$, $B
\coloneq
[n_i]\cup\{a\}$.
After reordering the indices in the exterior factors, the
straightening relation
$R_{A,B}^{i,i}$ of
\Cref{def: straightening relations} has the form
\[
R_{A,B}^{i,i}
=
\varepsilon
(x_J)_i\star(x_{[n_i]})_i
+
\sum_{b\in[n_i]\setminus J}
\varepsilon_b
(x_{J_b})_i\star(x_{K_b})_i,
\]
where $\varepsilon,\varepsilon_b\in\{\pm1\}$.
The terms indexed by $b\in[n_i]\cap J$ vanish because their first
exterior factor contains a repeated index. Let $M'
\in
(\cR_\lambda)_{\bd-2\bepsilon_i}$
be the slot monomial determined by all columns of $S$ except the
column with content $J$ and one highest-weight column of height
$n_i$. By
\Cref{prop: straightening Cartan kernel}, $M'\star R_{A,B}^{i,i}
\in
\ker(\pi_\lambda)$.
Since $M'\star(x_J)_i\star(x_{[n_i]})_i
=
\widetilde{x}_S$,
applying $\pi_\lambda$ gives
\[
0
=
\varepsilon\bx_S
+
\pi_\lambda
\left(
M'
\star
\sum_{b\in[n_i]\setminus J}
\varepsilon_b
(x_{J_b})_i\star(x_{K_b})_i
\right).
\]
Set
\[
\beta'
\coloneq
-\varepsilon
M'
\star
\sum_{b\in[n_i]\setminus J}
\varepsilon_b
(x_{J_b})_i\star(x_{K_b})_i.
\]
Each summand in $\beta'$ is a product of two degree-one generators of
$I_{P_{\lambda,n},X_\lambda}$. Therefore,
$\beta'
\in
\left(
I_{P_{\lambda,n},X_\lambda}^{\langle2\rangle}
\right)_{\bd}$, $\pi_\lambda(\beta')=\bx_S$.

We now assume that $d_i=1$. Since $F_{\lambda,n}$ is consistent, one
has $n_i-n_{i-1}\leq1$ or $n_{i+1}-n_i\leq1$. 

Assume first that
$n_i-n_{i-1}\leq1$.
If $i=1$, then $n_i=1$, and a column of height $n_i$ cannot contain
two distinct non-highest-weight entries. If $i>1$, then
$n_{i-1}=n_i-1$.
A column of height $n_{i-1}$ is immediately to the right of the
column with content $J$. Since $J\setminus[n_i]$ contains at least
two elements, the entry in row $n_i-1$ of the column with content
$J$ is strictly greater than $n_i-1$. By the row condition, the
entry in the same row of the adjacent shorter column is at least as
large. Hence the shorter column is also non-highest-weight,
contradicting the assumption that all non-highest-weight entries of
$S$ occur in the column with content $J$. Thus, this case cannot
occur.

It remains to consider $n_{i+1}-n_i\leq1$.
If $i=k$, then $n-n_i=1$,
so there is only one integer strictly greater than $n_i$. Therefore,
a column of height $n_i$ cannot contain two distinct
non-highest-weight entries.

Suppose that $i<k$. Then
$n_{i+1}=n_i+1$.
Since $J\setminus[n_i]$ contains at least two elements and only one of
them can equal $n_i+1$, we may choose
$a\in J\setminus[n_i]$ with
$a>n_{i+1}$.
The adjacent column of height $n_{i+1}$ has content $[n_{i+1}]$,
because all non-highest-weight entries of $S$ occur in the column
with content $J$. For every $b\in[n_{i+1}]\setminus J$,
set
$J_b
\coloneq
(J\setminus\{a\})\cup\{b\}$, $K_b
\coloneq
([n_{i+1}]\setminus\{b\})\cup\{a\}$.
The set $J_b$ is different from $[n_i]$, since
$J\setminus[n_i]$ contains a non-highest-weight entry different from
$a$. Similarly,
$K_b\neq[n_{i+1}]
$
because it contains $a>n_{i+1}$. Thus,
$(x_{J_b})_i,
(x_{K_b})_{i+1}
\in
I_{P_{\lambda,n},X_\lambda}$. Set $A'
\coloneq
J\setminus\{a\}$, $B'
\coloneq
[n_{i+1}]\cup\{a\}$.
After reordering the indices in the exterior factors, the
straightening relation
$R_{A',B'}^{i,i+1}$ has the form
\[
R_{A',B'}^{i,i+1}
=
\varepsilon
(x_J)_i\star(x_{[n_{i+1}]})_{i+1}
+
\sum_{b\in[n_{i+1}]\setminus J}
\varepsilon_b
(x_{J_b})_i\star(x_{K_b})_{i+1},
\]
where $\varepsilon,\varepsilon_b\in\{\pm1\}$. Let $M''
\in
(\cR_\lambda)_{\bd-\bepsilon_i-\bepsilon_{i+1}}$
be the slot monomial determined by all columns of $S$ except the
column with content $J$ and the highest-weight column of height
$n_{i+1}$. By
\Cref{prop: straightening Cartan kernel},
 $M''\star R_{A',B'}^{i,i+1}
\in
\ker(\pi_\lambda)$.
Since
$
M''
\star
(x_J)_i
\star
(x_{[n_{i+1}]})_{i+1}
=
\widetilde{x}_S$,
applying $\pi_\lambda$ gives
\[
0
=
\varepsilon\bx_S
+
\pi_\lambda
\left(
M''
\star
\sum_{b\in[n_{i+1}]\setminus J}
\varepsilon_b
(x_{J_b})_i\star(x_{K_b})_{i+1}
\right).
\] 
Set
\[
\beta''
\coloneq
-\varepsilon
M''
\star
\sum_{b\in[n_{i+1}]\setminus J}
\varepsilon_b
(x_{J_b})_i\star(x_{K_b})_{i+1}.
\]
Every summand in $\beta''$ contains one degree-one generator of
$I_{P_{\lambda,n},X_\lambda}$ in each of the two indicated slots.
Therefore,
\[
\beta''
\in
\left(
I_{P_{\lambda,n},X_\lambda}^{\langle2\rangle}
\right)_{\bd},
\qquad
\pi_\lambda(\beta'')=\bx_S.
\]
This proves the reverse inclusion when $F_{\lambda,n}$ is consistent.

It remains to prove that equality fails when $F_{\lambda,n}$ is not
consistent. Let $S_{\mathrm{bad}}$ and $\mu$ be as in
\Cref{lem: isolated column obstruction}. Since
$\dist(S_{\mathrm{bad}},S_0)=2$,
\Cref{prop: tangent and conormal bases} gives
$
\bx_{S_{\mathrm{bad}}}
\in
\left(
\widehat T_{P_{\lambda,n}}F_{\lambda,n}
\right)^\perp$. The point $P_{\lambda,n}$ is fixed by the diagonal torus, so $\left(
I_{P_{\lambda,n},X_\lambda}^{\langle2\rangle}
\right)_{\bd}$
is torus-stable. By
\Cref{rem: geom int of pilambda and straightening}, this space is
spanned by slot monomials containing at least two
non-highest-weight columns. By
\Cref{lem: isolated column obstruction}, no such monomial has content
$\mu$. Hence
\[
\left(
I_{P_{\lambda,n},X_\lambda}^{\langle2\rangle}
\right)_{\bd}[\mu]
=
0.
\]
On the other hand,
$\widetilde{x}_{S_{\mathrm{bad}}}
\in
(\cR_\lambda)_{\bd}[\mu]$,
and
$
\bx_{S_{\mathrm{bad}}}
=
\pi_\lambda
\left(
\widetilde{x}_{S_{\mathrm{bad}}}
\right)$
is a nonzero element of $\S_\lambda V^*[\mu]$. If $\bx_{S_{\mathrm{bad}}}$ belonged to $\pi_\lambda
\left(
\left(
I_{P_{\lambda,n},X_\lambda}^{\langle2\rangle}
\right)_{\bd}
\right)$,
then
\Cref{prop: slot content compatibility}
would give a preimage in
$
\left(
I_{P_{\lambda,n},X_\lambda}^{\langle2\rangle}
\right)_{\bd}[\mu]$,
which is zero. This is a contradiction. Therefore,
\[
\bx_{S_{\mathrm{bad}}}
\in
\left(
\widehat T_{P_{\lambda,n}}F_{\lambda,n}
\right)^\perp
\setminus
\pi_\lambda
\left(
\left(
I_{P_{\lambda,n},X_\lambda}^{\langle2\rangle}
\right)_{\bd}
\right).
\]
Thus equality fails at $P_{\lambda,n}$ and, by homogeneity, at every
point of $F_{\lambda,n}$.
\end{proof}

\begin{corollary}[The two-step flag variety]
\label{cor: consistency two step flag}
Let $n\geq3$ and let $P\in F_{(2,1),n}$.
Then
\[
\pi_{(2,1)}
\left(
\left(
I_{P,X_{(2,1)}}^{\langle2\rangle}
\right)_{(1,1)}
\right)
=
\left(
I_{\mathrm S}(P)^{\langle2\rangle}
\right)_{(2,1)}
=
\left(
\widehat T_PF_{(2,1),n}
\right)^\perp.
\]
\end{corollary}

\begin{proof}
For $\lambda=(2,1)$, one has $\lambda'=(2,1)$, $n_1=1$, $n_2=2$, $d_1=d_2=1$.
The column of height $1$ differs from the boundary height $0$ by one,
whereas the column of height $2$ differs from the preceding height
$1$ by one. Hence $F_{(2,1),n}$ is consistent, and the conclusion
follows from \Cref{thm: consistency}.
\end{proof}

\begin{remark}
\label{rem: computational role slotwise lift}
For every consistent embedding, the slotwise construction should be
viewed not as a replacement for Schur apolarity, but as a
multigraded realization of the $\lambda$-component of the geometric
Schur square:
\[
\pi_\lambda
\left(
\left(
I_{P,X_\lambda}^{\langle2\rangle}
\right)_{\bd}
\right)
=
\left(
I_{\mathrm S}(P)^{\langle2\rangle}
\right)_\lambda.
\]
Indeed, the point ideal in the slotwise algebra is generated by
linear equations in the individual slots, and the ideal of the
corresponding double point is computed by ordinary multihomogeneous
multiplication. In the family $F_{(2,1),n}$ considered below,
restrictions to the subflags entering the Horace argument are
likewise expressed by quotients of the same multigraded ring.
Thus explicit linear algebra and Horace-type arguments become
available without losing their Schur-apolar interpretation.
\end{remark}

\section{Low-dimensional defective flag varieties}
\label{sect: defectiveness}

\noindent We now apply the slot-by-slot realization to some low-dimensional
flag varieties. Besides providing the exceptional cases needed in the
classification of the secant varieties of $F_{(2,1),n}$, these
examples illustrate how the geometric Schur square can be used to
produce explicit equations for defective secant varieties.

\subsection{Two defective flags in dimension four}
We begin with the $\cO(1,1)$-embedding of $\Fl(1,2;V)$ for $\dim V=4$.

\begin{theorem}
\label{thm: def of F124}
Let $V\simeq\C^4$.
Then
$
F_{(2,1),4}
\simeq
\Fl(1,2;V)
\subseteq
\P(\S_{(2,1)}V)
\simeq
\P^{19}$
is defective. More precisely, $\sigma_2(F_{(2,1),4})$ has the
expected dimension, $\sigma_4(F_{(2,1),4})$ fills the ambient space,
and
$\dim\sigma_3(F_{(2,1),4})=16$.
In particular,
$
\delta_3(F_{(2,1),4})=1$.
The same holds for $F_{(2,2,1),4}$.
\end{theorem}

\begin{proof}
By the standard duality
$\Fl(2,3;V)
\simeq
\Fl(1,2;V^*)$
and the isomorphism
$\S_{(2,2,1)}V
\simeq
\S_{(2,1)}V^*
\otimes
(\det V)^2$,
the embeddings of $F_{(2,2,1),4}$ and $F_{(2,1),4}$ are
projectively equivalent. Hence, it is enough to prove the statement
for $F_{(2,1),4}$.

The expected dimension of $\sigma_3(F_{(2,1),4})$ is $17$.
By the Schur Dual Terracini Lemma
(\Cref{cor: schur dual terracini}), for general points
$P_1,P_2,P_3\in F_{(2,1),4}$ one has
\[
\dim\sigma_3(F_{(2,1),4})
=
19-
\dim
\left(
\bigcap_{i=1}^{3}
\left(
I_{\mathrm S}(P_i)^{\langle2\rangle}
\right)_{(2,1)}
\right).
\]
Therefore, it is enough for the upper bound to exhibit a
three-dimensional subspace contained in $\bigcap_{i=1}^{3}
\left(
I_{\mathrm S}(P_i)^{\langle2\rangle}
\right)_{(2,1)}$.
A direct computation with a computer algebra software
(which can be checked by running \path{computations/5.1_theorem.m2} in \cite{Repo})
gives the lower bound $\dim\sigma_3(F_{(2,1),4})\geq16$.
Together with the three-dimensional subspace constructed below, this
implies $\dim\sigma_3(F_{(2,1),4})=16$
and hence $\delta_3(F_{(2,1),4})=1$.
The same computation verifies the assertions concerning
$\sigma_2(F_{(2,1),4})$ and $\sigma_4(F_{(2,1),4})$.

Let $P_i=(\ell_i\subset\Pi_i)
\in
F_{(2,1),4}$, $i=1,2,3$,
be general points. Choose vectors $v_i,w_i\in V$ such that $\ell_i=\langle v_i\rangle$,
$\Pi_i=\langle v_i,w_i\rangle$.
Since the points are general, the vectors $v_1,v_2,v_3$
are linearly independent. Complete them to a basis $\{v_1,v_2,v_3,v_4\}$
 of $V$, and let $\alpha\in V^*$
be the dual vector of $v_4$. Thus, $\alpha(v_i)=0$ for $i=1,2,3$. Set
\[
U
\coloneq
\left\langle
v_i\wedge w_i
\ \middle|\
i=1,2,3
\right\rangle
\subseteq
\superwedge^2V.
\]
Since the points $P_1,P_2,P_3$ are general, the vectors
$v_1\wedge w_1$, $v_2\wedge w_2$,
$v_3\wedge w_3$
are linearly independent, and hence
$\dim U=3$.
Define
\[
W
\coloneq
U^\perp
=
\left\{
\beta\in\superwedge^2V^*
\ \middle|\
\beta(v_i\wedge w_i)=0
\text{ for }i=1,2,3
\right\}.
\]
Since $\dim\superwedge^2V=6$,
it follows that
$\dim W=3$. Consider now the three-dimensional subspace
\[
K
\coloneq
\langle\alpha\rangle\otimes W
\subseteq
(\cR_{(2,1)})_{(1,1)}
=
V^*\otimes\superwedge^2V^*.
\]
For every $\beta\in W$ and every $i=1,2,3$, the linear forms
supported in the two slots satisfy
$(\alpha)_1
\in
I_{P_i,X_{(2,1)}}$, $(\beta)_2
\in
I_{P_i,X_{(2,1)}}$.
Indeed, $\alpha(v_i)=0$, $\beta(v_i\wedge w_i)=0$.
Therefore,
\[
(\alpha)_1\star(\beta)_2
\in
I_{P_i,X_{(2,1)}}^2
=
I_{P_i,X_{(2,1)}}^{\langle2\rangle}.
\]
Since every element of $K$ has this form, we obtain $K
\subseteq
\bigcap_{i=1}^3
\left(
I_{P_i,X_{(2,1)}}^{\langle2\rangle}
\right)_{(1,1)}$.
By
\Cref{cor: consistency two step flag},
\[
\pi_{(2,1)}(K)
\subseteq
\bigcap_{i=1}^3
\left(
I_{\mathrm S}(P_i)^{\langle2\rangle}
\right)_{(2,1)}.
\]

It remains to prove that the restriction of
$\pi_{(2,1)}$
to $K$ is injective. By the Littlewood--Richardson decomposition,
\[
V^*\otimes\superwedge^2V^*
\simeq
\S_{(2,1)}V^*
\oplus
\superwedge^3V^*,
\]
and hence $\ker(\pi_{(2,1)})
=
\superwedge^3V^*$. Every nonzero element of $K$ is a simple tensor
$\alpha\otimes\beta
\in
V^*\otimes\superwedge^2V^*$
and therefore has tensor rank one. On the other hand, since
$\dim V=4$, every nonzero element of $\superwedge^3V^*$ is
decomposable. Under the natural inclusion
\[
\superwedge^3V^*
\lhook\joinrel\longrightarrow
V^*\otimes\superwedge^2V^*,
\]
the element $u_1\wedge u_2\wedge u_3$
is mapped to
$u_1\otimes(u_2\wedge u_3)
-
u_2\otimes(u_1\wedge u_3)
+
u_3\otimes(u_1\wedge u_2)$,
which has tensor rank three. Consequently,
$K\cap\ker(\pi_{(2,1)})
=
\{0\}$.
Thus, $\pi_{(2,1)}|_K$ is injective and $\dim\pi_{(2,1)}(K)=3$.
The desired upper bound follows from the Schur Dual Terracini Lemma, and the
computational lower bound recorded above concludes the proof.
\end{proof}

\subsection{A defective flag in dimension five}

We next consider the $\cO(1,1)$-embedding of
$\Fl(2,3;V)$ for $\dim V=5$.

\begin{lemma}\label{lem: tensorrank 221}
Let $V\simeq\C^5$, let $\lambda=(2,2,1)$, and consider the projection
\[
\pi_\lambda:(\cR_\lambda)_{(1,1)}=\superwedge^2V^*\otimes\superwedge^3V^*\twoheadrightarrow\S_\lambda V^*.
\]
Then every nonzero element of $\ker(\pi_\lambda)$ has tensor rank at least $3$.
\end{lemma}

\begin{proof}
Fix a volume form $\omega\in \superwedge^5V\simeq\C$. It induces an isomorphism $\superwedge^3V^*\simeq\superwedge^2V$, and hence an identification
\[
\superwedge^2V^*\otimes\superwedge^3V^*\simeq\superwedge^2V^*\otimes\superwedge^2V\simeq\operatorname{End}(\superwedge^2V).
\]
Under this identification, tensor rank coincides with the rank of the corresponding endomorphism of $\superwedge^2V$. By the Littlewood--Richardson rule, the source decomposes as
\[
\superwedge^2V^*\otimes\superwedge^3V^*=\S_{(2,2,1)}V^*\oplus\S_{(2,1,1,1)}V^*\oplus\superwedge^5V^*,
\]
where the three summands have dimensions $75$, $24$, and $1$, respectively. Therefore,
\[
\ker(\pi_\lambda)=\S_{(2,1,1,1)}V^*\oplus\superwedge^5V^*.
\]
The chosen volume form makes the above identification
$\SL(V)$-equivariant. Consider the $\SL(V)$-equivariant map
\[
\Lie{sl}(V)
\longrightarrow
\operatorname{End}(\superwedge^2V),
\qquad
A\longmapsto D_A,
\]
where
\[
D_A(x\wedge y)
=
Ax\wedge y+x\wedge Ay.
\]
This map is nonzero and hence injective, since
$\Lie{sl}(V)$ is simple. Its image has dimension $24$ and therefore
coincides with the summand $\S_{(2,1,1,1)}V^*$.
Moreover,
\[
D_{t\operatorname{Id}_V}
=
2t\operatorname{Id}_{\superwedge^2V},
\]
so that the scalar endomorphisms correspond to the summand
$\superwedge^5V^*$. Consequently,
\[
\ker(\pi_\lambda)
=
\left\{
D_B
\ \middle|\
B\in\operatorname{End}(V)
\right\}.
\]
Thus, it is enough to prove that every nonzero endomorphism of the
form $D_B$ has rank at least $3$.

Let $\mu_1,\dots,\mu_5$ be the eigenvalues of $B$, counted with algebraic multiplicity. After triangularizing $B$, one sees that the eigenvalues of $D_B$
are the ten numbers $\mu_i+\mu_j$, for $1\leq i<j\leq5$. Suppose first that at least one of the $\mu_i$ is nonzero. We claim that at least three of the sums $\mu_i+\mu_j$, with $1\leq i<j\leq5$, are nonzero. Otherwise, at least eight of these ten sums would vanish. Since each index occurs in four pairs, there would exist an index $i$ such that $\mu_i+\mu_j=0$ for every $j\neq i$. Hence, $\mu_j=-\mu_i$ for every $j\neq i$. For any two distinct indices $j,k\neq i$, we would then have $\mu_j+\mu_k=-2\mu_i$. On the other hand, among the six pairs not containing $i$, at least four must have vanishing sum. Therefore, $-2\mu_i=0$, and hence $\mu_i=0$. It follows that $\mu_j=0$ for every $j$, contradicting our assumption. Thus, at least three eigenvalues of $D_B$ are nonzero, counted with algebraic multiplicity, and therefore $\rank(D_B)\geq3$.

It remains to consider the case in which all the eigenvalues of $B$ are zero. Then $B$ is nilpotent. Since $D_B$ is nonzero, we have $B\neq0$. Choose $e_2\in V$ such that $e_1:=Be_2\neq0$. Since $B$ is nilpotent, the vectors $e_1$ and $e_2$ are linearly independent. Choose a three-dimensional subspace $W\subset V$ such that $V=\langle e_1,e_2\rangle\oplus W$.
For every $f\in W$, we have
\[
D_B(e_2\wedge f)=e_1\wedge f+e_2\wedge Bf.
\]
We claim that the restriction of $D_B$ to $e_2\wedge W$ is injective. Indeed, suppose that $D_B(e_2\wedge f)=0$ for some $f\in W$. Passing to the quotient $\superwedge^2V/(e_2\wedge V)$ gives $e_1\wedge f=0$ modulo $e_2\wedge V$. This means that $e_1\wedge f$ belongs to $e_2\wedge V$, and hence $e_1,e_2,f$ are linearly dependent. Since $f\in W$ and $V=\langle e_1,e_2\rangle\oplus W$, it follows that $f=0$. Therefore, $D_B|_{e_2\wedge W}$ is injective. Since $\dim(e_2\wedge W)=3$, we conclude that $\rank(D_B)\geq3$.

The analysis above shows that every nonzero element of $\ker(\pi_\lambda)$ has operator rank, and hence tensor rank, at least $3$, concluding the proof.
\end{proof}

\begin{theorem}\label{thm: def of F235}
    Let $V\simeq\C^5$ and $\lambda=(2,2,1)$. $F_{(2,2,1),5}$ is defective. In particular, $\sigma_s(F_{(2,2,1),5})$ has the expected dimension for $s\le7$, $\sigma_{10}(F_{(2,2,1),5})$ fills the ambient space $\P(\S_{(2,2,1)} V)\simeq\P^{74}$, and $\sigma_8(F_{(2,2,1),5})$ and
$\sigma_9(F_{(2,2,1),5})$ both have defect $1$.
\end{theorem}

\begin{proof}
A direct computation at suitable explicit configurations (which can be checked by running \path{computations/5.3_theorem.m2} in \cite{Repo}) shows, by
semicontinuity, that $\sigma_s(F_{(2,2,1),5})$ has the expected
dimension for $s\leq 7$, that
$\sigma_{10}(F_{(2,2,1),5})$ fills the ambient space, and that
\[
\dim \sigma_8(F_{(2,2,1),5})\geq 70
\qquad\text{and}\qquad
\dim \sigma_9(F_{(2,2,1),5})\geq 73.
\] Since $\dim F_{(2,2,1),5}=8$ and $\P(\S_\lambda V)\simeq\P^{74}$, $\expdim \sigma_8(F_{(2,2,1),5})=71$ and $\expdim \sigma_9(F_{(2,2,1),5})=74$. It is worth noticing that
$\sigma_9(F_{(2,2,1),5})$ is expected to fill the ambient space
and is superabundant by $6$, so its defectivity does not follow
automatically from the defectivity of
$\sigma_8(F_{(2,2,1),5})$.

We first prove that $\dim \sigma_8(F_{(2,2,1),5})\leq 70$. Let $P_i=(\pi_i\subset H_i)\in F_{(2,2,1),5}$, for $i=1,\dots,8$, be general points. Choose vectors $v_i,w_i,z_i\in V$ such that $\pi_i=\langle v_i,w_i\rangle$ and $H_i=\langle v_i,w_i,z_i\rangle$. Set
\[
U_2:=\left\langle v_i\wedge w_i\mid i=1,\dots,8\right\rangle\subseteq\superwedge^2V,\qquad
U_3:=\left\langle v_i\wedge w_i\wedge z_i\mid i=1,\dots,8\right\rangle\subseteq\superwedge^3V.
\]
Since the points $P_1,\dots,P_8$ are general, the eight vectors $v_i\wedge w_i$ are linearly independent in $\superwedge^2V$, and the eight vectors $v_i\wedge w_i\wedge z_i$ are linearly independent in $\superwedge^3V$. Hence, $\dim U_2=\dim U_3=8$. Define
\[
W_2:=U_2^\perp\subseteq\superwedge^2V^*,\qquad W_3:=U_3^\perp\subseteq\superwedge^3V^*.
\]
Since $\dim\superwedge^2V=\dim\superwedge^3V=10$, we have $\dim W_2=\dim W_3=2$. Consider the $4$-dimensional subspace $K_8:=W_2\otimes W_3$ of $(\cR_\lambda)_{(1,1)}=\superwedge^2V^*\otimes\superwedge^3V^*$. We claim that {$K_8\subseteq\bigcap_{i=1}^8 \bigl(I_{P_i,X_\lambda}^{\langle 2\rangle}\bigr)_{(1,1)}$}. Indeed, let $\alpha\in W_2$ and $\beta\in W_3$. By construction, $\langle\alpha,v_i\wedge w_i\rangle=0$ and $\langle\beta,v_i\wedge w_i\wedge z_i\rangle=0$ for every $i=1,\dots,8$. Thus, the two linear forms $\alpha\otimes 1$ and $1\otimes\beta$ vanish at $P_i$, and therefore
\[
(\alpha\otimes 1)\star(1\otimes\beta)
=
\alpha\otimes\beta
\in
I_{P_i,X_\lambda}^2
=
I_{P_i,X_\lambda}^{\langle2\rangle}.
\]
Since $K_8$ is spanned by tensors of the form $\alpha\otimes\beta$, it follows that {$K_8\subseteq\bigcap_{i=1}^8 \bigl(I_{P_i,X_\lambda}^{\langle 2\rangle}\bigr)_{(1,1)}$}. By \Cref{thm: consistency}, we obtain $\pi_{(2,2,1)}(K_8)\subseteq\bigcap_{i=1}^8 \bigl(I_{\mathrm S}(P_i)^{\langle 2\rangle}\bigr)_{(2,2,1)}$.

It remains to prove that the restriction of $\pi_{(2,2,1)}$ to $K_8$ is injective. Since $\dim W_2=\dim W_3=2$, every element of $W_2\otimes W_3$ has tensor rank at most $2$. On the other hand, by \Cref{lem: tensorrank 221}, every nonzero element of $\ker(\pi_{(2,2,1)})$ has tensor rank at least $3$. Therefore, $K_8\cap\ker(\pi_{(2,2,1)})=\{0\}$. Hence, $\pi_{(2,2,1)}|_{K_8}$ is injective and $\dim\pi_{(2,2,1)}(K_8)=4$. The Schur Dual Terracini Lemma therefore gives
\[
\dim\sigma_8(F_{(2,2,1),5})
=74-\dim\left(\bigcap_{i=1}^8 \bigl(I_{\mathrm S}(P_i)^{\langle 2\rangle}\bigr)_{(2,2,1)}\right)
\leq 74-\dim\pi_{(2,2,1)}(K_8)
=70.
\]
Together with the computational lower bound $\dim\sigma_8(F_{(2,2,1),5})\geq70$, we get $\dim\sigma_8(F_{(2,2,1),5})=70$. Since its expected dimension is $71$, we conclude that $\delta_8(F_{(2,2,1),5})=1$.

We now prove that $\dim\sigma_9(F_{(2,2,1),5})\leq73$. Let $Q_i=(\rho_i\subset G_i)\in F_{(2,2,1),5}$, for $i=1,\dots,9$, be general points. Choose vectors $a_i,b_i,c_i\in V$ such that $\rho_i=\langle a_i,b_i\rangle$ and $G_i=\langle a_i,b_i,c_i\rangle$. Set
\[
U'_2:=\left\langle a_i\wedge b_i\mid i=1,\dots,9\right\rangle\subseteq\superwedge^2V,\qquad
U'_3:=\left\langle a_i\wedge b_i\wedge c_i\mid i=1,\dots,9\right\rangle\subseteq\superwedge^3V.
\]
Since the points are general, both spaces have dimension $9$. Therefore, their annihilators $W'_2:=(U'_2)^\perp\subseteq\superwedge^2V^*$ and $W'_3:=(U'_3)^\perp\subseteq\superwedge^3V^*$ are both one-dimensional. Consider the one-dimensional subspace $K_9:=W'_2\otimes W'_3\subseteq\superwedge^2V^*\otimes\superwedge^3V^*$. As above, for every $\alpha\in W'_2$ and $\beta\in W'_3$, we have $\langle\alpha,a_i\wedge b_i\rangle=0$ and $\langle\beta,a_i\wedge b_i\wedge c_i\rangle=0$ for every $i=1,\dots,9$. It follows that {$K_9\subseteq\bigcap_{i=1}^9\bigl(I_{Q_i,X_{(2,2,1)}}^{\langle 2\rangle}\bigr)_{(1,1)}$} and, by \Cref{thm: consistency},
\[
\pi_{(2,2,1)}(K_9)\subseteq\bigcap_{i=1}^9\bigl(I_{\mathrm S}(Q_i)^{\langle2\rangle}\bigr)_{(2,2,1)}.
\]
Every nonzero element of $K_9$ has tensor rank $1$, whereas every nonzero element of $\ker(\pi_{(2,2,1)})$ has tensor rank at least $3$ by \Cref{lem: tensorrank 221}. Therefore, $K_9\cap\ker(\pi_{(2,2,1)})=\{0\}$, and hence $\dim\pi_{(2,2,1)}(K_9)=1$. Applying again the Schur Dual Terracini Lemma, we obtain
\[
\dim\sigma_9(F_{(2,2,1),5})
=74-\dim\left(\bigcap_{i=1}^9 \bigl(I_{\mathrm S}(Q_i)^{\langle2\rangle}\bigr)_{(2,2,1)}\right)
\leq74-\dim\pi_{(2,2,1)}(K_9)
=73.
\]
Together with the computational lower bound $\dim\sigma_9(F_{(2,2,1),5})\geq73$, this gives $\dim\sigma_9(F_{(2,2,1),5})=73$. Since its expected dimension is $74$, we conclude that $\delta_9(F_{(2,2,1),5})=1$.

Thus, $\sigma_8(F_{(2,2,1),5})$ and $\sigma_9(F_{(2,2,1),5})$ both have defect $1$. Since $\sigma_{10}(F_{(2,2,1),5})$ fills the ambient space, the same
holds for every $s\geq10$. All the remaining secant varieties have the expected dimension, as claimed.
\end{proof}

\subsection{Computational classifications in dimensions four and five}
We conclude this section by recording the computational
classifications of the flag varieties embedded by
$\cO(\mathbf 1)$ in dimensions four and five.

\begin{theorem}
\label{thm: classification flags dimension four}
Let $V\simeq\C^4$.
All the partial flag varieties of $V$ embedded by
$\cO(\mathbf 1)$ are nondefective, with the exception of
\[
F_{(2,1),4}
\simeq
\Fl(1,2;V),
\qquad
F_{(2,2,1),4}
\simeq
\Fl(2,3;V),
\qquad
\text{and}
\qquad
{
F_{(2,1,1),4}
\simeq
\Fl(1,3;V).
}
\]
\end{theorem}

\begin{proof}
The defectivity of $F_{(2,1),4}$
and of its dual variety
$F_{(2,2,1),4}
$
are proved in \Cref{thm: def of F124}. The defectivity of
$F_{(2,1,1),4}\simeq\Fl(1,3;V)$ follows from
\cite[Corollary~1.2]{Baur2004}. The cases $F_{(1),4}$ and $F_{(1,1,1),4}$ are immediate, since the associated flag varieties are isomorphic to their ambient projective spaces. The remaining cases can be checked by running the scripts \path{computations/5.4_theorem_flag_*.m2} provided in \cite{Repo}.
\end{proof}

\begin{theorem}
\label{thm: classification flags dimension five}
Let $V\simeq\C^5$.
All the partial flag varieties of $V$ embedded by
$\cO(\mathbf 1)$ are nondefective, with the exception of
\[
F_{(2,2,1),5}
\simeq
\Fl(2,3;V)
\qquad
\text{and}
\qquad
F_{(2,1,1,1),5}
\simeq
\Fl(1,4;V).
\]
\end{theorem}

\begin{proof}
The defectivity of
$F_{(2,2,1),5}$
is proved in \Cref{thm: def of F235}. The defectivity of
$F_{(2,1,1,1),5}
\simeq
\Fl(1,4;V)$
follows from
\cite[Corollary~1.2]{Baur2004}. The cases $F_{(1),5}$ and $F_{(1,1,1,1),5}$ are immediate, since the associated flag varieties are isomorphic to their ambient projective spaces. The remaining cases can be checked by running the scripts \path{computations/5.5_theorem_flag_*.m2} provided in \cite{Repo}.
\end{proof}

The remainder of the paper is devoted to the family $F_{(2,1),n}
\simeq
\Fl(1,2;V_n)$,
for which we obtain the complete classification of the dimensions of
all secant varieties.

\section{Secant varieties of
\texorpdfstring{$F_{(2,1),n}$}{F(2,1;n)}}\label{sec: secants two step flags}

\noindent In this section, $V_n$ denotes an $n$-dimensional complex vector
space and $F_{(2,1),n}
\simeq
\Fl(1,2;V_n)$
denotes the flag variety embedded by $\cO(1,1)$ in $\P\left(\S_{(2,1)}V_n\right)$.
We also consider the incidence realization
$
F_{(2,1),n}
\subseteq
X_n$, $X_n
\coloneq
\P\left(\superwedge^2V_n\right)\times\P(V_n)$.

The order of the factors in $X_n$ is the reverse of the order
used for $X_{(2,1)}$ in the general slot-by-slot construction.
Throughout this section, we use the canonical isomorphism exchanging
the two factors and continue to denote the corresponding Cartan
projection by $\pi_{(2,1)}$.

The expected dimension of $\sigma_s(F_{(2,1),n})$ is
\[
\expdim\sigma_s(F_{(2,1),n})
=
\min
\left\{
\dim\S_{(2,1)}V_n,\,
s(2n-2)
\right\}
-1.
\]
We prove that, except for the defective cases analyzed in
\Cref{sect: defectiveness}, the secant variety
$\sigma_s(F_{(2,1),n})$ has the expected dimension.

For $P\in X_n$, the notation $I_{P,X_n}^{\langle2\rangle}$
denotes the ideal of the geometric double point supported at $P$. As recalled in
\Cref{rem: geom int of pilambda and straightening}, $I_{P,X_n}^{\langle2\rangle}
=
I_{P,X_n}^{2}$,
but we use angled brackets throughout this section in order to keep
the geometric meaning visible.

\subsection{Slotwise formulation and numerical reduction}
\label{subsect: slotwise formulation numerical reduction}

The following proposition translates the secant-dimension problem into
an interpolation problem for geometric double points in $X_n$.

\begin{proposition}
\label{terracini21}
{Let $n\geq3$ and $s\geq1$,} and let $P_1,\ldots,P_s \in F_{(2,1),n} \subseteq X_n$ be general points. Then
\[
\dim\sigma_s(F_{(2,1),n})
=
\dim\S_{(2,1)}V_n
-
\dim
\left(
\bigcap_{i=1}^s
\pi_{(2,1)}
\left(
\left(
I_{P_i,X_n}^{\langle2\rangle}
\right)_{(1,1)}
\right)
\right)
-1.
\]
\end{proposition}

\begin{proof}
By \Cref{cor: consistency two step flag}, one has
\[
\pi_{(2,1)}
\left(
\left(
I_{P_i,X_n}^{\langle2\rangle}
\right)_{(1,1)}
\right)
=
\left(
I_{\mathrm S}(P_i)^{\langle2\rangle}
\right)_{(2,1)}
\]
for every $i=1,\ldots,s$. The statement therefore follows from
\Cref{cor: schur dual terracini}.
\end{proof}

{Thus, passing to $X_n$ is not a return to classical symmetric apolarity: the double-point conditions used below are the slotwise representatives, provided by consistency, of the Schur-apolar conormal spaces appearing in the Schur Dual Terracini Lemma.}

We next record the numerical data governing the critical secant
orders.

\begin{lemma}[Critical numerics]
\label{lemmanumerico}
Let
$n=6k+r$, $0\leq r\leq5$,
and set
\[
k_n
\coloneq
\left\lfloor
\frac{\dim\S_{(2,1)}V_n^*}{2n-2}
\right\rfloor,
\qquad
\delta_n
\coloneq
\dim\S_{(2,1)}V_n^*-k_n(2n-2).
\]
Then
\[
k_n
=
\begin{cases}
6k^2+k,       & r=0,\\
6k^2+3k,      & r=1,\\
6k^2+5k+1,    & r=2,\\
6k^2+7k+2,    & r=3,\\
6k^2+9k+3,    & r=4,\\
6k^2+11k+5,   & r=5,
\end{cases}
\qquad
\delta_n
=
\begin{cases}
0,       & r=0,\\
4k,      & r=1,\\
0,       & r=2,\\
0,       & r=3,\\
4k+2,    & r=4,\\
0,       & r=5.
\end{cases}
\]
Moreover,
$k_n-k_{n-6}
=
2n-5$
and
$
k_n-2k_{n-6}+k_{n-12}
=
12$.
\end{lemma}

\begin{proof}
One has
\[
\begin{aligned}
\dim\S_{(2,1)}V_n^*
&=
\frac{n(n^2-1)}{3}
\\
&=
\frac{
(6k+r)
\left(
36k^2+12kr+r^2-1
\right)
}{3}
\\
&=
72k^3
+
36k^2r
+
6kr^2
-
2k
+
\frac{r^3-r}{3},
\end{aligned}
\]
whereas
$2n-2
=
12k+2r-2$. For $r=0$, one has $n=6k$ and
\[
\dim\S_{(2,1)}V_n^*
=
72k^3-2k
=
(6k^2+k)(12k-2).
\]
Thus,
\[
k_n=6k^2+k,
\qquad
\delta_n=0.
\]
Furthermore,
\[
\begin{aligned}
k_n-k_{n-6}
&=
6k^2+k
-
\bigl(
6(k-1)^2+k-1
\bigr)
\\
&=
12k-5
=
2n-5,
\end{aligned}
\]
and
\[
\begin{aligned}
k_n-2k_{n-6}+k_{n-12}
&=
(k_n-k_{n-6})
-
(k_{n-6}-k_{n-12})
\\
&=
(2n-5)
-
\bigl(
2(n-6)-5
\bigr)
\\
&=
12.
\end{aligned}
\]
The other congruence classes are obtained by the same computation.
\end{proof}

\begin{remark}
\label{rem:reduction}
Fix $n,s\in\N_{>0}$,
and let
$P_1,\ldots,P_s
\in
F_{(2,1),n}$ be general points. We have
\[
\expdim\sigma_s(F_{(2,1),n})
=
\min
\left\{
\dim\S_{(2,1)}V_n,\,
s(2n-2)
\right\}
-1,
\]
and, by \Cref{terracini21},
\[
\dim\sigma_s(F_{(2,1),n})
=
\dim\S_{(2,1)}V_n
-
\dim
\left(
\bigcap_{i=1}^s
\pi_{(2,1)}
\left(
\left(
I_{P_i,X_n}^{\langle2\rangle}
\right)_{(1,1)}
\right)
\right)
-1.
\]
Thus,
\[
\dim\sigma_s(F_{(2,1),n})
=
\expdim\sigma_s(F_{(2,1),n})
\]
if and only if
\begin{equation}
\label{goal}
\dim
\left(
\bigcap_{i=1}^s
\pi_{(2,1)}
\left(
\left(
I_{P_i,X_n}^{\langle2\rangle}
\right)_{(1,1)}
\right)
\right)
=
\max
\left\{
0,\,
\dim\S_{(2,1)}V_n^*
-
s(2n-2)
\right\}.
\end{equation}
Since
\[
\codim_{\S_{(2,1)}V_n^*}
\pi_{(2,1)}
\left(
\left(
I_{P_i,X_n}^{\langle2\rangle}
\right)_{(1,1)}
\right)
=
2n-2,
\]
this means that these vector spaces meet as transversely as possible.

Now let $k_n$ and $\delta_n$ be as in
\Cref{lemmanumerico}, and let
\[
\Delta_n
\subseteq
\superwedge^2V_n^*\otimes V_n^*
\]
be a vector subspace such that
\[
\codim_{\superwedge^2V_n^*\otimes V_n^*}(\Delta_n)
=
\delta_n
\qquad
\text{and}
\qquad
\Delta_n
\supseteq
\left(
I_{P,X_n}^{\langle2\rangle}
+
I_{F_{(2,1),n},X_n}
\right)_{(1,1)}
\]
for some general point $
P\in F_{(2,1),n}$.
For $n\geq18$, whenever it is used in the inductive arguments
below, $\Delta_n$ is chosen in a special configuration compatible
with the induction, as described in \Cref{rem:construction}. If
\begin{equation}
\label{goal2}
\left(
\bigcap_{i=1}^{k_n}
\pi_{(2,1)}
\left(
\left(
I_{P_i,X_n}^{\langle2\rangle}
\right)_{(1,1)}
\right)
\right)
\cap
\pi_{(2,1)}(\Delta_n)
=
\{0\}
\end{equation}
for general points
$
P_1,\ldots,P_{k_n}
\in
F_{(2,1),n}$,
then, by a standard linear-algebra argument,
\eqref{goal} holds for every $s\in\N_{>0}$.
By
\Cref{rem: geom int of pilambda and straightening},
$
\ker\pi_{(2,1)}
=
\left(
I_{F_{(2,1),n},X_n}
\right)_{(1,1)}$.
Thus, \eqref{goal2} is in turn equivalent to
\begin{equation}
\label{goal3}
\left(
\bigcap_{i=1}^{k_n}
\left(
I_{P_i,X_n}^{\langle2\rangle}
+
I_{F_{(2,1),n},X_n}
\right)_{(1,1)}
\right)
\cap
\Delta_n
=
\left(
I_{F_{(2,1),n},X_n}
\right)_{(1,1)}.
\end{equation}
Hence, to prove that $
\sigma_s(F_{(2,1),n})$
is nondefective for every $s\in\N_{>0}$, it is enough to prove
\eqref{goal3}.
\end{remark}

\subsection{Inductive specializations and subflags}
\label{subsect: recursive specializations}

\begin{remark}
\label{rem:construction}
In the proofs of \Cref{lemL} and \Cref{goodpost}, we need to choose
$\Delta_n$ in some special configurations.
We give a simultaneous construction of a nonempty family of
subspaces satisfying both the two-subflag compatibility required in
\Cref{lemL} and the compatibility with the chosen subspace in
dimension $n-6$ required in \Cref{goodpost}. The base-case computations in \cite{Repo} also provide the initial compatible choices used below: on the prescribed coordinate subflags, the chosen $\Delta_{13}$ induces $\Delta_7$, and the chosen $\Delta_{16}$ induces $\Delta_{10}$.

Assume $n\geq18$, which is precisely the range in which this
construction is used. If $\delta_n>0$, then
$\delta_{n-12}>0$ by \Cref{lemmanumerico}. Let $L,M \subseteq F_{(2,1),n}$ be two general subvarieties isomorphic to $F_{(2,1),n-6}$, and let $P\in L\cap M$ be general. Denote by $V_L,V_M\subseteq V_n$ the corresponding $(n-6)$-dimensional vector subspaces, so that
\[
L=\Fl(1,2;V_L),
\qquad
M=\Fl(1,2;V_M),
\]
and set
\[
X_L
=
\P\left(\superwedge^2V_L\right)\times\P(V_L),
\qquad
X_M
=
\P\left(\superwedge^2V_M\right)\times\P(V_M).
\]
We want to construct a vector subspace $\Delta_n \subseteq \superwedge^2V_n^*\otimes V_n^* $ such that
\[
\Delta_n
\supseteq
\left(
I_{P,X_n}^{\langle2\rangle}
+
I_{F_{(2,1),n},X_n}
\right)_{(1,1)}
\qquad
\text{and}
\qquad
\codim_{\superwedge^2V_n^*\otimes V_n^*}(\Delta_n)
=
\delta_n,
\]
with the following additional properties:
\begin{enumerate}[leftmargin=*]
    \item
    $
    \codim_{\superwedge^2V_n^*\otimes V_n^*}
    \left(
    \Delta_n+(I_{L,X_n})_{(1,1)}
    \right)
    =
    \delta_{n-6}$;

    \item
    $
    \codim_{\superwedge^2V_M^*\otimes V_M^*}
    \left(
    \frac{
    \Delta_n+(I_{M,X_n})_{(1,1)}
    }{
    (I_{X_M,X_n})_{(1,1)}
    }
    \right)
    =
    \delta_{n-6}$;

    \item
    $
    \codim_{\superwedge^2V_M^*\otimes V_M^*}
    \left(
    \frac{
    \Delta_n+(I_{M,X_n})_{(1,1)}
    }{
    (I_{X_M,X_n})_{(1,1)}
    }
    +
    (I_{L\cap M,X_M})_{(1,1)}
    \right)
    =
    \delta_{n-12}$.
\end{enumerate}
If $\delta_n=0$, we take
\[
\Delta_n
=
\superwedge^2V_n^*\otimes V_n^*,
\]
and all the required conditions are immediate. Hence, we may assume
that $\delta_n>0$.
Let
$Z\subset F_{(2,1),n}$
be the geometric double point of $F_{(2,1),n}$ supported at $P$, regarded as a
subscheme of $X_n$. Since $\pi_{(2,1)}$ is the restriction map and
\[
\ker(\pi_{(2,1)})
=
\left(
I_{F_{(2,1),n},X_n}
\right)_{(1,1)},
\]
while
\Cref{cor: consistency two step flag} and \Cref{rem: geometric meaning Schur square}
give
\[
\pi_{(2,1)}
\left(
\left(
I_{P,X_n}^{\langle2\rangle}
\right)_{(1,1)}
\right)
=
H^0
\left(
F_{(2,1),n},
\cO_{F_{(2,1),n}}(1,1)
\otimes
\cI_{Z,F_{(2,1),n}}
\right),
\]
we obtain
\[
\begin{aligned}
\left(
I_{Z,X_n}
\right)_{(1,1)}
&=
\pi_{(2,1)}^{-1}
\left(
\pi_{(2,1)}
\left(
\left(
I_{P,X_n}^{\langle2\rangle}
\right)_{(1,1)}
\right)
\right)
\\
&=
\left(
I_{P,X_n}^{\langle2\rangle}
\right)_{(1,1)}
+
\ker(\pi_{(2,1)})
\\
&=
\left(
I_{P,X_n}^{\langle2\rangle}
+
I_{F_{(2,1),n},X_n}
\right)_{(1,1)}.
\end{aligned}
\]
Given a subscheme $\eta\subseteq Z$, one has
$I_{\eta,X_n}\supseteq I_{Z,X_n}$.
We shall choose below a subscheme
$\eta\subseteq Z$ of length $\delta_n$ and set
\[
\Delta_n
\coloneq
\left(
I_{\eta,X_n}
\right)_{(1,1)}.
\]
Since $\cO_{X_n}(1,1)$ separates first jets along
$F_{(2,1),n}$, the double point $Z$ imposes independent
conditions on $H^0(X_n,\cO_{X_n}(1,1))$. Hence the same holds
for every subscheme $\eta\subseteq Z$, and therefore
\[
\codim_{\superwedge^2V_n^*\otimes V_n^*}(\Delta_n)
=
\deg(\eta)
=
\delta_n.
\]
To translate the remaining three codimension conditions, for every
smooth subvariety
$Y\subseteq F_{(2,1),n}$
containing $P=[p]$, set
\[
T_PY
\coloneq
\widehat T_PY/\langle p\rangle.
\]
Thus, $T_PY$ is the Zariski tangent space of $Y$ at
$P$. The local algebra of the double point $Z$ is
\[
\cO_{Z,P}
\simeq
\C\oplus T_PF_{(2,1),n}^*,
\]
where $T_PF_{(2,1),n}^*$ is the cotangent space at
$P$, and the second summand is the square-zero maximal ideal.
Consequently, subschemes
$\eta\subseteq Z$ of length $\delta_n$ are in bijection, by
annihilation, with $(\delta_n-1)$-dimensional vector subspaces $U\subseteq T_PF_{(2,1),n}$.
Under the first-jet identification, for $Y\in\{L,M,L\cap M\}$, the image of $\left(I_{Y,X_n}\right)_{(1,1)}$ in
the cotangent space
$T_PF_{(2,1),n}^*$ is the conormal space
$\left(T_PY\right)^\perp$. Therefore,
\[
\codim
\left(
\left(I_{\eta,X_n}\right)_{(1,1)}
+
\left(I_{Y,X_n}\right)_{(1,1)}
\right)
=
1+\dim\left(U\cap T_PY\right)
=
\deg(\eta\cap Y).
\]
The analogous statement holds after restriction to $X_M$. Hence, the three codimension conditions above are respectively
equivalent to
\begin{enumerate}[leftmargin=*]
    \item
    $\deg(\eta\cap L)=\delta_{n-6}$;

    \item
    $\deg(\eta\cap M)=\delta_{n-6}$;

    \item
    $\deg(\eta\cap L\cap M)=\delta_{n-12}$.
\end{enumerate}
Therefore, finding $\eta$ with the required properties is equivalent
to finding $U
\in
\Gr
\left(
\delta_n-1, T_PF_{(2,1),n}
\right)$
such that
\begin{enumerate}[leftmargin=*]
    \item
    $
    \dim\bigl(U\cap T_PL\bigr)
    =
    \delta_{n-6}-1$;

    \item
    $
    \dim\bigl(U\cap T_PM\bigr)
    =
    \delta_{n-6}-1$;

    \item
$
    \dim\bigl(U\cap T_P(L\cap M)\bigr)
    =
    \delta_{n-12}-1$.
\end{enumerate}
For these subflags, set
\[
T_0
\coloneq
 T_P(L\cap M)
=
 T_PL\cap T_PM.
\]
Moreover,
$
\dim
\frac{ T_PL}{T_0}
=
\dim
\frac{ T_PM}{T_0}
=
12$,
and, by \Cref{lemmanumerico},
$
\delta_n-\delta_{n-6}
=
\delta_{n-6}-\delta_{n-12}
=
4$.

For the inductive applications in both \Cref{lemL,goodpost},
choose special subspaces
\[
\Delta_{n-6}^{L}
\subseteq
\superwedge^2V_L^*\otimes V_L^*,
\qquad
\Delta_{n-6}^{M}
\subseteq
\superwedge^2V_M^*\otimes V_M^*,
\]
compatible with the corresponding induction hypotheses and
inducing on the common subflag
$L\cap M\simeq F_{(2,1),n-12}$ the same fixed special subspace
$\Delta_{n-12}$. By homogeneity, we may assume that these subspaces
are supported at $P$. Let
$U_L\subseteq T_PL$ and $U_M\subseteq T_PM$
be the tangent subspaces corresponding, under the first-jet
identification, to $\Delta_{n-6}^{L}$ and
$\Delta_{n-6}^{M}$, respectively. Since these two subspaces induce
the same fixed subspace $\Delta_{n-12}$ on $X_{L\cap M}$, there is a
subspace $U'\subseteq T_0$ such that
\[
U_L\cap T_0
=
U'
=
U_M\cap T_0,
\qquad
\dim U'
=
\delta_{n-12}-1.
\]
Set $U
\coloneq
U_L+U_M$.
Since
$ T_PL\cap T_PM=T_0$, one has
$U_L\cap U_M=U'$. Therefore, by \Cref{lemmanumerico},
\[
\begin{aligned}
\dim U
&=
\dim U_L+\dim U_M-\dim U'
\\
&=
2(\delta_{n-6}-1)-(\delta_{n-12}-1)
\\
&=
\delta_n-1.
\end{aligned}
\]
Moreover,
\[
U\cap T_PL=U_L,
\qquad
U\cap T_PM=U_M,
\qquad
U\cap T_0=U'.
\]
Indeed, if $u_L+u_M\in T_PL$, with
$u_L\in U_L$ and $u_M\in U_M$, then $u_M\in
 T_PL\cap T_PM
=
T_0$,
and hence $u_M\in U'\subseteq U_L$. Thus
$u_L+u_M\in U_L$. The argument for $M$ is identical.

Let $\eta\subseteq Z$ be the subscheme corresponding, by
annihilation, to $U$, and set
\[
\Delta_n
\coloneq
\left(
I_{\eta,X_n}
\right)_{(1,1)}.
\]
Then $\Delta_n$ satisfies the three required codimension
conditions. Furthermore, the induced subspaces
\[
\Delta_n^{L}
\coloneq
\frac{
\Delta_n+(I_{L,X_n})_{(1,1)}
}{
(I_{X_L,X_n})_{(1,1)}
},
\qquad
\Delta_n^{M}
\coloneq
\frac{
\Delta_n+(I_{M,X_n})_{(1,1)}
}{
(I_{X_M,X_n})_{(1,1)}
}
\]
correspond respectively to
$U\cap T_PL=U_L$ and
$U\cap T_PM=U_M$. Consequently,
\[
\Delta_n^{L}
=
\Delta_{n-6}^{L},
\qquad
\Delta_n^{M}
=
\Delta_{n-6}^{M},
\]
and the further subspace induced on $X_{L\cap M}$ is exactly the
fixed special subspace $\Delta_{n-12}$.
Thus the same $\Delta_n$ satisfies the hypotheses of
\Cref{lemL} and induces on $X_L$ the fixed subspace to which the
induction hypothesis in \Cref{goodpost} applies.
The construction is not canonical: varying the compatible
lower-dimensional choices gives a nonempty family of admissible
subspaces $\Delta_n$. Throughout the remainder of this section,
whenever such a subspace is required, we fix one member of this
family. All subsequent semicontinuity arguments are applied with
$\Delta_n$ fixed and with only the support points varying.
\end{remark}

\begin{remark}[Intersections of subflags]
\label{flagintersec}
Let $F_{(2,1),n-a}$, $F_{(2,1),n-b}
\subseteq
F_{(2,1),n}$ be two subflags associated with vector subspaces $V_{n-a},V_{n-b}
\subseteq
V_n$ of codimensions $a$ and $b$, respectively. Thus, $F_{(2,1),n-a}
=
\left\{
(\ell,\Pi)\in F_{(2,1),n}
\ \middle|\
\Pi\subseteq V_{n-a}
\right\}$ and 
$F_{(2,1),n-b}
=
\left\{
(\ell,\Pi)\in F_{(2,1),n}
\ \middle|\
\Pi\subseteq V_{n-b}
\right\}$.
Consequently,
\[
F_{(2,1),n-a}
\cap
F_{(2,1),n-b}
=
\left\{
(\ell,\Pi)\in F_{(2,1),n}
\ \middle|\
\Pi\subseteq V_{n-a}\cap V_{n-b}
\right\}.
\]
If the two subflags are general, then the corresponding vector
subspaces are general and
\[
\codim_{V_n}
\left(
V_{n-a}\cap V_{n-b}
\right)
=
\min\{n,a+b\}.
\]
Therefore,
\[
F_{(2,1),n-a}
\cap
F_{(2,1),n-b}
=
\begin{cases}
F_{(2,1),n-a-b},
&
\text{if }a+b\leq n-2,
\\[2pt]
\emptyset,
&
\text{if }a+b>n-2.
\end{cases}
\]
\end{remark}

\begin{lemma}[The ideal of a subflag]
\label{lemma1}
Let $m,n\in\N_{>0}$, $n\geq3m$, and let $L
\simeq
F_{(2,1),n-m}
\subseteq
F_{(2,1),n}$ be the subflag associated with an $(n-m)$-dimensional vector subspace
$
V_L\subseteq V_n$.
Set
\[
X_L
\coloneq
\P\left(\superwedge^2V_L\right)
\times
\P(V_L)
\subseteq
X_n.
\]
Then
$I_{L,X_n}
=
I_{X_L,X_n}
+
I_{F_{(2,1),n},X_n}$.
\end{lemma}

\begin{proof}
Choose a basis
$
\{e_1,\ldots,e_n\}$ of $V_n$ such that
$
V_L
=
\langle e_1,\ldots,e_{n-m}\rangle$,
and let
$
\{x_1,\ldots,x_n\}$
be the dual basis. We use the homogeneous coordinates
$
[x_1:\cdots:x_n]
$
on $\P(V_n)$ and the induced Pl\"ucker coordinates
$
[x_i\wedge x_j]_{1\leq i<j\leq n}$
on $\P(\superwedge^2V_n)$. The ideal of $X_L$ in $X_n$ is generated by the coordinates that
vanish upon restriction to $V_L$:
\[
I_{X_L,X_n}
=
\left\langle
x_j,\,
x_p\wedge x_q
\ \middle|\
n-m+1\leq j\leq n,\,
\{p,q\}\not\subseteq[n-m]
\right\rangle.
\]
The ideal of $F_{(2,1),n}$ in $X_n$ is generated by the Pl\"ucker
relations together with the incidence relations
\[
(x_j\wedge x_k)\otimes x_i
-
(x_i\wedge x_k)\otimes x_j
+
(x_i\wedge x_j)\otimes x_k,
\qquad
1\leq i<j<k\leq n.
\]
Equivalently,
\[
\begin{aligned}
I_{F_{(2,1),n},X_n}
={}&
\left\langle
(x_j\wedge x_k)\otimes x_i
-
(x_i\wedge x_k)\otimes x_j
+
(x_i\wedge x_j)\otimes x_k
\right\rangle_{1\leq i<j<k\leq n}
\\
&+
I_{\Gr(2,V_n)\times\P(V_n),X_n}.
\end{aligned}
\]
After quotienting by $I_{X_L,X_n}$, all coordinates involving an
index greater than $n-m$ vanish. Hence the surviving Pl\"ucker and
incidence relations are exactly those defining
\[
L
=
F_{(2,1),n-m}
\subseteq
X_L.
\]
Therefore,
\[
\frac{
I_{X_L,X_n}
+
I_{F_{(2,1),n},X_n}
}{
I_{X_L,X_n}
}
=
I_{L,X_L}.
\]
Let
\[
\alpha:
\cR_{(2,1)}
\longrightarrow
\frac{\cR_{(2,1)}}{I_{X_L,X_n}}
\]
be the quotient map. Since $\ker(\alpha)=I_{X_L,X_n}$,
we obtain
\[
\begin{aligned}
I_{L,X_n}
&=
\alpha^{-1}(I_{L,X_L})
\\
&=
\alpha^{-1}
\left(
\alpha
\left(
I_{X_L,X_n}
+
I_{F_{(2,1),n},X_n}
\right)
\right)
\\
&=
I_{X_L,X_n}
+
I_{F_{(2,1),n},X_n}.
\end{aligned}
\]
\end{proof}

\subsection{The three-subflag calculation}
\label{subsect: three subflag calculation}

\begin{lemma}[The $432$-dimensional residual space]
\label{lemmabasis}
Let $n\geq18$, and let $L,M,N\subseteq F_{(2,1),n}$ be three general subflags isomorphic to $F_{(2,1),n-6}$.
For $J\in\{L,M,N\}$, let $V_J\subseteq V_n$ be the corresponding
$(n-6)$-dimensional vector subspace. Let $\{\lambda_1,\ldots,\lambda_6\}$, $\{\mu_1,\ldots,\mu_6\}$, $\{\eta_1,\ldots,\eta_6\}$ be bases of $V_L^\perp$, $V_M^\perp$, and $V_N^\perp$, respectively.
For $1\leq i,j,k\leq6$, set
\[
\begin{aligned}
u_{ijk}
&\coloneq
\mu_j\wedge\eta_k\otimes\lambda_i
-
\eta_k\wedge\lambda_i\otimes\mu_j,
\\
v_{ijk}
&\coloneq
\eta_k\wedge\lambda_i\otimes\mu_j
-
\lambda_i\wedge\mu_j\otimes\eta_k.
\end{aligned}
\]
Then
$
\left(
I_{L\cup M\cup N,X_n}
\right)_{(1,1)}
=
T
\oplus
\left(
I_{F_{(2,1),n},X_n}
\right)_{(1,1)}$,
where $T\subseteq\S_{(2,1)}V_n^*$ has dimension $432$ and
$
\mathcal B
\coloneq
\left\{
u_{ijk},v_{ijk}
\ \middle|\
1\leq i,j,k\leq6
\right\}$
is a basis of $T$.
\end{lemma}

\begin{proof}
Recall that
$H^0(X_n,\cO_{X_n}(1,1))
=\superwedge^2V_n^*\otimes V_n^*$
and that
$
\superwedge^2V_n^*\otimes V_n^*
=
\superwedge^3V_n^*
\oplus
\S_{(2,1)}V_n^*$.
Moreover,
$
\left(
I_{F_{(2,1),n},X_n}
\right)_{(1,1)}
=
\superwedge^3V_n^*$. Since $L\cup M\cup N\subseteq F_{(2,1),n}$, the preceding space is
contained in
$\left(I_{L\cup M\cup N,X_n}\right)_{(1,1)}$. Therefore,
\[
\begin{aligned}
\left(
I_{L\cup M\cup N,X_n}
\right)_{(1,1)}
={}&
\left(
I_{F_{(2,1),n},X_n}
\right)_{(1,1)}
\\
&\oplus
\left(
\left(
I_{L\cup M\cup N,X_n}
\right)_{(1,1)}
\cap
\S_{(2,1)}V_n^*
\right).
\end{aligned}
\]
Since $L$, $M$, and $N$ are general and $n\geq18$, the annihilators
$V_L^\perp$, $V_M^\perp$, and $V_N^\perp$ are in direct sum. Choose a
complement $U\subseteq V_n^*$ of dimension $n-18$, so that
$V_n^*=V_L^\perp\oplus V_M^\perp\oplus V_N^\perp\oplus U$. Iterating the Littlewood--Richardson decomposition for Schur modules
of a direct sum gives
\[
\S_{(2,1)}V_n^*
\simeq
\left(
V_L^\perp
\otimes
V_M^\perp
\otimes
V_N^\perp
\right)^{\oplus2}
\oplus A,
\]
where $A$ is a direct sum of tensor products of Schur modules of
$V_L^\perp$, $V_M^\perp$, $V_N^\perp$, and $U$, and no summand of
$A$ has positive degree simultaneously in
$V_L^\perp$, $V_M^\perp$, and $V_N^\perp$. Indeed, with respect to the chosen direct-sum decomposition of
$V_n^*$, the restriction to $\S_{(2,1)}V_L^*$ kills precisely the
summands having positive degree in $V_L^\perp$, and analogously for
$M$ and $N$. Hence a summand belongs to the kernels of all three
restriction maps if and only if it has positive degree in each of
$V_L^\perp$, $V_M^\perp$, and $V_N^\perp$. Since
$\S_{(2,1)}$ has total degree $3$, such a summand must have degree
$1$ in each of these three spaces and degree $0$ in $U$. The
corresponding component $V_L^\perp\otimes V_M^\perp\otimes V_N^\perp$ occurs with multiplicity $2$, equal
to the dimension of the irreducible representation of
$\mathfrak S_3$ associated with the partition $(2,1)$. It follows that
\[
\left(
I_{L\cup M\cup N,X_n}
\right)_{(1,1)}
\cap
\S_{(2,1)}V_n^*
\simeq
\left(
V_L^\perp
\otimes
V_M^\perp
\otimes
V_N^\perp
\right)^{\oplus2}.
\]
In particular, this space has dimension $2\cdot6^3=432$.

{It remains to verify that the elements of $\mathcal B$ form a
basis of this space. For a fixed triple $(i,j,k)$, set
$t^{(1)}_{ijk}
=\mu_j\wedge\eta_k\otimes\lambda_i$,
$t^{(2)}_{ijk}
=\eta_k\wedge\lambda_i\otimes\mu_j$, and
$t^{(3)}_{ijk}
=\lambda_i\wedge\mu_j\otimes\eta_k$.
These tensors belong respectively to the three distinct summands
\[
(V_M^\perp\wedge V_N^\perp)\otimes V_L^\perp,
\qquad
(V_N^\perp\wedge V_L^\perp)\otimes V_M^\perp,
\qquad
(V_L^\perp\wedge V_M^\perp)\otimes V_N^\perp
\]
of $\superwedge^2V_n^*\otimes V_n^*$.
}
{
The wedge map
$\omega:\superwedge^2V_n^*\otimes V_n^*
\longrightarrow\superwedge^3V_n^*$
sends the three tensors to the same element
$\lambda_i\wedge\mu_j\wedge\eta_k$. Therefore,
$u_{ijk}=t^{(1)}_{ijk}-t^{(2)}_{ijk}$ and
$v_{ijk}=t^{(2)}_{ijk}-t^{(3)}_{ijk}$ belong to
$\ker(\omega)=\S_{(2,1)}V_n^*$.
}
{
Each of them contains a factor in every one of
$V_L^\perp$, $V_M^\perp$, and $V_N^\perp$, and hence restricts to
zero on $L$, on $M$, and on $N$. Thus
$u_{ijk},v_{ijk}$ belong to
$\left(I_{L\cup M\cup N,X_n}\right)_{(1,1)}
\cap\S_{(2,1)}V_n^*$.
}
{
For a fixed triple $(i,j,k)$, the tensors
$t^{(1)}_{ijk}$, $t^{(2)}_{ijk}$, and $t^{(3)}_{ijk}$ lie in three
distinct direct summands, so $u_{ijk}$ and $v_{ijk}$ are linearly
independent. As $(i,j,k)$ varies, the tensors $t^{(a)}_{ijk}$ are
distinct elements of the natural tensor bases of those summands.
Hence the full collection $\mathcal B$ is linearly independent.}

Since $\mathcal B$ consists of $2\cdot6^3=432$ elements, it is a basis
of the required space. Taking $T\coloneq\langle\mathcal B\rangle$
concludes the proof.
\end{proof}

\begin{lemma}
\label{lemmaLMN}
Let $n\geq18$, and let
$L,M,N\subseteq F_{(2,1),n}$ be three general subflags isomorphic to
$F_{(2,1),n-6}$. Moreover, let
$l_i\in L$, $m_i\in M$, and $n_i\in N$ be general points on
$L$, $M$, and $N$, respectively, for $i=1,\ldots,12$. Set
$$
K_{L,i}
=
\left(
I_{l_i,X_n}^{\langle2\rangle}
+
I_{F_{(2,1),n},X_n}
\right)_{(1,1)},
K_{M,i}
=
\left(
I_{m_i,X_n}^{\langle2\rangle}
+
I_{F_{(2,1),n},X_n}
\right)_{(1,1)},
K_{N,i}
=
\left(
I_{n_i,X_n}^{\langle2\rangle}
+
I_{F_{(2,1),n},X_n}
\right)_{(1,1)}
$$
for $i=1,\ldots,12$. Then
$$
\left(
I_{L\cup M\cup N,X_n}
\right)_{(1,1)}
\cap
\bigcap_{\substack{J=L,M,N\\ i=1,\ldots,12}}
K_{J,i}
=
\left(
I_{F_{(2,1),n},X_n}
\right)_{(1,1)}.
$$
\end{lemma}

\begin{proof}
We use the notation of \Cref{lemmabasis}, which gives
$
\left(
I_{L\cup M\cup N,X_n}
\right)_{(1,1)}
=
T
\oplus
\left(
I_{F_{(2,1),n},X_n}
\right)_{(1,1)}$,
where a basis of $T$ is
$$
\mathcal B
=
\left\{
\underbrace{
\mu_j\wedge\eta_k\otimes\lambda_i
-
\eta_k\wedge\lambda_i\otimes\mu_j
}_{\coloneq u_{ijk}},
\underbrace{
\eta_k\wedge\lambda_i\otimes\mu_j
-
\lambda_i\wedge\mu_j\otimes\eta_k
}_{\coloneq v_{ijk}}
\ \middle|\
1\leq i,j,k\leq6
\right\}.
$$
An element of
$\left(I_{L\cup M\cup N,X_n}\right)_{(1,1)}$
can thus be written as
$$
s
=
\sum_{i,j,k=1}^6
\left(
\alpha_{ijk}u_{ijk}
+
\beta_{ijk}v_{ijk}
\right)
+
s',
$$
for some $\alpha_{ijk},\beta_{ijk}\in\C$ and
$s'\in\left(I_{F_{(2,1),n},X_n}\right)_{(1,1)}$.
Proving the statement is equivalent to showing that
\[
s
\in
\bigcap_{\substack{J=L,M,N\\ i=1,\ldots,12}}
K_{J,i}
\quad\Longrightarrow\quad
\alpha_{ijk}
=
\beta_{ijk}
=
0
\text{ for every }1\leq i,j,k\leq6.
\]
Since $s$ already vanishes identically on each of $L$, $M$, and $N$,
its first derivatives along directions tangent to these subflags
vanish. Therefore, imposing that $s$ vanish to order at least two at a
point of one of the subflags is equivalent to imposing the vanishing
of its first derivatives along the directions normal to that subflag.

We now write these conditions explicitly. Let
$\{a_1,\ldots,a_6\}$, $\{b_1,\ldots,b_6\}$, and
$\{c_1,\ldots,c_6\}$ be the bases dual to
$\{\lambda_1,\ldots,\lambda_6\}$,
$\{\mu_1,\ldots,\mu_6\}$, and
$\{\eta_1,\ldots,\eta_6\}$, respectively. Since
$V_L^\perp$, $V_M^\perp$, and $V_N^\perp$ are in direct sum, we may
choose them so that
\begin{gather*}
\lambda_i(a_j)=\delta_{ij},
\quad
\lambda_i(b_j)=\lambda_i(c_j)=0,
\quad
\lambda_i(x)=0,
\quad
\forall\,i,j=1,\ldots,6,\ x\in V_L,
\\
\mu_i(b_j)=\delta_{ij},
\quad
\mu_i(a_j)=\mu_i(c_j)=0,
\quad
\mu_i(x)=0,
\quad
\forall\,i,j=1,\ldots,6,\ x\in V_M,
\\
\eta_i(c_j)=\delta_{ij},
\quad
\eta_i(a_j)=\eta_i(b_j)=0,
\quad
\eta_i(x)=0,
\quad
\forall\,i,j=1,\ldots,6,\ x\in V_N.
\end{gather*}
Let $l\in\{l_1,\ldots,l_{12}\}$. There exist two vectors
$x,y\in V_L$ such that
$l=(\langle x\rangle\subset\langle x,y\rangle)$. A deformation of $l$ along the normal directions of $L$ is
$$
l_{s,t}
=
\left(
\left\langle
\underbrace{
x+\sum_{i=1}^6s_i a_i
}_{x_s}
\right\rangle
\subset
\left\langle
x+\sum_{i=1}^6s_i a_i,
\underbrace{
y+\sum_{i=1}^6t_i a_i
}_{y_t}
\right\rangle
\right),
\qquad
s_i,t_i\in\C.
$$
We have
\begin{gather*}
(\mu_j\wedge\eta_k\otimes\lambda_i)(l_{s,t})
=
\lambda_i(x_s)
\begin{vmatrix}
\mu_j(x_s) & \mu_j(y_t)\\
\eta_k(x_s) & \eta_k(y_t)
\end{vmatrix}
=
s_i
\begin{vmatrix}
\mu_j(x) & \mu_j(y)\\
\eta_k(x) & \eta_k(y)
\end{vmatrix}
=
s_i
\left(
\mu_j(x)\eta_k(y)-\mu_j(y)\eta_k(x)
\right),
\\
(\eta_k\wedge\lambda_i\otimes\mu_j)(l_{s,t})
=
\mu_j(x_s)
\begin{vmatrix}
\eta_k(x_s) & \eta_k(y_t)\\
\lambda_i(x_s) & \lambda_i(y_t)
\end{vmatrix}
=
\mu_j(x)
\begin{vmatrix}
\eta_k(x) & \eta_k(y)\\
s_i & t_i
\end{vmatrix}
=
\mu_j(x)
\left(
\eta_k(x)t_i-\eta_k(y)s_i
\right),
\\
(\lambda_i\wedge\mu_j\otimes\eta_k)(l_{s,t})
=
\eta_k(x_s)
\begin{vmatrix}
\lambda_i(x_s) & \lambda_i(y_t)\\
\mu_j(x_s) & \mu_j(y_t)
\end{vmatrix}
=
\eta_k(x)
\begin{vmatrix}
s_i & t_i\\
\mu_j(x) & \mu_j(y)
\end{vmatrix}
=
\eta_k(x)
\left(
\mu_j(y)s_i-\mu_j(x)t_i
\right),
\\
u_{ijk}(l_{s,t})
=
\left(
2\mu_j(x)\eta_k(y)-\mu_j(y)\eta_k(x)
\right)s_i
-
\mu_j(x)\eta_k(x)t_i,
\\
v_{ijk}(l_{s,t})
=
-
\left(
\mu_j(x)\eta_k(y)+\mu_j(y)\eta_k(x)
\right)s_i
+
2\mu_j(x)\eta_k(x)t_i.
\end{gather*}
Since the differential of $s'$ along the normal directions of $L$
vanishes, the differential of $s$ along the normal directions of $L$,
evaluated at $l$, is
\begin{gather*}
d_ls
=
\sum_{i,j,k=1}^6
\alpha_{ijk}d_lu_{ijk}
+
\sum_{i,j,k=1}^6
\beta_{ijk}d_lv_{ijk}
\\
=
\sum_{i,j,k=1}^6
\alpha_{ijk}
\left(
\left(
2\mu_j(x)\eta_k(y)-\mu_j(y)\eta_k(x)
\right)ds_i
-
\mu_j(x)\eta_k(x)dt_i
\right)
\\
+
\sum_{i,j,k=1}^6
\beta_{ijk}
\left(
-
\left(
\mu_j(x)\eta_k(y)+\mu_j(y)\eta_k(x)
\right)ds_i
+
2\mu_j(x)\eta_k(x)dt_i
\right)
\\
=
\sum_{i=1}^6
\left(
\sum_{j,k=1}^6
\left(
\left(
2\mu_j(x)\eta_k(y)-\mu_j(y)\eta_k(x)
\right)\alpha_{ijk}
-
\left(
\mu_j(x)\eta_k(y)+\mu_j(y)\eta_k(x)
\right)\beta_{ijk}
\right)
\right)ds_i
\\
+
\sum_{i=1}^6
\left(
\sum_{j,k=1}^6
\left(
-\mu_j(x)\eta_k(x)\alpha_{ijk}
+
2\mu_j(x)\eta_k(x)\beta_{ijk}
\right)
\right)dt_i.
\end{gather*}
Thus, imposing $d_ls=0$, we obtain the following twelve equations in
the unknowns $\alpha_{ijk}$ and $\beta_{ijk}$:
\begin{gather*}
\sum_{j,k=1}^6
\left(
2\mu_j(x)\eta_k(y)-\mu_j(y)\eta_k(x)
\right)\alpha_{ijk}
-
\sum_{j,k=1}^6
\left(
\mu_j(x)\eta_k(y)+\mu_j(y)\eta_k(x)
\right)\beta_{ijk}
=
0,
\quad
i=1,\ldots,6,
\\
\sum_{j,k=1}^6
\mu_j(x)\eta_k(x)\alpha_{ijk}
-
2
\sum_{j,k=1}^6
\mu_j(x)\eta_k(x)\beta_{ijk}
=0,
\quad
i=1,\ldots,6.
\end{gather*}
Substituting the corresponding vectors $x,y\in V_L$ for every
$l\in\{l_1,\ldots,l_{12}\}$, we obtain $144$ linear equations in the
unknowns $\alpha_{ijk}$ and $\beta_{ijk}$. Repeating the same procedure for a point
$m=(\langle x\rangle\subset\langle x,y\rangle)$, with
$x,y\in V_M$, we obtain the twelve equations
\begin{gather*}
\sum_{i,k=1}^6
\left(
2\lambda_i(x)\eta_k(y)-\lambda_i(y)\eta_k(x)
\right)\alpha_{ijk}
+
\sum_{i,k=1}^6
\left(
2\lambda_i(y)\eta_k(x)-\lambda_i(x)\eta_k(y)
\right)\beta_{ijk}
=
0,
\quad
j=1,\ldots,6,
\\
\sum_{i,k=1}^6
\lambda_i(x)\eta_k(x)\alpha_{ijk}
+
\sum_{i,k=1}^6
\lambda_i(x)\eta_k(x)\beta_{ijk}
=
0,
\quad
j=1,\ldots,6.
\end{gather*}
As $m$ varies in $\{m_1,\ldots,m_{12}\}$, these give another $144$
linear equations.

Finally, for a point
$q=(\langle x\rangle\subset\langle x,y\rangle)$, with
$x,y\in V_N$, we obtain the twelve equations
\begin{gather*}
\sum_{i,j=1}^6
\left(
\lambda_i(x)\mu_j(y)+\lambda_i(y)\mu_j(x)
\right)\alpha_{ijk}
-
\sum_{i,j=1}^6
\left(
2\lambda_i(y)\mu_j(x)-\lambda_i(x)\mu_j(y)
\right)\beta_{ijk}
=0,
\quad
k=1,\ldots,6,
\\
2
\sum_{i,j=1}^6
\lambda_i(x)\mu_j(x)\alpha_{ijk}
-
\sum_{i,j=1}^6
\lambda_i(x)\mu_j(x)
\beta_{ijk}
=
0,
\quad
k=1,\ldots,6.
\end{gather*}
Substituting the corresponding vectors for the points
$n_1,\ldots,n_{12}$, we obtain the last $144$ equations. In
conclusion, we obtain a $432\times432$ full-rank linear system.
The full-rank computation can be checked by running \path{computations/6.8_lemma.m2}; see \cite{Repo}.
\end{proof}

\subsection{The Horace induction}
\label{subsect: Horace induction}
The auxiliary induction statements below are only needed in the
range $n\geq18$, and we state them in that range.
\begin{lemma}
\label{lemmaLM}
Let $n\geq18$.
Let $L,M
\subseteq
F_{(2,1),n}$
be two general subflags isomorphic to $F_{(2,1),n-6}$.
Let $l_i\in L$ and $m_i\in M$ be general points on $L$ and $M$,
respectively, for $i=1,\ldots,2n-17$, and let
$f_j\in F_{(2,1),n}$ be general points for $j=1,\ldots,12$. Set
$$
K_{L,i}
=
\left(
I_{l_i,X_n}^{\langle2\rangle}
+
I_{F_{(2,1),n},X_n}
\right)_{(1,1)},
K_{M,i}
=
\left(
I_{m_i,X_n}^{\langle2\rangle}
+
I_{F_{(2,1),n},X_n}
\right)_{(1,1)},
K_j
=
\left(
I_{f_j,X_n}^{\langle2\rangle}
+
I_{F_{(2,1),n},X_n}
\right)_{(1,1)}
$$
for $i=1,\ldots,2n-17$ and $j=1,\ldots,12$. Then
$$
E
\coloneq
\left(
I_{L\cup M,X_n}
\right)_{(1,1)}
\cap
\bigcap_{i=1}^{2n-17}
\left(
K_{L,i}\cap K_{M,i}
\right)
\cap
\bigcap_{j=1}^{12}K_j
=
\left(
I_{F_{(2,1),n},X_n}
\right)_{(1,1)}.
$$
\end{lemma}

\begin{proof}
We proceed by induction from $n-6$ to $n$. The base cases $n=18,\dots,23$ can be checked by running the scripts \path{computations/6.9_lemma_n*.m2} provided in \cite{Repo}.

Suppose that the statement holds in dimension $n-6$. Let
$N
\subseteq
F_{(2,1),n}
$
be a general subflag isomorphic to $F_{(2,1),n-6}$. Let
$V_N\subseteq V_n$ be the corresponding $(n-6)$-dimensional vector
subspace, so that
\[
N
=
\Fl(1,2;V_N),
\qquad
X_N
=
\P\left(\superwedge^2V_N\right)\times\P(V_N).
\]
We specialize the points as follows:
\begin{itemize}[leftmargin=*]
    \item $l_i\in L\cap N$ for $i=1,\ldots,2n-29$, and these points
    are general in $L\cap N$;

    \item $m_i\in M\cap N$ for $i=1,\ldots,2n-29$, and these points
    are general in $M\cap N$;

    \item $l_i\in L$ for $i=2n-28,\ldots,2n-17$, and these points are
    general in $L$;

    \item $m_i\in M$ for $i=2n-28,\ldots,2n-17$, and these points are
    general in $M$;

    \item $f_j\in N$ for $j=1,\ldots,12$, and these points are general
    in $N$.
\end{itemize}
By semicontinuity, it is enough to prove the statement for this
special configuration of points. Set
$$
E_1
\coloneq
\bigcap_{i=1}^{2n-29}
\left(
K_{L,i}\cap K_{M,i}
\right),
\qquad
E_2
\coloneq
\bigcap_{i=2n-28}^{2n-17}
\left(
K_{L,i}\cap K_{M,i}
\right)
\cap
\bigcap_{j=1}^{12}K_j,
$$
so that
$
E
=
\left(
I_{L\cup M,X_n}
\right)_{(1,1)}
\cap E_1\cap E_2$. We have the following exact sequence of vector spaces:
$$
0
\longrightarrow
E\cap
\left(
I_{N,X_n}
\right)_{(1,1)}
\longrightarrow
E
\longrightarrow
\frac{
E+(I_{N,X_n})_{(1,1)}
}{
(I_{N,X_n})_{(1,1)}
}
\longrightarrow
0.
$$
Moreover,
\[
\begin{aligned}
E\cap
\left(
I_{N,X_n}
\right)_{(1,1)}
&=
\left(
I_{L\cup M\cup N,X_n}
\right)_{(1,1)}
\cap E_1\cap E_2
\\
&\subseteq
\left(
I_{L\cup M\cup N,X_n}
\right)_{(1,1)}
\cap E_2.
\end{aligned}
\]
Since
$
2n-17-(2n-28)+1=12$,
the point conditions appearing in $E_2$ are in the configuration of
\Cref{lemmaLMN}. Hence,
$$
\left(
I_{L\cup M\cup N,X_n}
\right)_{(1,1)}
\cap E_2
=
\left(
I_{F_{(2,1),n},X_n}
\right)_{(1,1)}.
$$
On the other hand,
\[
\left(
I_{F_{(2,1),n},X_n}
\right)_{(1,1)}
\subseteq
E\cap
\left(
I_{N,X_n}
\right)_{(1,1)}.
\]
Therefore,
\[
E\cap
\left(
I_{N,X_n}
\right)_{(1,1)}
=
\left(
I_{F_{(2,1),n},X_n}
\right)_{(1,1)}.
\]
It remains to prove that the right-hand term of the exact sequence is
zero. Set
\begin{gather*}
K_{L,i}'
=
\left(
I_{l_i,X_n}^{\langle2\rangle}
+
I_{N,X_n}
\right)_{(1,1)},
\quad
K_{M,i}'
=
\left(
I_{m_i,X_n}^{\langle2\rangle}
+
I_{N,X_n}
\right)_{(1,1)},
\quad
K_j'
=
\left(
I_{f_j,X_n}^{\langle2\rangle}
+
I_{N,X_n}
\right)_{(1,1)}
\\
K_{L,i}''
=
\left(
I_{l_i,X_N}^{\langle2\rangle}
+
I_{N,X_N}
\right)_{(1,1)},
\quad
K_{M,i}''
=
\left(
I_{m_i,X_N}^{\langle2\rangle}
+
I_{N,X_N}
\right)_{(1,1)},
\quad
K_j''
=
\left(
I_{f_j,X_N}^{\langle2\rangle}
+
I_{N,X_N}
\right)_{(1,1)}
\end{gather*}
for $i=1,\ldots,2n-29$ and $j=1,\ldots,12$. Since
$$
\left(
I_{N,X_n}
\right)_{(1,1)}
\subseteq
E+
\left(
I_{N,X_n}
\right)_{(1,1)}
\subseteq
\left(
I_{L\cup M,X_n}
+
I_{N,X_n}
\right)_{(1,1)}
\cap
\bigcap_{i=1}^{2n-29}
\left(
K_{L,i}'\cap K_{M,i}'
\right)
\cap
\bigcap_{j=1}^{12}K_j'
\eqcolon E',
$$
it is enough to prove that
$
E'
\subseteq
\left(
I_{N,X_n}
\right)_{(1,1)}$.
Since
$
I_{X_N,X_n}
\subseteq
I_{N,X_n}$,
we may consider the quotient of $E'$ by
$(I_{X_N,X_n})_{(1,1)}$. We have
$$
\frac{
E'
}{
(I_{X_N,X_n})_{(1,1)}
}
\subseteq
\left(
I_{(L\cup M)\cap X_N,X_N}
+
I_{N,X_N}
\right)_{(1,1)}
\cap
\bigcap_{i=1}^{2n-29}
\left(
K_{L,i}''\cap K_{M,i}''
\right)
\cap
\bigcap_{j=1}^{12}K_j''
\eqcolon E''.
$$
Since $L\cup M
\subseteq
F_{(2,1),n}$,
one has
\[
(L\cup M)\cap X_N
=
(L\cup M)\cap
\left(
X_N\cap F_{(2,1),n}
\right)
=
(L\cup M)\cap N.
\]
By \Cref{flagintersec}, this is the union of two general subflags
isomorphic to $F_{(2,1),n-12}$ inside
$
N
\simeq
F_{(2,1),n-6}$.
Moreover, since $(L\cup M)\cap N\subseteq N$, we have
$$
\left(
I_{(L\cup M)\cap X_N,X_N}
+
I_{N,X_N}
\right)_{(1,1)}
=
\left(
I_{(L\cup M)\cap N,X_N}
\right)_{(1,1)}.
$$
Thus, in $X_N$, the space $E''$ is in the same configuration as the
one appearing in the statement of the lemma in dimension $n-6$.
The induction hypothesis therefore gives
$
E''
=
\left(
I_{N,X_N}
\right)_{(1,1)}$.
It follows that
$
E'
\subseteq
\left(
I_{N,X_n}
\right)_{(1,1)}$,
and this concludes the proof.
\end{proof}

\begin{lemma}\label{lemL}
Let $n\geq18$, let
$L,M\subset F_{(2,1),n}$ be two general subflags isomorphic to
$F_{(2,1),n-6}$, and let $P\in L\cap M$ be general.
Consider $l_i\in L$ in general position on $L$ for
$i=1,\dots,k_{n-6}$ and
$f_j\in F_{(2,1),n}$ in general position on
$F_{(2,1),n}$ for $j=1,\dots,2n-5$.
Let
$\Delta_n\subset\superwedge^2V_n^*\otimes V_n^*$
be a fixed member of the nonempty family constructed in
\Cref{rem:construction} for the data $P,L,M$. In particular, assume
that
\begin{itemize}[leftmargin=*]
\item $\codim_{\superwedge^2V_n^*\otimes V_n^*}(\Delta_n)=\delta_n$;
\item $\Delta_n\supset(I_{P,X_n}^{\langle2\rangle}+I_{F_{(2,1),n},X_n})_{(1,1)}$
\item $\codim_{\superwedge^2V_n^*\otimes V_n^*}(\Delta_n+(I_{L,X_n})_{(1,1)})=\delta_{n-6}.$
\end{itemize}
Set
$$K_{L,i}=(I^{\langle2\rangle}_{l_i,X_n}+I_{F_{(2,1),n},X_n})_{(1,1)}, \quad K_j=(I^{\langle2\rangle}_{f_j,X_n}+I_{F_{(2,1),n},X_n})_{(1,1)}$$
for $i=1,\dots,k_{n-6}$ and $j=1,\dots,2n-5$. 
Then
$$E\coloneq(I_{L,X_n})_{(1,1)}\cap\bigcap_{i=1}^{k_{n-6}}K_{L,i}\cap\bigcap_{j=1}^{2n-5}K_j\cap \Delta_n=(I_{F_{(2,1),n},X_n})_{(1,1)}.$$
\end{lemma}
\begin{proof}
We proceed by induction from $n-6$ to $n$.
The base cases $n=18,\dots,23$ can be checked by running
the files matching \path{computations/6.10_lemma_n*.m2} in
\cite{Repo}.

Assume now that the statement holds in dimension $n-6$. Let
$V_M\subset V_n$ be the $(n-6)$-dimensional vector subspace
such that $M=\Fl(1,2;V_M)$, and set
$X_M=\P(\superwedge^2V_M)\times\P(V_M)$.
Keeping $P,L,M$, and $\Delta_n$ fixed, specialize the points as
follows:
\begin{itemize}[leftmargin=*]
\item $l_i\in L\cap M$ for $i=1,\dots,k_{n-12}$ and in general position in $L\cap M$;
\item $f_j\in M$ for $j=1,\dots,2n-17$ and in general position in $M$;
\item $l_i\in L$ for $i=k_{n-12}+1,\dots,k_{n-6}$ and in general position in $L$;
\item $f_j\in F_{(2,1),n}$ for $j=2n-16,\dots,2n-5$ and in general position in $F_{(2,1),n}$.
\end{itemize}
By semicontinuity with respect to the support points, it is
enough to prove the statement for this special configuration. Set
$$E_1\coloneq \bigcap_{i=1}^{k_{n-12}}K_{L,i}\cap\Delta_n,\quad E_2\coloneq \bigcap_{i=k_{n-12}+1}^{k_{n-6}}K_{L,i}\cap\bigcap_{j=1}^{2n-5}K_j$$
so that $E=(I_{L,X_n})_{(1,1)}\cap E_1\cap E_2$. Consider the following exact sequence of vector spaces
$$0\longrightarrow E\cap (I_{M,X_n})_{(1,1)}\longrightarrow E\longrightarrow \frac{E+(I_{M,X_n})_{(1,1)}}{(I_{M,X_n})_{(1,1)}}\longrightarrow 0.$$
We have 
\[
(I_{F_{(2,1),n},X_n})_{(1,1)}\subseteq E\cap (I_{M,X_n})_{(1,1)}=(I_{L\cup M,X_n})_{(1,1)}\cap E_1\cap E_2\subseteq (I_{L\cup M,X_n})_{(1,1)}\cap E_2.
\]
By \Cref{lemmanumerico} we have $k_{n-6}-(k_{n-12}+1)+1=2(n-6)-5=2n-17$, so that among the points whose ideals appear in $E_2$ there are $2n-17$ points in general position on $L$, $2n-17$ points in general position on $M$ and 12 points in general position on $F_{(2,1),n}$. Thus, by \Cref{lemmaLM}, we get $E\cap(I_{M,X_n})_{(1,1)}=(I_{F_{(2,1),n},X_n})_{(1,1)}$. 

Now we prove that the right-hand term of the exact sequence is 0. Set
\begin{gather*}K_{L,i}'=(I^{\langle2\rangle}_{l_i,X_n}+I_{M,X_n})_{(1,1)},\quad K_{j}'=(I^{\langle2\rangle}_{f_j,X_n}+I_{M,X_n})_{(1,1)},\quad \Delta_n'=\Delta_n+(I_{M,X_n})_{(1,1)}\\
K_{L,i}''=(I^{\langle2\rangle}_{l_i, X_{M}}+I_{M,X_M})_{(1,1)}, \quad K_{j}''=(I^{\langle2\rangle}_{f_j, X_{M}}+I_{M,X_M})_{(1,1)}\quad \Delta_n''=\frac{\Delta_n+(I_{M,X_n})_{(1,1)}}{(I_{X_M,X_n})_{(1,1)}}
\end{gather*}
for $i=1,\dots,k_{n-12}$ and $j=1,\dots,2n-17$. Since
$$(I_{M,X_n})_{(1,1)}\subseteq E+(I_{M,X_n})_{(1,1)}\subseteq(I_{L,X_n}+I_{M,X_n})_{(1,1)}\cap\bigcap_{i=1}^{k_{n-12}}K_{L,i}'\cap\bigcap_{j=1}^{2n-17}K_j'\cap \Delta_n'\eqcolon E',$$
to conclude it is enough to show that $E'\subseteq(I_{M,X_n})_{(1,1)}$. Since $I_{X_M,X_n}\subset I_{M,X_n}$, we can consider the quotient $E'/(I_{X_M,X_n})_{(1,1)}$ and we have that 
$$E'/(I_{X_M,X_n})_{(1,1)}\subseteq (I_{L\cap X_M, X_M}+I_{M,X_M})_{(1,1)}\cap\bigcap_{i=1}^{k_{n-12}}K_{L,i}''\cap\bigcap_{j=1}^{2n-17}K_j''\cap\Delta_n''\eqcolon E''.$$
Since $L\subset F_{(2,1),n}$, we have that $L\cap X_M=L\cap(X_M\cap F_{(2,1),n})=L\cap M$ and, by \Cref{flagintersec}, $L\cap M$ is a generic $F_{(2,1),n-12}$ in $M=F_{(2,1),n-6}$. Moreover, since $(L\cap M)\subset M$ we have that 
$$
(I_{L\cap M,X_M}+I_{M,X_M})_{(1,1)}=(I_{L\cap M,X_M})_{(1,1)}.
$$
The points whose ideals appear in $E''$ are $k_{n-12}$ points in general position on $L\cap M\subset M$ and $2n-17$ points in general position on $M$.
By the choice of $\Delta_n$ as a member of the family
constructed in \Cref{rem:construction}, the quotient
\[
\Delta_n''
=
\frac{
\Delta_n+(I_{M,X_n})_{(1,1)}
}{
(I_{X_M,X_n})_{(1,1)}
}
\]
is the fixed special subspace $\Delta_{n-6}^{M}$ introduced in \Cref{rem:construction} and used in the
induction hypothesis. In particular,
\begin{itemize}[leftmargin=*]
    \item {
    $
    \codim_{\superwedge^2V_M^*\otimes V_M^*}
    (\Delta_n'')
    =
    \delta_{n-6}$;
    }

    \item {
    $
    \Delta_n''
    \supseteq
    \left(
    I_{P,X_M}^{\langle2\rangle}
    +
    I_{M,X_M}
    \right)_{(1,1)}$;
    }

    \item {
    $
    \codim_{\superwedge^2V_M^*\otimes V_M^*}
    \left(
    \Delta_n''
    +
    (I_{L\cap M,X_M})_{(1,1)}
    \right)
    =
    \delta_{n-12}$.
    }
\end{itemize}
The first and the third conditions follow directly from the compatibility built into the construction in \Cref{rem:construction}; the second follows from the restriction of $I_{P,X_n}^{\langle2\rangle}$ to $X_M$. Therefore, the induction
hypothesis for \Cref{lemL} applies in dimension $n-6$.
{Hence, $E''
=
\left(
I_{M,X_M}
\right)_{(1,1)}$.
In particular,
$
E'
\subseteq
\left(
I_{M,X_n}
\right)_{(1,1)}$,
and this concludes the proof}.
\end{proof}

\begin{theorem}
\label{goodpost}
Let $n\geq5$ and let $\Delta_n\subset\bigwedge^2V_n^*\otimes V_n^*$ be a fixed member of the nonempty family constructed in \Cref{rem:construction} for the data of two general subflags $L,M\subset F_{(2,1),n}$ isomorphic to $F_{(2,1),n-6}$ and a general point $P\in L\cap M$. In particular, suppose that
\[
\codim_{\superwedge^2V_n^*\otimes V_n^*}(\Delta_n)
=
\delta_n
\qquad
\text{and}
\qquad
\Delta_n
\supseteq
\left(
I_{P,X_n}^{\langle2\rangle}
+
I_{F_{(2,1),n},X_n}
\right)_{(1,1)}.
\]
For general points
$P_1,\ldots,P_{k_n}
\in
F_{(2,1),n}$,
set
\[
K_i
\coloneq
\left(
I_{P_i,X_n}^{\langle2\rangle}
+
I_{F_{(2,1),n},X_n}
\right)_{(1,1)}.
\]
Then
\[
\bigcap_{i=1}^{k_n}K_i
\cap
\Delta_n
=
\left(
I_{F_{(2,1),n},X_n}
\right)_{(1,1)}.
\]
\end{theorem}

\begin{proof}
We proceed by induction from $n-6$ to $n$.
The base cases $n=5,\dots,17$ can be checked by running
the files matching \path{computations/6.11_theorem_n*.m2} in
\cite{Repo}.
For $5\leq n\leq17$, these computations provide both the
required choice of $P,\Delta_n$ and the asserted equality. Assume now
that $n\geq18$ and that the statement holds in dimension $n-6$.

Let $V_L\subset V_n$ be the 
$(n-6)$-dimensional subspace such that $L=\Fl(1,2;V_L)$, and set $X_L=\P(\superwedge^2V_L)\times\P(V_L)$. Since $\Delta_n$ is constructed as in \Cref{rem:construction}, then
\[
\Delta_n''
\coloneq
\frac{
\Delta_n+(I_{L,X_n})_{(1,1)}
}{
(I_{X_L,X_n})_{(1,1)}
}
\]
is the fixed special subspace $\Delta_{n-6}^{L}$ introduced in \Cref{rem:construction} to which the
induction hypothesis applies. We keep the notation $K_i$ and set
\[
E
\coloneq
\bigcap_{i=1}^{k_n}K_i
\cap
\Delta_n.
\]
Since the dimension of this intersection is upper semicontinuous in the support points and $\left(I_{F_{(2,1),n},X_n}\right)_{(1,1)}$ is always contained in it, it is enough to specialize the points $P_1,\ldots,P_{k_n}$ so that
\begin{itemize}[leftmargin=*]
\item $P_i\in L$ for $i=1,\dots,k_{n-6}$ and in general position in
$L$;
\item $P_i\in F_{(2,1),n}$ for
$i=k_{n-6}+1,\dots,k_n$ and in general position in
$F_{(2,1),n}$.
\end{itemize}
Thus, the hypotheses of \Cref{lemL} are satisfied and
\[
E\cap
\left(
I_{L,X_n}
\right)_{(1,1)}
=
\left(
I_{F_{(2,1),n},X_n}
\right)_{(1,1)}.
\]
To conclude, it remains to show that $E \subseteq \left( I_{L,X_n}
\right)_{(1,1)}$. Set
\begin{gather*}
K_{i}'=(I^{\langle2\rangle}_{P_i,X_n}+I_{L,X_n})_{(1,1)},\quad
\Delta_n'=\Delta_n+(I_{L,X_n})_{(1,1)}
\\
K_{i}''=(I^{\langle2\rangle}_{P_i, X_{L}}+I_{L,X_L})_{(1,1)}
\quad
\Delta_n''=\frac{\Delta_n+(I_{L,X_n})_{(1,1)}}{(I_{X_L,X_n})_{(1,1)}}
\end{gather*}
for $i=1,\dots,k_{n-6}$. Since
$$(I_{L,X_n})_{(1,1)}\subseteq E+(I_{L,X_n})_{(1,1)}\subseteq\bigcap_{i=1}^{k_{n-6}}K_i'\cap \Delta_n'\eqcolon E',$$
to conclude it is enough to show that $E'\subseteq(I_{L,X_n})_{(1,1)}$. Since $I_{X_L,X_n}\subset I_{L,X_n}$, we can consider the quotient $E'/(I_{X_L,X_n})_{(1,1)}$ and we have that 
$$E'/(I_{X_L,X_n})_{(1,1)}\subseteq \bigcap_{i=1}^{k_{n-6}}K_{i}''\cap\Delta_n''\eqcolon E''.$$
The points whose ideals appear in $E''$ are $k_{n-6}$ general
points of $L \simeq F_{(2,1),n-6}$. Moreover, by the choice of $\Delta_n$, one has
\[
\codim_{\superwedge^2V_L^*\otimes V_L^*}
(\Delta_n'')
=
\delta_{n-6}
\qquad \text{and}
\qquad
\Delta_n''
\supseteq
\left(
I_{P,X_L}^{\langle2\rangle}
+
I_{L,X_L}
\right)_{(1,1)}.
\]
Therefore, the induction hypothesis for \Cref{goodpost} applies
in dimension $n-6$. Hence, $E''
=
\left(
I_{L,X_L}
\right)_{(1,1)}$.
In particular,
$
E'
\subseteq
\left(
I_{L,X_n}
\right)_{(1,1)}$,
and this concludes the proof.
\end{proof}

We are finally ready to give the complete classification for the defectiveness of $F_{(2,1),n}$. For $n=2$, one has $\S_{(2,1)}V_2
\simeq
\det(V_2)\otimes V_2$,
so the corresponding closed orbit is projectively a line
$\P(V_2)\simeq\P^1$ and its secant dimensions are immediate. We
therefore state the classification for $n\geq3$.

\subsection{Classification}
\label{subsect: classification two step flags}

\begin{theorem}[Classification of the secant varieties of
$F_{(2,1),n}$]
\label{thm: classification two step flags}
Let $n\geq3$. Then
\[
\dim\sigma_s(F_{(2,1),n})
=
\begin{cases}
\min\left\{
\dim\S_{(2,1)}V_n,\,
s(2n-2)
\right\}-2,
&
\text{if }(n,s)=(3,2)\text{ or }(4,3),
\\[2mm]
\min\left\{
\dim\S_{(2,1)}V_n,\,
s(2n-2)
\right\}-1,
&
\text{otherwise}.
\end{cases}
\]
In particular, the only defective cases are
\[
\sigma_2(F_{(2,1),3})
\qquad\text{and}\qquad
\sigma_3(F_{(2,1),4}).
\]
\end{theorem}

\begin{proof}
The case $n=3$ follows from
\cite[Theorem~1.4]{BD10}, and the case $n=4$ follows from
\Cref{thm: def of F124}. For $n\geq5$, the statement follows from
\Cref{rem:reduction} and \Cref{goodpost}.
\end{proof}

\bibliographystyle{alpha}
\bibliography{references.bib}

@article{Baur2004,
author = {Baur, K. and Draisma, J.},
title = {{Higher secant varieties of the minimal adjoint orbit}},
journal = {Journal of Algebra},
volume = {280},
number = {2},
pages = {743-761},
year = {2004},
issn = {0021-8693},
}

@misc{Repo,
  author       = {Bernardi, A. and Canino, S. and Isoldi, V. A.},
  title        = {{Supplementary code for ``Secant varieties of flag varieties via Schur apolarity''}},
  year         = {2026},
  publisher    = {Zenodo},
  howpublished = {\url{https://doi.org/10.5281/zenodo.22262086}},
}

@phdthesis{StaffolaniThesis,
    author = {Staffolani, R.},
    title  = {{Schur apolarity and how to use it}},
    school = {University of Trento},
    year   = {2022},
}

@mastersthesis{IsoldiThesis,
    author = {Isoldi, V. A.},
    title  = {Schur Apolarity and the Dual Terracini Lemma for Flag Varieties: A Geometric and Representation Theoretic Approach},
    school = {University of Trento},
    year   = {2025},
}

@article{Staffolani2023,
    author       = {Staffolani, R.},
    title        = {{Schur apolarity}},
    journal      = {Journal of Symbolic Computation},
    volume       = {114},
    pages        = {37--73},
    year         = {2023},
}

@article{ArrondoBernardiMarquesMourrain2021,
    author       = {Arrondo, E. and Bernardi, A. and Macias Marques, P. and Mourrain, B.},
    title        = {Skew-symmetric tensor decomposition},
    journal      = {{Communications in Contemporary Mathematics}},
    volume       = {23},
    number       = {02},
    pages        = {1950061},
    year         = {2021},
}

@article{Ottaviani2013,
  author  = {Ottaviani, G.},
  title   = {Five Lectures on Projective Invariants},
  journal = {Rendiconti del Seminario Matematico della Università Politecnica di Torino},
  volume   = {71},
  number   = {1},
  year     = {2013},
  pages    = {119--194},
}

@article{BD10,
	author = {Baur, K. and Draisma, J.},
	doi = {10.1515/ADVGEOM.2010.001},
	fjournal = {Advances in Geometry},
	issn = {1615-715X,1615-7168},
	journal = {Adv. Geom.},
	mrclass = {14N05 (14T05)},
	mrnumber = {2603719},
	mrreviewer = {Nicolas\ Perrin},
	number = {1},
	pages = {1--29},
	title = {Secant dimensions of low-dimensional homogeneous varieties},
	url = {https://doi.org/10.1515/ADVGEOM.2010.001},
	volume = {10},
	year = {2010}}

@book{Fulton1997,
  author    = {Fulton, W.},
  title     = {Young Tableaux: With Applications to Representation Theory and Geometry},
  series    = {London Mathematical Society Student Texts},
  volume    = {35},
  publisher = {Cambridge University Press},
  address   = {Cambridge},
  year      = {1997}
}

@book{FultonHarris1991,
  author    = {Fulton, W. and Harris, J.},
  title     = {Representation Theory: A First Course},
  series    = {Graduate Texts in Mathematics},
  volume    = {129},
  publisher = {Springer-Verlag},
  address   = {New York},
  year      = {1991},
  doi       = {10.1007/978-1-4612-0979-9}
}

@article{bernardi_hitchhiker_2018,
  author  = {Bernardi, A. and Carlini, E. and
             Catalisano, M.V. and Gimigliano, A. and
             Oneto, A.},
  title   = {The {Hitchhiker} {Guide} to: {Secant} {Varieties} and
             {Tensor} {Decomposition}},
  journal = {Mathematics},
  volume  = {6},
  number  = {12},
  year    = {2018},
  doi     = {10.3390/math6120314}
}

@article{AlexanderHirschowitz1995,
  author  = {Alexander, J. and Hirschowitz, A.},
  title   = {{Polynomial interpolation in several variables}},
  journal = {Journal of Algebraic Geometry},
  volume  = {4},
  number  = {2},
  pages   = {201--222},
  year    = {1995}
}

@article{BrambillaOttaviani2008,
  author  = {Brambilla, M.C. and Ottaviani, G.},
  title   = {{On the Alexander--Hirschowitz theorem}},
  journal = {Journal of Pure and Applied Algebra},
  volume  = {212},
  number  = {5},
  pages   = {1229--1251},
  year    = {2008},
  doi     = {10.1016/j.jpaa.2007.09.014}
}

@book{IarrobinoKanev1999,
  author    = {Iarrobino, A. and Kanev, V.},
  title     = {{Power Sums, Gorenstein Algebras, and Determinantal
                Loci}},
  series    = {Lecture Notes in Mathematics},
  volume    = {1721},
  publisher = {Springer-Verlag},
  address   = {Berlin},
  year      = {1999},
  doi       = {10.1007/BFb0093426}
}

@article{BFCM22,
  author  = {Barbosa Freire, A. and Casarotti, A. and Massarenti, A.},
  title   = {{On secant dimensions and identifiability of flag varieties}},
  journal = {Journal of Pure and Applied Algebra},
  volume  = {226},
  number  = {6},
  pages   = {106969},
  year    = {2022},
  doi     = {10.1016/j.jpaa.2021.106969},
}

@misc{BC23,
  author        = {Taveira Blomenhofer, A. and Casarotti, A.},
  title         = {{Nondefectivity of invariant secant varieties}},
  year          = {2023},
  eprint        = {2312.12335},
  archivePrefix = {arXiv},
  primaryClass  = {math.AG},
}

@article{Lasker1904,
  author  = {Lasker, E.},
  title   = {{Zur Theorie der kanonischen Formen}},
  journal = {Mathematische Annalen},
  volume  = {58},
  pages   = {434--440},
  year    = {1904},
}

\end{document}